\documentclass[11pt]{amsart}
\usepackage{geometry}
\usepackage{amscd,amssymb,verbatim,xcolor}
\usepackage{longtable, multirow}
\usepackage{soul,cancel,tikz}
\usepackage{pgf,tikz,pgfplots}
\pgfplotsset{compat=1.15}
\usepackage{mathrsfs}
\usetikzlibrary{arrows.meta,bending}
\date{\today}

\newtheorem{theorem}{Theorem}[section]

\newtheorem{corollary}[theorem]{Corollary}
\newtheorem{conjecture}[theorem]{Conjecture}
\newtheorem{definition}[theorem]{Definition}
\newtheorem{example}[theorem]{Example}
\newtheorem{lemma}[theorem]{Lemma}
\newtheorem{proposition}[theorem]{Proposition}
\newtheorem{remark}[theorem]{Remark}

\begin{document}
\title[The unitary dual of $U(n,2)$]{The unitary dual of $U(n,2)$}
\author{Kayue Daniel Wong} \author{Hongfeng Zhang}

\address{School of Science and Engineering, The Chinese University of Hong Kong, Shenzhen, Guangdong 518172, P. R. China }
\email{kayue.wong@gmail.com}

\address{Department of Mathematics, the University of Hong Kong, P. R. China }
\email{zhanghongf@pku.edu.cn}

\begin{abstract}
In this paper, we give a full classification of the unitary dual of $G = U(n,2)$ for $n \geq 3$. As a consequence, we determine which of these representations are weakly fair $A_{\mathfrak{q}}(\lambda)$-modules or special unipotent representations, resulting in a verification of a conjecture of Trapa and Vogan for $G$.
\end{abstract}

\maketitle

\section{Introduction}\label{sec:intro}
A main unsolved problem in the representation theory of real reductive Lie groups $G$ is the classification of all irreducible, unitarizable $(\mathfrak{g},K)$-modules, i.e. the unitary dual $\widehat{G}$.
With the help of the \texttt{atlas} software \cite{AvLTV20}, one can check the unitarizability of any specified irreducible module. Although one cannot obtain $\widehat{G}$ directly from the software, it provides a powerful tool in verifying some conjectures, as well as giving some new insights on the structure of $\widehat{G}$ (see \cite{W22}, \cite{WZ23} for instance).

\medskip
In this manuscript, we give a full classification of the unitary dual for $G = U(n,2)$ for $n > 2$ with real infinitesimal characters (the general case can be obtained from the real infinitesimal case by real parabolic induction). Before doing so, we give a short review of some work related to its classification problem -- firstly, Knapp and Speh has a full classification of $\widehat{U(2,2)}$ in \cite{KS82}. As for $n > 2$, a lot of effort was made in the study of its unitary dual by using the notion of \emph{basic cases} \cite{KS83}. Some more results were given by Baldoni-Silva and Knapp (e.g. \cite{BK87}, \cite{BK89a}, \cite{BK89b}) for reductive groups of real rank two in general. 

\medskip
There are two main objectives for the manuscript. 
Firstly, we aim to incorporate some discoveries from the aforementioned works into the language of \emph{combinatorial $\theta$-stable datum} introduced in \cite{W22}. This framework provides a parametrization of all irreducible $(\mathfrak{g},K)$-modules that is useful for solving problems related to the unitary dual of $G$. Special emphasis is put on the notion of \emph{fundamental modules} (Section \ref{sec-fund}), which are conjectured to be the `core' of the unitary dual of $U(p,q)$ and more generally all real reductive groups. In particular, we will give evidences on the importance of fundamental modules in the case of $U(n,2)$ (see Remark \ref{rmk-un1} and Remark \ref{rnk-un2} below).

\smallskip
Secondly, our approach enables us to answer some questions closely related to the unitary dual problem (and beyond). To start with, a large part of the unitary dual is given by cohomologically inducing a unitary representation within some specified \emph{`ranges'} depending on its infinitesimal character \cite{KV95}. If the inducing module is a unitary character, such modules are called {\it $A_{\mathfrak{q}}(\lambda)$-modules}, and they play a vital role in the study of Arthur's parametrization of automorphic representations (cf. \cite{VZ84}, \cite{AJ87}). A fundamental problem in representation theory is to determine when the $A_{\mathfrak{q}}(\lambda)$-modules in the \emph{weakly fair range} is nonzero. In this manuscript, we will determine which representations in $\widehat{U(n,2)}$ are weakly fair $A_{\mathfrak{q}}(\lambda)$-modules, and specify which $\mathfrak{q}$ and $\lambda$ one may take in all such cases. As applications, we verify a conjecture of Trapa and Vogan for $U(n,2)$ (Theorem \ref{thm-trapa}) that all representations in $\widehat{U(n,2)}$ with integral infinitesimal characters are weakly fair $A_{\mathfrak{q}}(\lambda)$-modules. On the other hand, we will see in Section \ref{subsec-ii} that one class of unitary representations not covered in the basic cases in \cite{KS83} is a fair range $A_{\mathfrak{q}}(\lambda)$-module with $\mathfrak{l}_0 = \mathfrak{u}(n,1) + \mathfrak{u}(0,1)$.
Finally, we also expect that one can obtain from our results a full classification of the \emph{Dirac series} of $U(n,2)$, i.e. unitary representations with non-zero Dirac cohomology. 

\medskip
As suggested above, fundamental modules are expected to play an essential role in the classification of the unitary dual of $U(p,q)$. In an upcoming work, we will generalize the work of \cite{B04} (which only considers spherical representations) to study the fundamental unitary dual of $U(p,q)$ with integral infinitesimal characters. 

\medskip
The manuscript is organized as follows: In Section \ref{sec-prelim}, we introduce some combinatorics in \cite{W22} to parametrize all irreducible Hermitian $(\mathfrak{g},K)$-modules $X$ with real infinitesimal characters for $U(p,q)$. 
In particular, we will recall the notion of fundamental cases and some techniques on determining the (non)-unitarity of $X$.
In Section \ref{sec-un1}, we will determine the unitary dual of $U(n,1)$, and Section \ref{sec-un2} continues with the classification of the unitary dual of $U(n,2)$. In Section \ref{sec-coh}, we find out which of the unitary representations for $U(n,1)$ and $U(n,2)$ are weakly-fair $A_{\mathfrak{q}}(\lambda)$-modules. 
Finally, in Section \ref{sec-unip}, we obtain all special unipotent representations for $U(n,1)$ and $U(n,2)$ (in the sense of \cite{BV85}, \cite{BMSZ21} and \cite{BMSZ22}) in terms of their Langlands parameters.

\section{Preliminaries} \label{sec-prelim}
In Section \ref{sec-tempered} and \ref{sec-irrep}, we review the Langlands classification of irreducible $(\mathfrak{g},K)$-modules for $G= U(p,q)$ given in \cite{W22} using {\bf $\lambda_a$-blocks} and {\bf combinatorial $\theta$-stable data}. 
\subsection{$\lambda_a$-blocks and tempered modules} \label{sec-tempered}
\begin{definition}
Let $G = U(p,q)$ with $p + q \equiv \epsilon\ (\text{mod}\ 2)$ ($\epsilon = 0$ or $1$). 
A {\bf $\lambda_a$-block of size $(r,s)$ with content $\gamma$} (or simply a {\bf $\gamma$-block}),
where $0 \leq r \leq p$, $0 \leq s \leq q$ are non-negative integers satisfying $|r-s| \leq 1$, and $\gamma \in \frac{1}{2}\mathbb{Z}$ is of one of the following forms:
\begin{itemize}
\item {\bf Rectangle} of size $(r,r)$:
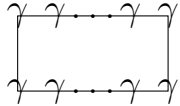
\begin{center}
\begin{tikzpicture}
\draw
    (0,0) node {\Large $\gamma$}
 -- (0.5,0) node {\Large $\gamma$}		
 -- (1,0) node {\Large $\ldots$}
 -- (1.5,0) node {\Large $\gamma$}
 -- (2,0) node {\Large $\gamma$}
 -- (2,-1) node {\Large $\gamma$}
 -- (1.5,-1) node {\Large $\gamma$}
-- (1,-1) node {\Large $\ldots$}
-- (0.5,-1) node {\Large $\gamma$}
-- (0,-1) node {\Large $\gamma$}
 -- cycle;
\end{tikzpicture} 
\end{center}
where $\gamma + \frac{\epsilon}{2} \in \mathbb{Z}$.

\item {\bf Parallelogram} of size $(r,r)$:
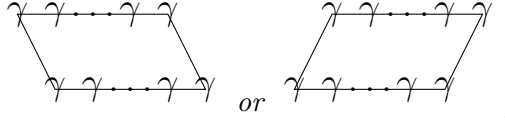
\begin{center}
\begin{tikzpicture}
\draw
    (0,0) node {\Large $\gamma$}
 -- (0.5,0) node {\Large $\gamma$}		
 -- (1,0) node {\Large $\ldots$}
 -- (1.5,0) node {\Large $\gamma$}
 -- (2,0) node {\Large $\gamma$}
 -- (2.5,-1) node {\Large $\gamma$}
 -- (2,-1) node {\Large $\gamma$}
-- (1.5,-1) node {\Large $\ldots$}
-- (1,-1) node {\Large $\gamma$}
-- (0.5,-1) node {\Large $\gamma$}
 -- cycle;
\end{tikzpicture} or 
\begin{tikzpicture}
\draw
    (0,0) node {\Large $\gamma$}
 -- (0.5,0) node {\Large $\gamma$}		
 -- (1,0) node {\Large $\ldots$}
 -- (1.5,0) node {\Large $\gamma$}
 -- (2,0) node {\Large $\gamma$}
 -- (1.5,-1) node {\Large $\gamma$}
 -- (1,-1) node {\Large $\gamma$}
-- (0.5,-1) node {\Large $\ldots$}
-- (0,-1) node {\Large $\gamma$}
-- (-0.5,-1) node {\Large $\gamma$}
 -- cycle;
\end{tikzpicture},
\end{center}
where $\gamma + \frac{\epsilon + 1}{2} \in \mathbb{Z}$.

\item {\bf Trapezoid} of size $(r,r-1)$ or $(r,r+1)$:
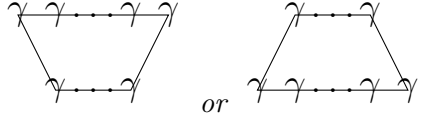
\begin{center}
\begin{tikzpicture}
\draw
    (0,0) node {\Large $\gamma$}
 -- (0.5,0) node {\Large $\gamma$}		
 -- (1,0) node {\Large $\ldots$}
 -- (1.5,0) node {\Large $\gamma$}
 -- (2,0) node {\Large $\gamma$}
 -- (1.5,-1) node {\Large $\gamma$}
-- (1,-1) node {\Large $\ldots$}
-- (0.5,-1) node {\Large $\gamma$}
 -- cycle;
\end{tikzpicture} or 
\begin{tikzpicture}
\draw
 (0.5,0) node {\Large $\gamma$}		
 -- (1,0) node {\Large $\ldots$}
 -- (1.5,0) node {\Large $\gamma$}
 -- (2,-1) node {\Large $\gamma$}
 -- (1.5,-1) node {\Large $\gamma$}
-- (1,-1) node {\Large $\ldots$}
-- (0.5,-1) node {\Large $\gamma$}
-- (0,-1) node {\Large $\gamma$}
 -- cycle;
\end{tikzpicture},
\end{center}
where $\gamma + \frac{\epsilon + 1}{2} \in \mathbb{Z}$.

\end{itemize}
\end{definition}

\begin{definition} \label{def-datum}
A {\bf $\lambda_a$-datum for $G$} is a collection of $\gamma_i$-blocks of sizes $(r_i,s_i)$ such that $\sum r_i = p$, $\sum s_i = q$ and all $\gamma_i$ are distinct.
\end{definition}

The following result can be found in Section 3 of \cite{W22}:
\begin{theorem}  \label{thm-datum}
    Let $G = U(p,q)$. There is a $1:1$-correspondence between the following sets:
\begin{itemize}
    \item[(a)] All $\lambda_a$-data of $G$.
    \item[(b)] All $K$-types $\delta \in \widehat{K}$.
    \item[(c)] All tempered representations with real infinitesimal characters.
\end{itemize}
\end{theorem}

\begin{example}[\cite{W22}, Example 3.6] \label{eg-lambdaa}
We present the correspondence between (a) and (b) in the above theorem with an example: Let $G = U(6,3)$, and consider the $\lambda_a$-datum
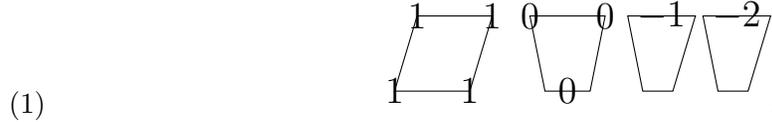
\begin{equation} \label{eq-parallel1}
\begin{tikzpicture}
\draw
    (0,0) node {\Large $1$}
 -- (1,0) node {\Large $1$}
-- (0.7,-1) node {\Large $1$}
-- (-0.3,-1) node {\Large $1$}
 -- cycle;

\draw
    (1.5,0) node {\Large $0$}
 -- (2.5,0) node {\Large $0$}
-- (2.3,-1) 
-- (2,-1) node {\Large $0$}
-- (1.7,-1) 
 -- cycle;

\draw
    (2.8,0) 
 -- (3.25,0) node {\Large $-1$}		
-- (3.7,0)
 -- (3.4,-1) 
-- (3,-1) 
 -- cycle;

\draw
    (3.8,0) 
 -- (4.25,0) node {\Large $-2$}		
-- (4.7,0)
 -- (4.4,-1) 
-- (4,-1) 
 -- cycle;
\end{tikzpicture}.
\end{equation}
Then it defines a $\theta$-stable quasisplit parabolic subalgebra 
\begin{equation} \label{eq-thetaeg}
\mathfrak{q}_0 = \mathfrak{l}_0 + \mathfrak{u}_0, \quad \quad \mathfrak{l}_0 = \mathfrak{u}(2,2) + \mathfrak{u}(2,1) + \mathfrak{u}(1,0) + \mathfrak{u}(1,0),
\end{equation}
and $\mathfrak{u}_0$ is determined by the ordering of the contents of the blocks.

To get the $K$-type corresponding to the above datum, consider 
\begin{align*}
\lambda_a &:= (1,1,0,0,-1-2|1,1,0)\\
 \rho(\mathfrak{u}) &= \frac{1}{2}(5,5,-2,-2,-6,-8|5,5,-2)   \\
 2\rho(\mathfrak{u} \cap \mathfrak{p}) &= (1,1,-2,-2,-3,-3|4,4,0)
\end{align*}
and
$$\lambda_a - \rho(\mathfrak{u}) + 2\rho(\mathfrak{u} \cap \mathfrak{p}) = \left(-\frac{1}{2},-\frac{1}{2},-1,-1,-1,-1\Big|\frac{5}{2},\frac{5}{2},1\right).$$
If there are parallelograms of shape \begin{tikzpicture}
\draw
 (0.25,0)	
 -- (0.75,0) 
 -- (0.5,-0.5) 
-- (0,-0.5) 
 -- cycle;
\end{tikzpicture} in the datum, one needs to further subtract all top entries of the block by $1/2$ and add all 
bottom entries of the block by $1/2$ (and analogously for \begin{tikzpicture} \draw
    (0,0)  
 -- (0.5,0) 
 -- (0.75,-0.5) 
 -- (0.25,-0.5) 
 -- cycle;
\end{tikzpicture}). In our example, one has:
$$\left(\frac{-1}{2},\frac{-1}{2},-1,-1,-1,-1\Big|\frac{5}{2},\frac{5}{2},1\right) + \left(\frac{-1}{2},\frac{-1}{2},0,0,0,0\Big|\frac{1}{2},\frac{1}{2},0\right) = (-1,-1,-1,-1,-1,-1|3,3,1).$$
So the datum corresponds to the $K$-type $V_{(-1,-1,-1,-1,-1,-1|3,3,1)}$ under Theorem \ref{thm-datum}.
\end{example}

To relate (a) and (c) in Theorem \ref{thm-datum}, we recall the two constructions of tempered and, more generally, irreducible admissible $(\mathfrak{g},K)$-modules $X$ given in \cite[Chapter 6]{V82}. For the rest of this section, we present these constructions when $G = U(p,q)$ and $X$ is a tempered module, and describe its relationship with $\lambda_a$-data. In Section \ref{sec-irrep}, the results in this section will be generalized to all irreducible modules.

\medskip
\noindent \underline{\bf Method I - Cohomological Induction:} By \cite[Chapter 6.5]{V82}, all irreducible $(\mathfrak{g},K)$-modules $X$ can be realized as (components of) cohomologically inducing a principal series $(\mathfrak{l},L\cap K)$-module $\pi_L$ from a $\theta$-stable quasisplit parabolic subalgebra $\mathfrak{q}_0 = \mathfrak{l}_0 + \mathfrak{u}_0$ of $\mathfrak{g}_0$. We refer the reader to \cite{KV95} for
more details on cohomological induction.
In particular, if $X$ is a tempered representation with real infinitesimal character, then $\pi_L$ is chosen such that the $A$-character of in its inducing module is trivial, where $A$ is the noncompact part of the maximally split Cartan subgroup of $L$. 

\smallskip
In the case when $G = U(p,q)$, the Levi subalgebra $\mathfrak{l}_0$ of a quasisplit $\theta$-stable parabolic subalgebra $\mathfrak{q}_0$ must be of the form $\mathfrak{l}_0 = \bigoplus_i \mathfrak{u}(r_i,s_i)$ with $|r_i - s_i| \leq 1$, $\sum_i r_i = p$ and $\sum_i s_i = q$. So $\pi_L$ must be of the form $\pi_L := \bigotimes_i \pi_{L_i}$, where each $\pi_{L_i}$ is a principal series representation of $L_i := U(r_i,s_i)$. Here are the choices of $\pi_{L_i}$ in the construction of tempered modules.

\smallskip
If $r_i = s_i$, there are two possibilities of $\pi_{L_i}$:
\begin{itemize}
    \item the irreducible {\it (pseudo)spherical principal series representation} with lowest $K$-type $(a,\dots,a|a,\dots,a)$ ($a \in \mathbb{Z}$) and infinitesimal character $(a,\dots,a)$;
    \item the {\bf reducible} {\it non-(pseudo)spherical principal series representation} with two factors, whose lowest $K$-types are $(a+1,\dots,a+1|a,\dots,a)$ and $(a,\dots,a|a+1,\dots,a+1)$ ($a \in \mathbb{Z}$) with infinitesimal character $(\frac{2a+1}{2},\dots,\frac{2a+1}{2})$. 
\end{itemize}

If $r_i \neq s_i$, there is only one possibility of $\pi_{L_i}$:
\begin{itemize}
    \item the irreducible (pseudo)spherical principal series representation with lowest $K$-type $(a,\dots,a|a,\dots,a)$ ($a \in \mathbb{Z}$) and infinitesimal character $(a,\dots,a)$.
\end{itemize}
In particular, $\pi_L$ has $2^r$ irreducible modules, where $r$ is the number of non-(pseudo)spherical $\pi_{L_i}$'s. Then the cohomologically induced module $\mathcal{R}_{\mathfrak{q}}(\pi_L)$ consists of precisely $2^r$ distinct tempered representations. Moreover, all tempered representations can be uniquely constructed in such a way.  

\smallskip
We are now in the position to establish the correspondence of (a) and (c) in Theorem \ref{thm-datum}: Firstly, a $\lambda_a$-datum determines a quasisplit $\theta$-stable parabolic subalgebra $\mathfrak{q}_0 = \mathfrak{l}_0 + \mathfrak{u}_0$ as in \eqref{eq-thetaeg}. 
As for $\pi_L = \bigotimes_i \pi_{L_i}$, one considers $\lambda_a - \rho(\mathfrak{u})$, which amounts to shifting the content of each $\lambda_a$-block from $\lambda_i$ to $\lambda_i'$. Then $\pi_{L_i}$ is taken to be one of the three principal series representations of $L_i = U(r_i,s_i)$ listed above with infinitesimal character $(\lambda_i',\dots,\lambda_i')$. 

For instance, in Example \ref{eg-lambdaa}, the $\theta$-stable parabolic algebra $\mathfrak{q}_0$ is given by \eqref{eq-thetaeg}, and $\pi_L = \pi_{L_1} \otimes \pi_{L_2} \otimes \pi_{L_3} \otimes \pi_{L_4}$ is a principal series representation of $L = U(2,2) \times U(2,1) \times U(1,0) \times U(1,0)$ with infinitesimal character 
$$\lambda_a' := \lambda_a - \rho(\mathfrak{u}) = \left((\frac{-3}{2},\frac{-3}{2}\Big|\frac{-3}{2},\frac{-3}{2}),\ (1,1|1),\ (2|),\ (2|)\right).$$ 
In particular, the only non-(pseudo)spherical, reducible principal series representation is $\pi_{L_1}$ (corresponding to the $(2,2)$-parallelogram). Therefore, both $\pi_L$ and $\mathcal{R}_{\mathfrak{q}}(\pi_L)$ consist of $2^1 = 2$ tempered modules. 

To determine which of the $2$ modules correspond to the datum \eqref{eq-parallel1}, one can look at their lowest $K$-types. Indeed, their lowest $K$-types are $V_{(-1,-1,-1,-1,-1,-1|3,3,1)}$ and $V_{(0,0,-1,-1,-1,-1|2,2,1)}$ respectively. 
In view of Example \ref{eg-lambdaa}, the one with lowest $K$-type $V_{(-1,-1,-1,-1,-1,-1|3,3,1)}$ corresponds to the datum \eqref{eq-parallel1}, while the other tempered module corresponds to the datum:
\begin{equation} \label{eq-parallel2}
\begin{tikzpicture}
\draw
    (-0.3,0) node {\Large $1$}
 -- (0.7,0) node {\Large $1$}
-- (1,-1) node {\Large $1$}
-- (0,-1) node {\Large $1$}
 -- cycle;

\draw
    (1.5,0) node {\Large $0$}
 -- (2.5,0) node {\Large $0$}
-- (2.3,-1) 
-- (2,-1) node {\Large $0$}
-- (1.7,-1) 
 -- cycle;

\draw
    (2.8,0) 
 -- (3.25,0) node {\Large $-1$}		
-- (3.7,0)
 -- (3.4,-1) 
-- (3,-1) 
 -- cycle;

\draw
    (3.8,0) 
 -- (4.25,0) node {\Large $-2$}		
-- (4.7,0)
 -- (4.4,-1) 
-- (4,-1) 
 -- cycle;
\end{tikzpicture}
\end{equation}
by `flipping' the $(2,2)$-parallelogram in datum \eqref{eq-parallel1}.

\medskip
In general, the correspondences in Theorem \ref{thm-datum} is given as follows: Fix a $\lambda_a$-datum, then one gets a $K$-type $\delta$ from the datum by Example \ref{eg-lambdaa}. As for tempered module, one constructs $\mathcal{R}_{\mathfrak{q}}(\pi_L)$ from the datum as described above. In particular, it contains $2^r$ tempered modules, where $r$ is the number of parallelograms in the datum. Among them, there is exactly one with lowest $K$-type equal to $\delta$. This is precisely the tempered module under the correspondence in Theorem \ref{thm-datum}, and the other $2^r - 1$ tempered modules in $\mathcal{R}_{\mathfrak{q}}(\pi_L)$ correspond to $2^r -1$ $\lambda_a$-data by `flipping' the $r$ parallelograms in the original datum.

\medskip
\noindent \underline{\bf Method II - Real Parabolic Induction:} 
A more well-known version Langlands classification of irreducible admissible $(\mathfrak{g},K)$-modules is given by real parabolically inducing a (limit of) discrete series representation. We now describe how our definition of $\lambda_a$-datum fits into this framework (cf. \cite[Chapter 6.6]{V82}). 

For $a \in \mathbb{Z}$, let $\Delta_{1,a} := (\frac{\det}{|\det|})^{a}$ be a discrete series representation of $GL(1,\mathbb{C})$. Then for any $\lambda_a$-datum, we do the following:
\begin{enumerate}
    \item[(i)] For each $(r,s)$-rectangle or trapezoid block in the datum, let $m := \min\{r,s\}$. Form a discrete series representation $\Delta_{m,2\gamma} := (\Delta_{1,2\gamma})^{\otimes m}$ of $GL(1,\mathbb{C})^m$, and remove $m$ copies of $\gamma$ on the top and bottom of the block.
    \item[(ii)] For each $(r,r)$-parallelogram block, form a discrete series representation $\Delta_{r,2\gamma}$ of $GL(1,\mathbb{C})^r$ as in (i), and remove the whole $(r,r)$-block.
    \item[(iii)] Collect all the remaining blocks. They correspond to a discrete series representation $\Sigma$, whose infinitesimal character is determined by the contents of the blocks.
\end{enumerate}
For instance, on the datum \eqref{eq-parallel2}, we have:
\begin{center}
    \begin{tikzpicture}
\draw
    (-0.15,0) node {$1$}
 -- (0.35,0) node {$1$}
-- (0.5,-0.5) node {$1$}
-- (0,-0.5) node {$1$}
 -- cycle;

\draw
    (0.75,0) node {$0$}
 -- (1.25,0) node {$0$}
-- (1.15,-0.5) 
-- (1,-0.5) node {$0$}
-- (0.85,-0.5) 
 -- cycle;

\draw
    (1.4,0) 
 -- (1.625,0) node {$-1$}		
-- (1.85,0)
 -- (1.7,-0.5) 
-- (1.5,-0.5) 
 -- cycle;

\draw
    (1.9,0) 
 -- (2.125,0) node {$-2$}		
-- (2.35,0)
 -- (2.2,-0.5) 
-- (2,-0.5) 
 -- cycle;
\end{tikzpicture}
$\xrightarrow{(i):\ \Delta_{1,0}}$     \begin{tikzpicture}
\draw
    (0.25,0) node {$1$}
 -- (0.75,0) node {$1$}
-- (0.9,-0.5) node {$1$}
-- (0.4,-0.5) node {$1$}
 -- cycle;

\draw
    (0.9,0) 
 -- (1.125,0) node {$0$}		
-- (1.35,0)
 -- (1.2,-0.5) 
-- (1,-0.5) 
 -- cycle;

\draw
    (1.4,0) 
 -- (1.625,0) node {$-1$}		
-- (1.85,0)
 -- (1.7,-0.5) 
-- (1.5,-0.5) 
 -- cycle;

\draw
    (1.9,0) 
 -- (2.125,0) node {$-2$}		
-- (2.35,0)
 -- (2.2,-0.5) 
-- (2,-0.5) 
 -- cycle;
\end{tikzpicture}
$\xrightarrow{(ii):\ \Delta_{2,2}}$  
\begin{tikzpicture}
\draw
    (0.9,0) 
 -- (1.125,0) node {$0$}		
-- (1.35,0)
 -- (1.2,-0.5) 
-- (1,-0.5) 
 -- cycle;

\draw
    (1.4,0) 
 -- (1.625,0) node {$-1$}		
-- (1.85,0)
 -- (1.7,-0.5) 
-- (1.5,-0.5) 
 -- cycle;

\draw
    (1.9,0) 
 -- (2.125,0) node {$-2$}		
-- (2.35,0)
 -- (2.2,-0.5) 
-- (2,-0.5) 
 -- cycle;
\end{tikzpicture} $\stackrel{(iii)}\rightsquigarrow \Sigma$
\end{center}
Then construct the induced module:
\begin{equation} \label{eq-glm}
\mathrm{Ind}_{MAN}^G(\Delta_{2,2} \boxtimes \Delta_{1,0} \boxtimes \Sigma \boxtimes 1) \quad (MA = GL(1,\mathbb{C})^2 \times GL(1,\mathbb{C})^1 \times U(3,0)),
\end{equation}
where $\Sigma$ is the discrete series representation of $U(3,0)$ with infinitesimal character $(0,-1,-2)$. In such a case, the induced module \eqref{eq-glm} is isomorphic to the cohomologically induced module $\mathcal{R}_{\mathfrak{q}}(\pi_L)$ in Method I. It consists of two tempered modules whose lowest $K$-types correspond to the $\lambda_a$-data \eqref{eq-parallel1} and \eqref{eq-parallel2} under Theorem \ref{thm-datum}.

\medskip
We remark that it is possible to obtain a single tempered module. Namely, for Method I, instead of cohomologically inducing the whole non-(pseudo)spherical principal series $\pi_{L_i}$, one can just induce an irreducible factor of it. As for Method II, one changes Step (ii) into:
\begin{center}
    (ii') For each $(r,r)$-parallelogram block, form the 
    discrete series representation $\Delta_{r-1,2\gamma}$ of $GL(1,\mathbb{C})^{r-1}$, and remove $r-1$ copies of $\gamma$ on the top and bottom row of the block.
\end{center}
For instance, on the datum \eqref{eq-parallel2}, we instead do the following:
\begin{center}
    \begin{tikzpicture}
\draw
    (-0.15,0) node {$1$}
 -- (0.35,0) node {$1$}
-- (0.5,-0.5) node {$1$}
-- (0,-0.5) node {$1$}
 -- cycle;

\draw
    (0.75,0) node {$0$}
 -- (1.25,0) node {$0$}
-- (1.15,-0.5) 
-- (1,-0.5) node {$0$}
-- (0.85,-0.5) 
 -- cycle;

\draw
    (1.4,0) 
 -- (1.625,0) node {$-1$}		
-- (1.85,0)
 -- (1.7,-0.5) 
-- (1.5,-0.5) 
 -- cycle;

\draw
    (1.9,0) 
 -- (2.125,0) node {$-2$}		
-- (2.35,0)
 -- (2.2,-0.5) 
-- (2,-0.5) 
 -- cycle;
\end{tikzpicture}
$\xrightarrow{(i):\ \Delta_{1,0}}$     \begin{tikzpicture}
\draw
    (0.25,0) node {$1$}
 -- (0.75,0) node {$1$}
-- (0.9,-0.5) node {$1$}
-- (0.4,-0.5) node {$1$}
 -- cycle;

\draw
    (0.9,0) 
 -- (1.125,0) node {$0$}		
-- (1.35,0)
 -- (1.2,-0.5) 
-- (1,-0.5) 
 -- cycle;

\draw
    (1.4,0) 
 -- (1.625,0) node {$-1$}		
-- (1.85,0)
 -- (1.7,-0.5) 
-- (1.5,-0.5) 
 -- cycle;

\draw
    (1.9,0) 
 -- (2.125,0) node {$-2$}		
-- (2.35,0)
 -- (2.2,-0.5) 
-- (2,-0.5) 
 -- cycle;
\end{tikzpicture}
$\xrightarrow{(ii'):\ \Delta_{1,2}}$  
\begin{tikzpicture}
\draw
    (0.5,0)
-- (0.625,0)   node {$1$}
 -- (0.75,0)
-- (0.9,-0.5) 
-- (0.775,-0.5) node {$1$} 
-- (0.65,-0.5) 
 -- cycle;

\draw
    (0.9,0) 
 -- (1.125,0) node {$0$}		
-- (1.35,0)
 -- (1.2,-0.5) 
-- (1,-0.5) 
 -- cycle;

\draw
    (1.4,0) 
 -- (1.625,0) node {$-1$}		
-- (1.85,0)
 -- (1.7,-0.5) 
-- (1.5,-0.5) 
 -- cycle;

\draw
    (1.9,0) 
 -- (2.125,0) node {$-2$}		
-- (2.35,0)
 -- (2.2,-0.5) 
-- (2,-0.5) 
 -- cycle;
\end{tikzpicture} $\stackrel{(iii)}\rightsquigarrow \Sigma'$
\end{center}
and obtains an irreducible tempered module corresponding to the datum \eqref{eg-lambdaa2}:
\begin{equation} 
    \mathrm{Ind}_{M'A'N'}^G(\Delta_{1,2} \boxtimes \Delta_{1,0} \boxtimes \Sigma' \boxtimes 1) \quad \quad M'A' = GL(1,\mathbb{C}) \times GL(1,\mathbb{C}) \times U(4,1),
\end{equation}
where
$\Sigma'$ corresponds to one (of the two) limit of discrete series representation of $U(4,1)$ with infinitesimal character $(1,0,-1,-2|1)$. 

\subsection{Combinatorial $\theta$-stable data and irreducible modules} \label{sec-irrep}
We now generalize the results in Section \ref{sec-tempered} to construct all irreducible, Hermitian $(\mathfrak{g},K)$-modules with real infinitesimal characters:
\begin{definition}[\cite{W22}, Section 4.3] \label{def-comb}
Let $G = U(p,q)$. A {\bf combinatorial $\theta$-stable datum} is given the following two components:
\begin{itemize}
\item[(a)] A $\lambda_a$-datum of $G$; and
\item[(b)] For each $\gamma_i$-block of size $(r_i,s_i)$ in (a), an element 
$\nu_i \in \mathbb{R}^{\min\{r_i,s_i\}}$ of the form $\nu_i = (\nu_{i,1} \geq \nu_{i,2} \geq \dots \geq -\nu_{i,2} \geq -\nu_{i,1})$ (we call $\nu_i$ the {\bf $\nu$-coordinates corresponding to the $\gamma_i$-block}), so that $\nu = (\nu_1; \nu_2; \dots)$ up to conjugacy.
\end{itemize}
\end{definition}

For each combinatorial $\theta$-stable datum, 
let $\delta$ be the $K$-type obtained from applying Example \ref{eg-lambdaa} on its $\lambda_a$-datum. Then one constructs an irreducible $(\mathfrak{g},K)$-module by slightly modifying the results in Section \ref{sec-tempered}:

\smallskip
\noindent \underline{\bf Method I:} By tensoring the inducing module of the principal series representation $\pi_{L_i}$ with an $A$-character, we replace $\pi_{L_i}$ with $\xi_{L_i}$, whose infinitesimal character equals $(\lambda_i'+\nu_{i,1}, \lambda_i'+\nu_{i,2}, \dots, \lambda_i'-\nu_{i,2}, \lambda_i'-\nu_{i,1})$. It is worth noting that 
the $\xi_{L_i}$ may be irreducible even if $\pi_{L_i}$ is reducible.

\smallskip
\noindent \underline{\bf Method II:} Replace 
$\Delta_{m,2\gamma}$, $\Delta_{r,2\gamma}$ in Step (i) and (ii) by
$J_1(2\gamma;2\nu_1) \otimes \dots \otimes J_1(2\gamma;2\nu_m)$ and $J_1(2\gamma;2\nu_1) \otimes \dots \otimes J_1(2\gamma;2\nu_r)$ respectively, where $J_1(\mu,\nu)$ is the $GL(1,\mathbb{C}) \cong \mathbb{C}^*$-character $z \mapsto (z)^{\frac{\mu+\nu}{2}}(\overline{z})^{\frac{-\mu+\nu}{2}}$.

\smallskip
The cohomologically or real parabolically induced module constructed by both methods above contain exactly one factor $X$, whose lowest $K$-types has a copy of $\delta$. Then we say $X$ is the irreducible $(\mathfrak{g},K)$-module {\bf corresponding to the combinatorial $\theta$-stable datum}. By the Langlands classification, all irreducible $(\mathfrak{g},K)$-modules with real infinitesimal characters arise in such a way. 

\begin{example} \label{eg-lambdaa2}
We continue with Example \ref{eg-lambdaa}, and 
consider the combinatorial $\theta$-stable datum with the $\nu$-coordinates as follows.
For the $1$-block, let $\nu_1=(\frac{3}{4},\frac{1}{2},-\frac{1}{2},-\frac{3}{4})$, and for $0$-block, let $\nu_2=(\frac{1}{3})$. We will describe the $\nu$-coordinates by a curved arrow with $\nu$-coordinates starting from the corresponding block:
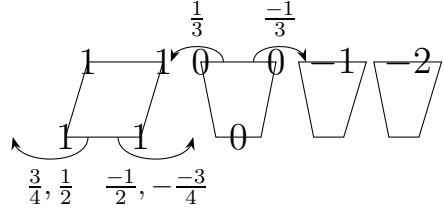
\begin{center}
\begin{tikzpicture}
\draw
    (0,0) node {\Large $1$}
 -- (1,0) node {\Large $1$}
-- (0.7,-1) node {\Large $1$}
-- (-0.3,-1) node {\Large $1$}
 -- cycle;

\draw
    (1.5,0) node {\Large $0$}
 -- (2.5,0) node {\Large $0$}
-- (2.3,-1) 
-- (2,-1) node {\Large $0$}
-- (1.7,-1) 
 -- cycle;

\draw
    (2.8,0) 
 -- (3.25,0) node {\Large $-1$}		
-- (3.7,0)
 -- (3.4,-1) 
-- (3,-1) 
 -- cycle;

\draw
    (3.8,0) 
 -- (4.25,0) node {\Large $-2$}		
-- (4.7,0)
 -- (4.4,-1) 
-- (4,-1) 
 -- cycle;

\draw[arrows = {-Stealth[]}]          (0,-1)   to [out=-90,in=-90]node[below]{$\frac{3}{4},\frac{1}{2}$} (-1,-1);
\draw[arrows = {-Stealth[]}]          (0.4,-1)   to [out=-90,in=-90]node[below]{$\frac{-1}{2},-\frac{-3}{4}$} (1.4,-1);

\draw[arrows = {-Stealth[]}]          (1.8,0)   to [out=90,in=90]node[above]{$\frac{1}{3}$} (1.1,0);
\draw[arrows = {-Stealth[]}]          (2.2,0)   to [out=90,in=90]node[above]{$\frac{-1}{3}$} (2.9,0);
\end{tikzpicture}.
\end{center}
Then its corresponding irreducible module $X$ is the lowest $K$-type subquotient of 

\smallskip
\noindent \underline{\bf Method I:} the cohomologically induced module $\mathcal{R}_{\mathfrak{q}}(\xi_L)$ with the same $\theta$-stable parabolic subalgebra $\mathfrak{q}_0$ as in Example \ref{eg-lambdaa}, and the (irreducible) principal series representation 
$\xi_L := \xi_{L_1} \otimes \xi_{L_2} \otimes \xi_{L_3} \otimes \xi_{L_4},$ 
with infinitesimal character
$$\lambda_a'+\nu := \left((\frac{-3}{2}+\frac{3}{4},\frac{-3}{2}+\frac{1}{2}\Big|\frac{-3}{2}-\frac{1}{2},\frac{-3}{2}-\frac{3}{4}),\ (1+\frac{1}{3},1\Big|1-\frac{1}{3}),\ (2|),\ (2|)\right).$$ 
In this case, the non-(pseudo)spherical principal series representation $\xi_{L_1}$ becomes irreducible with two lowest $L_1 \cap K$-types.

\smallskip
\noindent \underline{\bf Method II:} the real parabolically induced module 
$$\mathrm{Ind}_{MAN}^G\left(J_1(2; \frac{3}{2}) \boxtimes J_1(2;1) \boxtimes J_1(0;\frac{2}{3}) \boxtimes \Sigma \boxtimes 1\right)$$ 
with $P = MAN$ as given in \eqref{eq-glm}.

Note that $X$ has infinitesimal character $\Lambda = (\lambda_a, \nu) = (1,1,1,1,0,0,0,-1,-2)+(\frac{3}{4},-\frac{3}{4},\frac{1}{2},-\frac{1}{2},\frac{1}{3},-\frac{1}{3},0,0),$ which can be easily read off from the datum.
\end{example}
To end this section, we make the following remarks:

\begin{itemize}
    \item By induction in stages, the cohomologically induced module $\mathcal{R}_{\mathfrak{q}}(\pi_L)$ in Method I can be written as $\mathcal{R}_{\mathfrak{q}'}(\pi_{L'})$ for any $\theta$-stable parabolic $\mathfrak{q}'$ including $\mathfrak{q}$. The inducing module $\pi_{L'}$ can be read off by combining some of the $\lambda_a$-blocks in the datum.

\item As for the real parabolically induced module in Method II, the unipotent group $N$ can be chosen such that the $\nu$-parameter is dominant. In such a case, we call the real parabolically induced module a {\bf standard module}, whose lowest $K$-type subquotients are precisely the {\it quotients} of the induced module.
\end{itemize}

\subsection{Bottom-layer $K$-types} \label{subsec-bottom}
 In this section and the next section, we will introduce some criteria in determining the (non)-unitarity of $X$. In particular, we will reinterpret these criteria by just looking at the combinatorial $\theta$-stable datum of $X$.

\medskip
Let $X$ be an irreducible $(\mathfrak{g},K)$-module corresponding to the combinatorial $\theta$-stable datum $\mathcal{D}$.
As mentioned in the end of the previous subsection, one is interested in studying $\theta$-stable parabolic subalgebras $\mathfrak{q}'$ containing the quasisplit parabolic subalgebra $\mathfrak{q}$ determined by $\mathcal{D}$.

It is easy to see that the set all such $\mathfrak{q}'$
are in $1-1$ correspondence with the set of partitions of the 
$\lambda_a$-datum in $\mathcal{D}$. For instance, the smallest possible Levi subalgebra $\mathfrak{q}$ corresponds to the finest partition of the $\lambda_a$-datum, each consisting of a single $\gamma$-block. On the other extreme, $\mathfrak{g}$ corresponds to the single partition containing all $\gamma$-blocks. Consequently, all such $\theta$-stable parabolic subalgebras $\mathfrak{q}_0' \supset \mathfrak{q}_0$ corresponds to a partition 
\begin{equation} \label{eq-partition}
    \mathcal{D} = \bigsqcup_{i=1}^k \mathcal{D}_i
\end{equation}
of the $\theta$-stable datum of $X$.

\smallskip
Suppose the partition \eqref{eq-partition} defines a $\theta$-stable parabolic subalgebra $\mathfrak{q}_0' = \mathfrak{l}_0'+\mathfrak{u}_0'$, with 
\[L' = L_1' \times \dots \times L_k' := U(p_1,q_1) \times \dots \times U(p_k,q_k).\] 
Consider the {\it shifted datum}
\begin{equation} \label{eq-shifted}
    \widetilde{\mathcal{D}} = \bigsqcup_{i=1}^k \widetilde{\mathcal{D}_i}
\end{equation}
of $\mathcal{D}$ by subtracting the contents of \eqref{eq-partition} by $\rho(\mathfrak{u}) - \rho(\mathfrak{u}')$ (which is a constant on each $\mathcal{D}_i$). By induction in stages, $X$ is the lowest $K$-type subquotient of the cohomologically induced module
\[\mathcal{R}_{\mathfrak{q}'}(\widetilde{\pi_{L'}}), \quad \widetilde{\pi_{L'}} := \widetilde{\pi_1} \boxtimes \dots \boxtimes \widetilde{\pi_k}\] 
where each $\widetilde{\pi_i}$ is the $(\mathfrak{l}_i',L_i'\cap K)$-module corresponding the the $\theta$-stable datum $\widetilde{\mathcal{D}_i}$ in \eqref{eq-shifted}. 

\begin{definition} \label{def-bottom}
    Retain the above settings. We say a $K$-type $V_{\beta}$ is {\bf $\mathfrak{q}'$-bottom layer} if $\beta$ is of the form
    \begin{equation} \label{eq-bottombeta}
        \beta = (\mu_1; \dots; \mu_k) + 2\rho(\mathfrak{u}'\cap \mathfrak{p}),
    \end{equation}
    where each $V_{\mu_i}$ is a $(L_i' \cap K)$-type of $\widetilde{\pi_i}$. 
\end{definition}

\begin{theorem}[\cite{SV80}] \label{thm-bottom}
    Let $X$ be an irreducible Hermitian $(\mathfrak{g},K)$-module corresponding to the $\theta$-stable datum $\mathcal{D}$. Consider the shifted partition $\widetilde{\mathcal{D}} = \bigsqcup_{i=1}^k \widetilde{\mathcal{D}_i}$ in \eqref{eq-shifted} and the $(\mathfrak{l}_i', L_i' \cap K)$-modules $\widetilde{\pi_i}$ corresponding to $\widetilde{\mathcal{D}_i}$. 
    
    Suppose $V_{\beta} \in \widehat{K}$ is $\mathfrak{q}'$-bottom layer, and is of the form given in \eqref{eq-bottombeta}. Then $V_{\beta}$ appears in $X$ with the same multiplicity and signature as the $(L' \cap K)$-type $V_{\mu_1} \boxtimes \dots \boxtimes V_{\mu_k}$ in $\widetilde{\pi_{L'}} = \widetilde{\pi_1} \boxtimes \dots \boxtimes \widetilde{\pi_k}$.
\end{theorem}


\begin{example} \label{eg-parallel}
    Let $G=U(5,2)$, consider the combinatorial $\theta$-stable datum 
    
\begin{center}
$\mathcal{D} = $ \begin{tikzpicture} 
\draw
    (-0.8,0) 
 -- (-0.6,0) node {$2$}
-- (-0.4,0) 
-- (-0.52,-0.8) 
-- (-0.68,-0.8) 
 -- cycle;

\draw
    (0,0) 
 -- (0.2,0) node {$1$}
-- (0.4,0) 
-- (0.2,-0.8) 
-- (0,-0.8) node {$1$}
-- (-0.2,-0.8) 
 -- cycle;

\draw
    (0.8,0) 
 -- (1,0) node {$\frac{1}{2}$}
-- (1.2,0)   
-- (1.2,-0.8)   
-- (1,-0.8)  node {$\frac{1}{2}$}
-- (0.8,-0.8)
 -- cycle; 

\draw
    (1.6,0) 
 -- (1.8,0) node {$0$}
-- (2,0) 
-- (1.88,-0.8) 
-- (1.72,-0.8) 
 -- cycle;

\draw
    (2.4,0) 
 -- (2.6,0) node {$-1$}
-- (2.8,0) 
-- (2.68,-0.8) 
-- (2.52,-0.8) 
 -- cycle;
\end{tikzpicture}
 \end{center}
One may (for instance) partition $\mathcal{D}$ into $\mathcal{D}_1 = \{2\text{-block}\}$, 
$\mathcal{D}_2 = \{1\text{-block}\}$, and
$\mathcal{D}_3 = \{\frac{1}{2}\text{-block},  0\text{-block}, (-1)\text{-block}\}$ with $\mathfrak{l}_0' = \mathfrak{u}(1,0) + \mathfrak{u}(1,1) + \mathfrak{u}(3,1)$.

Consider the shifted datum $\widetilde{\mathcal{D}_2}$: its corresponding module $\widetilde{\pi_2}$ is the non-spherical principal series $L_2' = U(1,1)$ tensored with some power of $(\frac{\det}{|\det|})^{\frac{1}{2}}$. 
Assume that the $\nu$ parameter for $\widetilde{\mathcal{D}_2}$ is non-zero, then $\widetilde{\pi_2}$ is non-unitary at the two lowest $K$-types $V_{\mu_2} = V_{(a+1|a)}$ and $V_{\mu_2'} = V_{(a|a+1)}$. 

Take 
\begin{center} $\beta = (\mu_1,\mu_2,\mu_3) + 2\rho(\mathfrak{u}' \cap \mathfrak{p})$ \quad \quad  and \quad \quad $\beta' = (\mu_1,\mu_2',\mu_3) + 2\rho(\mathfrak{u}' \cap \mathfrak{p}),$\end{center}
where $V_{\mu_1}$ and $V_{\mu_3}$ are the lowest $K$-types of $\widetilde{\pi_1}$ and $\widetilde{\pi_3}$ respectively (lowest $(L_i \cap K)$-types are always $\mathfrak{q}'$-bottom layer).
Then the $(L' \cap K)$-types $V_{\mu_1} \boxtimes V_{\mu_2} \boxtimes V_{\mu_3}$ and $V_{\mu_1} \boxtimes V_{\mu_2'} \boxtimes V_{\mu_3}$ must have opposite signatures in the $(\mathfrak{l}', L' \cap K)$-module $\widetilde{\pi_{L'}} = \widetilde{\pi_1} \boxtimes \widetilde{\pi_2} \boxtimes \widetilde{\pi_3}$, and Theorem \ref{thm-bottom} implies that $X$ must be non-unitarity on $V_{\beta}$ and $V_{\beta'}$.
\end{example}

\begin{remark} \label{rmk-shifted}
In practice, one may just look at $\mathcal{D}_i$ rather than the shifted datum $\widetilde{\mathcal{D}_i}$. In doing so, the corresponding module of $\mathcal{D}_i$ may occur at the double cover of $U(p_i,q_i)$. 
Nevertheless, the modules corresponding to $\mathcal{D}_i$ and $\widetilde{\mathcal{D}_i}$ only differ from each other by tensoring with the unitary character $(\frac{\det}{|\det|})^{\frac{t}{2}}$ for some $t \in \mathbb{Z}$. From now on, we will just study $\mathcal{D}_i$, and write 
$\pi_i$ (instead of $\widetilde{\pi_i}$) as the $(\mathfrak{l}_i',L_i' \cap K)$-module corresponding to the shifted datum
$\widetilde{\mathcal{D}_i}$.
\end{remark}

For later purposes, we introduce the following:
\begin{definition} \label{def-upto}
Let $X$ be an irreducible, Hermitian $(\mathfrak{g},K)$-module with lowest $K$-types $\{\delta_1, \dots, \delta_r\}$, and $V$ be a finite-dimensional $K$-module. 
Suppose
$\{\chi_1, \dots, \chi_s\}$ be the set of $K$-types of $X$ appearing in the tensor products $\delta_i \otimes V$
for $1 \leq i \leq r$. Then the set 
$$\{\delta_1, \dots, \delta_r, \chi_1, \dots, \chi_s\}$$ 
are called the $K$-types of $X$ {\bf up to level $V$}.
And we say $X$ is {\bf non-unitary up to level $V$} if the Hermitian form of $X$ has indefinite signatures on the $K$-types up to level $V$.
\end{definition}

In \cite[Proposition 4.12]{W22}, one can determine whether the $K$-types $X$ up to level $V = \mathfrak{p}$ is $\mathfrak{q}'$-bottom layer by just looking at the shapes of its $\theta$-stable datum $\mathcal{D} = \bigsqcup_{i=1}^k \mathcal{D}_i$. This gives a powerful tool in determining (non)-unitarity of $X$.

\subsection{Good range condition}
\label{subsec-good}
In this section, we recall the notions of (weakly) good range and (weakly) fair range in cohomological induction. A more detailed account can be found in \cite{KV95}.

\begin{definition} \label{good range}
Suppose that the $(\mathfrak{l}, L\cap K)$-module $Z$ has a infinitesimal character $\lambda$. We say that $Z$ or $\lambda$ is in the {\bf weakly good range} or that $Z$ is {\bf weakly good} (relative to $\mathfrak{q}$ and $\mathfrak{g}$) if
\begin{equation}\label{weakly-good}
{\rm Re} \langle \lambda + \rho(\mathfrak{u}), \alpha \rangle \ge 0, \quad \forall \alpha\in \Delta(\mathfrak{u}, \mathfrak{h}).
\end{equation}
If all the inequalities in \eqref{weakly-good} are strict, then $Z$ is said to be in the {\bf good} range.

We say that $Z$ is {\bf weakly fair} if
\begin{equation}\label{weakly-fair}
{\rm Re} \langle \lambda + \rho(\mathfrak{u}), \alpha|_{\mathfrak{z}} \rangle \ge 0, \quad \forall \alpha\in \Delta(\mathfrak{u}, \mathfrak{h}),
\end{equation}
where $\mathfrak{z}$ is the center of $L$. If all the inequalities in \eqref{weakly-fair} are strict, then $Z$ is said to be in the {\bf fair} range.
\end{definition}

As for why the above definition is significant, let us state the following results (Theorem 0.51 of \cite{KV95}, Theorem 1.3 of \cite{V84}).

\begin{theorem}\label{vanishing}
Let $Z$ be an $(\mathfrak{l},L \cap K)$-module of finite length with  infinitesimal character $\lambda$, and suppose $Z$ is weakly good. Then
\begin{itemize}
\item[(a)] If $Z$ is irreducible, then $\mathcal{L}_{\mathfrak{q}}(Z)$ is irreducible or zero. Moreover, if $Z$ is good, then $Z$ is nonzero and irreducible if and only if $\mathcal{L}_{\mathfrak{q}}(Z)$ is nonzero and irreducible.
\item[(b)] If $Z$ is unitary, then $\mathcal{L}_{\mathfrak{q}}(Z)$ is unitary. Moreover, if $Z$ is good, $Z$ is unitary if and only if $\mathcal{L}_{\mathfrak{q}}(Z)$ is unitary.
\end{itemize}
\end{theorem}
 
We now describe how to interpret Definition \ref{good range} using combinatorial $\theta$-stable datum. Let $\mathfrak{q}$ be the quasisplit parabolic subalgebra corresponding to the $\theta$-stable datum $\mathcal{D}$, and 
$\mathfrak{q}' \supset \mathfrak{q}$ be any
$\theta$-stable parabolic subalgebra containing $\mathfrak{q}$ corresponding to the  partition 
$\mathcal{D} = \bigsqcup_{i=1}^k \mathcal{D}_i$
of $\mathcal{D}$. Then the infinitesimal character $\Lambda = (\lambda_a,\nu)$ of its corresponding module $X$ is also partitioned into
\[\Lambda = (\Lambda_1; \cdots; \Lambda_k)\]
up to permutation of coordinates. Define the $i^{th}$-{\bf segment} of the partition by the line segment $[e_i, b_i]$, where $e_i$ (resp. $b_i$) is the largest (resp. smallest) number
in $\Lambda_i$. 

\begin{example} \label{eg-good1}
    Let $G = U(5,4)$, and $\Sigma$ be the irreducible representation with $\theta$-stable datum of the form: 
    
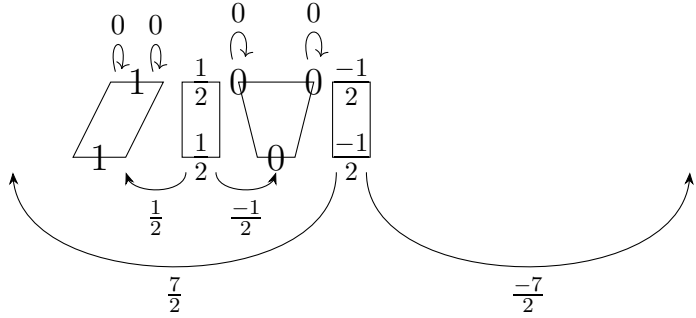
\begin{center}
\begin{tikzpicture}
\draw
    (0,0) node (1) {\Large $0$}
 -- (1,0) node (2) {\Large $0$} 
 -- (0.75,-1)
-- (0.5,-1) node {\Large $0$}
-- (0.25,-1) 
 -- cycle;

\draw[arrows = {-Stealth[]}]          (1.3,-1.2)   to [out=-90,in=-90]node[below]{$\frac{7}{2}$} (-3,-1.2);
\draw[arrows = {-Stealth[]}]          (1.7,-1.2)   to [out=-90,in=-90]node[below]{$\frac{-7}{2}$} (6,-1.2);

\draw
    (1.25,0) 
 -- (1.5,0) node {\Large $\frac{-1}{2}$}
-- (1.75,0)
-- (1.75,-1) 
-- (1.5,-1) node {\Large $\frac{-1}{2}$}
-- (1.25,-1)
 -- cycle;

\path (1) edge     [loop above]       node[] {$0$}         (1);
\path (2) edge     [loop above]       node[] {$0$}         (2);

\draw
    (-0.25,0) 
 -- (-0.5,0) node {\Large $\frac{1}{2}$}
-- (-0.75,0)
-- (-0.75,-1) 
-- (-0.5,-1) node {\Large $\frac{1}{2}$}
-- (-0.25,-1)
 -- cycle;
\draw[arrows = {-Stealth[]}]          (-0.7,-1.2)   to [out=-90,in=-90]node[below]{$\frac{1}{2}$} (-1.5,-1.2);
\draw[arrows = {-Stealth[]}]          (-0.3,-1.2)   to [out=-90,in=-90]node[below]{$\frac{-1}{2}$} (0.5,-1.2);

\draw
    (-1,0) 
 -- (-1.1,0) node (3) {} 
 -- (-1.35,0) node {\Large $1$}
-- (-1.6,0) node (4) {}
-- (-1.7,0) 
-- (-2.2,-1) 
-- (-1.85,-1) node {\Large $1$}
-- (-1.5,-1)
 -- cycle;

\path (3) edge     [loop above]       node[] {$0$}         (3);
\path (4) edge     [loop above]       node[] {$0$}         (4);
\end{tikzpicture}
\end{center}
    Suppose we partition the combinatorial $\theta$-stable datum of $\Sigma$ by $\mathcal{D} = \mathcal{D}_1 \sqcup \mathcal{D}_2$, where
    \[\mathcal{D}_1 := \{1\text{-block},\ \frac{1}{2}\text{-block}\} \quad \quad \quad \mathcal{D}_2 := \{0\text{-block},\ \frac{-1}{2}\text{-block}\}\]
Then $\Lambda_1 = (1,1,\frac{1}{2}+\frac{1}{2},\frac{1}{2}-\frac{1}{2}) = (1,1,1,0)$ and 
    $\Lambda_2 = (0,0+0,0+0,\frac{-1}{2}+\frac{7}{2},\frac{-1}{2}-\frac{7}{2}) = (0,0,0,3,-4)$, and hence
    \[[e_1,b_1]=[1,0], \quad \quad \quad [e_2, b_2] = [3,-4].\]
\end{example}

The following lemma determines whether $X$
is cohomologically induced from some proper $\theta$-stable parabolic subalgebra in good range by looking at its combinatorial $\theta$-stable data:

\begin{lemma} \label{lem-good}
Retain the setting in the above paragraphs. Then the following holds for $X$:
\begin{itemize}
    
    \item[(a)] 
    $X$
    is cohomologically induced from a $\theta$-stable parabolic subalgebra $\mathfrak{q}'$ in the good range if and only if the segments of the partition of the combinatorial $\theta$-stable datum of $X$ corresponding to $\mathfrak{q}'$ satisfy the following inequalities:
    \[e_1 \geq b_1 > e_2 \geq b_2 > \cdots > e_k \geq b_k.\]

\item[(b)] $X$ is always cohomologically induced from $\mathfrak{q}'$ in the fair range.
    \end{itemize}
\end{lemma}
The lemma follows immediately from Definition \ref{good range}. For instance, the above lemma
implies that the module $\Sigma$ in Example \ref{eg-good1} cannot be cohomologically induced from $\mathfrak{q}' = \mathfrak{l}'+\mathfrak{u}'$ with $\mathfrak{l}'_0 = \mathfrak{u}(2,2) \oplus \mathfrak{u}(3,2)$ in good range.
Indeed, one can easily check that $\Sigma$ cannot be cohomologically induced from any proper parabolic subalgebra $\mathfrak{q}' \supset \mathfrak{q}$ of $\mathfrak{g}$ in the good range. In other words, $\Sigma$ is \emph{fully supported}.

\bigskip
Now we can apply Theorem \ref{vanishing} to study the unitarity of $(\mathfrak{g},K)$-modules by just looking at its $\theta$-stable datum.
For example, consider the irreducible representation corresponding to the $\theta$-stable datum:
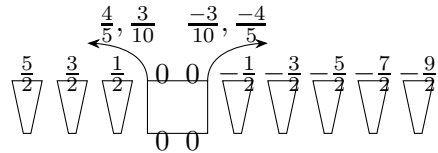
\begin{center}\begin{tikzpicture}
\foreach \x in {0,0.6,1.2,2.8,3.4,4.0,4.6,5.2}
	\draw (\x+0,0)--(\x+0.4,0)--(\x+0.25,-0.7) --(\x+0.15,-0.7)--cycle; 
\draw (1.8,0)--(2.6,0)--(2.6,-0.7)--(1.8,-0.7)--cycle; 
\node at (0.2,0.1) {$\frac{5}{2}$};
\node at (0.8,0.1) {$\frac{3}{2}$};
\node at (1.4,0.1) {$\frac{1}{2}$};
\node at (3.0,0.1) {$-\frac{1}{2}$};
\node at (3.6,0.1) {$-\frac{3}{2}$};
\node at (4.2,0.1) {$-\frac{5}{2}$}; 
\node at (4.8,0.1) {$-\frac{7}{2}$};
\node at (5.4,0.1) {$-\frac{9}{2}$};
\node at (2.0,0.1) {$0$};
\node at (2.4,0.1) {$0$};
\node at (2.0,-0.8) {$0$};
\node at (2.4,-0.8) {$0$};

\draw[arrows = {-Stealth[]}]          (1.8,0)   to [out=90,in=0]node[above]{$\frac{4}{5},\frac{3}{10}$} (1,0.5);
\draw[arrows = {-Stealth[]}]          (2.6,0)   to [out=90,in=180]node[above]{$\frac{-3}{10},\frac{-4}{5}$} (3.4,0.5);
\end{tikzpicture}\end{center}  
in $U(10,2)$. By partitioning the datum into 
\[\left\{\frac{5}{2}\text{-block},\frac{3}{2}\text{-block} \right\} \sqcup \left\{\frac{1}{2}\text{-block}, \dots, \frac{-1}{2}\text{-block} \right\} \sqcup \left\{\frac{-3}{2}\text{-block}, \dots, \frac{-9}{2}\text{-block}\right\}\]
Then by Lemma \ref{lem-good}, the module is cohomologically induced from $\mathfrak{q}' = \mathfrak{l}'+\mathfrak{u}'$ given by the above partition in the good range. Since the first and last datum is a unitary $U(2,0)$ and $U(4,0)$-module, its unitarity is determined by the middle datum for $U(4,2)$  (up to some $\rho(\mathfrak{u})$-shift of coordinates):
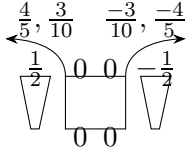
\begin{center}\begin{tikzpicture}
\foreach \x in {1.2,2.8}
	\draw (\x+0,0)--(\x+0.4,0)--(\x+0.25,-0.7) -- (\x+0.14,-0.7) --cycle; 
\draw (1.8,0)--(2.6,0)--(2.6,-0.7)--(1.8,-0.7)--cycle;  
\node at (1.4,0.1) {$\frac{1}{2}$};
\node at (3.0,0.1) {$-\frac{1}{2}$}; 
\node at (2.0,0.1) {$0$};
\node at (2.4,0.1) {$0$};
\node at (2.0,-0.8) {$0$};
\node at (2.4,-0.8) {$0$};

\draw[arrows = {-Stealth[]}]          (1.8,0)   to [out=90,in=0]node[above]{$\frac{4}{5},\frac{3}{10}$} (1,0.5);
\draw[arrows = {-Stealth[]}]          (2.6,0)   to [out=90,in=180]node[above]{$\frac{-3}{10},\frac{-4}{5}$} (3.4,0.5);
\end{tikzpicture}\end{center}  
By the results in Section \ref{subsec-i} below, it is not unitary. So Theorem \ref{vanishing}(b) implies that the full representation is also not unitary.


\subsection{Fundamental representations}
\label{sec-fund} 
In \cite{W22}, the first named author introduced the notion of fundamental cases and fundamental data to study the unitary dual of $U(p,q)$ and possibly for all real reductive groups:
\begin{definition}\label{Fund_datum} Let $G$ be any real connected reductive group. An irreducible, Hermitian $(\mathfrak{g},K)$-module with real infinitesimal character is called {\bf fundamental} if its $\lambda_a$-value satisfies $\langle \lambda_a, \alpha^{\vee} \rangle \leq 1$
for all simple roots $\alpha \in \Delta(\mathfrak{g},\mathfrak{h})$.
\end{definition}

In the case when $G = U(p,q)$, and $X$
is an irreducible module corresponding to 
the combinatorial $\theta$-stable datum $\mathcal{D}$, 
the above definition can be translated into:
$X$ is fundamental if and only if all the neighbouring $\lambda_a$-blocks of $\mathcal{D}$ have differences $\leq 1$.

Note that one can always partition any combinatorial $\theta$-stable datum $\mathcal{D}$ into: 
\[\mathcal{D} = \bigsqcup_i \mathcal{F}_i\] 
so that each sub-datum $\mathcal{F}_i$ is fundamental (from now on, we call such datum and its corresponding $(\mathfrak{g},K)$-module {\bf fundamental}), and the gap between two neighbouring sub-data is $> 1$, 
i.e. the $\lambda_a$-datum of $X$ is of the form:

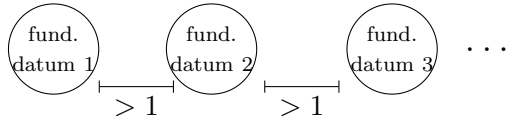
\begin{center}
\begin{tikzpicture}
\draw (-1.1,-0.3) node {\scriptsize{fund.}};
\draw (-1.1,-0.7) node {\scriptsize{datum 1}};
\draw (-1.1,-0.5) circle (0.6);

\draw [|-|] (-0.5,-1) -- node[below] {$>1$} (0.5,-1);

\draw (1,-0.3) node {\scriptsize{fund.}};
\draw (1,-0.7) node {\scriptsize{datum 2}};
\draw (1,-0.5) circle (0.6);

\draw   [|-|] (1.7,-1) -- node[below] {$>1$} (2.7,-1);

\draw (3.4,-0.3) node {\scriptsize{fund.}};
\draw (3.4,-0.7) node {\scriptsize{datum 3}};
\draw (3.4,-0.5) circle (0.6);

\draw (4.7,-0.5) node {\Large $\dots$};
\end{tikzpicture}
\end{center}
For instance, the following $\lambda_a$-datum is partitioned into two fundamental data:

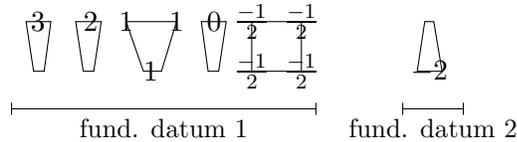
\begin{center}
\begin{tikzpicture}

\draw [|-|](-0.2,-1.16) -- node[below] {\small{fund. datum 1}} (3.86,-1.16);

\draw
    (0,0) 
 -- (0.16,0) node {$3$}
-- (0.33,0) 
-- (0.24,-0.66) 
-- (0.1,-0.66) 
 -- cycle;

\draw
    (0.66,0) 
 -- (0.86,0) node {$2$}
-- (1,0) 
-- (0.9,-0.66) 
-- (0.76,-0.66) 
 -- cycle;

\draw
    (1.33,0) node {$1$}
 -- (2,0) node {$1$}		
 -- (1.8,-0.66) 
-- (1.66,-0.66) node {$1$}
-- (1.56,-0.66) 
 -- cycle;

\draw
    (2.33,0) 
 -- (2.5,0) node {$0$}
-- (2.66,0) 
-- (2.56,-0.66) 
-- (2.42,-0.66) 
 -- cycle;

\draw
    (3,0) node {$\frac{-1}{2}$}
 -- (3.66,0) node {$\frac{-1}{2}$}
-- (3.66,-0.66) node {$\frac{-1}{2}$}
-- (3,-0.66) node {$\frac{-1}{2}$}
 -- cycle;

\draw [|-|](5,-1.16) -- node[below] {\small{fund. datum 2}} (5.82,-1.16);

\draw
    (5.3,0) 
 -- (5.42,0) 
-- (5.53,-0.66) 
-- (5.36,-0.66) node {$-2$}
-- (5.2,-0.66) 
 -- cycle;
\end{tikzpicture} 
\end{center}

Here is a conjecture on the structure of unitary dual for fundamental representations:
\begin{conjecture} \label{conj-fundamental}
Let $X$ be a fundamental module with real infinitesimal character $\Lambda$. Suppose $\langle \Lambda, \alpha^{\vee} \rangle > 1$ for some simple root $\alpha \in \Delta(\mathfrak{g},\mathfrak{h})$, then $X$ is not unitary up to level $\mathfrak{p}$.
\end{conjecture}

For $G = U(p,q)$, we have the following:
\begin{theorem}[\cite{W22}] \label{thm-fund}
The above conjecture holds for $U(p,q)$.
\end{theorem}

To show how one can apply Theorem \ref{thm-fund} to study the full unitary dual of $U(p,q)$, suppose $X$ is \emph{any} irreducible Hermitian module with real infinitesimal character, whose $\theta$-stable datum $\mathcal{D} = \bigsqcup_i \mathcal{F}_i$ is partitioned into fundamental data. If there exists $\mathcal{F}_i$ such that its corresponding $(\mathfrak{l}_i', L_i \cap K)$-module $\pi_i$ (with infinitesimal character $\Lambda_i$ up to a $\rho(\mathfrak{u})$-shift) satisfies the hypothesis of Conjecture \ref{conj-fundamental}. Then by Theorem \ref{thm-fund}, $\pi_i$ is not unitary up to level $\mathfrak{l}_i' \cap \mathfrak{p}$. As a result, the bottom layer $K$-type arguments in Section \ref{subsec-bottom} imply that $X$ is also not unitary up to level $\mathfrak{p}$. Combining the above arguments with Lemma \ref{lem-good}(a), one can prove Vogan's \emph{fundamental parallelepiped conjecture} for $U(p,q)$ (cf. \cite[Section 7]{W22}).

\medskip
Inspired by the above line of arguments, we propose a possible strategy to determine the unitary dual for all $G$: firstly, one determines the unitary dual of fundamental representations of $G$. Special focus is put on the \emph{non-unitary certificates} for these representations, i.e. $K$-types that have opposite signatures with the lowest $K$-types of the fundamental representations. Then one may apply bottom layer $K$-type arguments to relate the non-unitary certificates of the fundamental representations with that of all general representations. For instance, in the above paragraph, such a relationship is built for level $\mathfrak{p}$ $K$-types. In the next two sections, we will apply these line of ideas to obtain the full unitary dual of $U(n,1)$ and $U(n,2)$.



\section{Unitary dual of $U(n,1)$}
\label{sec-un1}

\subsection{Intertwining operators of $U(n,1)$} \label{subsec-un1}

We recall the intertwining operators of $U(n,1)$ between the representations which are induced from one-dimensional representations of minimal parabolic subgroup (cf. \cite[Section 10]{KS83}).

Let $P_1=M_1A_1N_1$ be a minimal parabolic subgroup of $U(n,1)$, where $M_1A_1 = GL(1,\mathbb{C}) \times U(n-1,0)$.
For $m \in \frac{1}{2}\mathbb{Z}$ and $\nu \in \mathbb{R}_{\geq 0}$, consider the intertwining operator on the principal series representations:
\begin{equation} \label{eq-iota}
\iota : \mathrm{Ind}_{M_1A_1N_1}^{U(n,1)}(J_1(2m;2\nu) \boxtimes \mathrm{triv} \boxtimes 1)\longrightarrow  \mathrm{Ind}_{M_1A_1N_1}^{U(n,1)}(J_1(2m;-2\nu) \boxtimes \mathrm{triv} \boxtimes 1). \end{equation}
If $|2m| \leq n$, then the lowest $K$-type of the above principal series representations is $V_{(0,\cdots,0|2m)}$. More explicitly, the other $K$-types are of the form $V_{(k,0,\cdots,0,-l\ |\ 2m-k+l)}$ for $k, l\geq 0$, all appearing with multiplicity one.
By normalizing $\iota$ such that it is equal to identity on the lowest $K$-type, $\iota$ acts on $V_{(k,0,\cdots,0,-l\ |\ 2m-k+l)}$ by a scalar
\[\prod_{j=0}^{k-1}\big(\frac{2j-(2\nu+2m-n)}{2j+(2\nu-2m+n)}\big)\cdot \prod^{l-1}_{j=0}\big(\frac{2j-(2\nu-2m-n)}{2j+(2\nu+2m+n)}\big).\]
In particular, if $0\leq 2\nu < \min\{-2m+n,2m+n\}$, then $\iota$ is positive definite for all $k, l \geq 0$.

\subsection{Fundamental cases} \label{subsec-un1fund}
We begin by classifying which fundamental representations of $U(n,1)$ are unitary, i.e. the fundamental 
unitary dual of $U(n,1)$.
Firstly, if the $\theta$-stable datum of $X$ consists only of $(1,0)$ or $(0,1)$-blocks, then it must be a discrete series representation. 
If there is a $(1,1)$-parallelogram block, then Example \ref{eg-parallel} implies that $\nu = 0$ for this block (which corresponds to a limit of discrete series representation), otherwise it is not unitary up to level $\mathfrak{p}$. 

Therefore, by shifting the module by some power of the unitary character $\frac{\det}{|\det|}$ (which amounts to shifting the content of the $\lambda_a$-blocks by a constant), it suffices to study the combinatorial $\theta$-stable datum of the form:

\begin{center} \begin{tikzpicture}

\draw (0,-0.3) node  {$\cdots \ast \cdots$};

\draw (0.875,0.05) node {\small $-\frac{n-4}{2}$}
    (0.75,0) 
 -- (0.875,0) 
-- (1,0) 
-- (0.925,-0.5) 
-- (0.875,-0.5) 
-- (0.825,-0.5) 
 -- cycle;

\draw (-0.875,0.05) node {\small $\frac{n-4}{2}$}
    (-0.75,0) 
 -- (-0.875,0)
-- (-1,0) 
-- (-0.925,-0.5) 
-- (-0.875,-0.5) 
-- (-0.825,-0.5) 
 -- cycle;

\draw (1.825,0.05) node {\small $-\frac{n-2}{2}$}
    (1.7,0) 
 -- (1.825,0)
-- (1.95,0) 
-- (1.875,-0.5) 
-- (1.825,-0.5) 
-- (1.775,-0.5) 
 -- cycle;

\draw (-1.825,0.05) node {\small $\frac{n-2}{2}$}
    (-1.7,0) 
 -- (-1.825,0)
-- (-1.95,0) 
-- (-1.875,-0.5) 
-- (-1.825,-0.5) 
-- (-1.775,-0.5) 
 -- cycle;\end{tikzpicture},
 \quad  where $\ast =$ \begin{tikzpicture}
\draw (0.0,0.05) node {\small $m$} 
 (0.0,-0.5) node {\small $m$} 
    (-0.15,0) 
 -- (0.0,0)  
 -- (0.15, 0) 
-- (0.15,-0.5) 
-- (0.0,-0.5) 
-- (-0.15,-0.5) 
 -- cycle;
\end{tikzpicture} (if $\frac{n}{2} - m \in \mathbb{Z}+\frac{1}{2}$) or 
\begin{tikzpicture}
\draw (0.0,0.05) node {\small $m$\ $m$}
    (-0.3,0) 
 -- (0.0,0) 
 -- (0.3, 0) 
-- (0.15,-0.5) 
-- (0.0,-0.5) node {\small $m$}
-- (-0.15,-0.5) 
 -- cycle;
\end{tikzpicture} (if $\frac{n}{2} - m \in \mathbb{Z}$)
 \end{center}  
for $|2m| \geq n-1$ (note that the $|2m| = n$ case corresponds to a $(1,1)$-parallelogram). We only consider the first case in full detail. By the discussions in Section \ref{sec-irrep}, the induced module corresponding to the combinatorial $\theta$-stable datum
\begin{center}
\begin{tikzpicture}
\draw
    (-0.15,0) 
 -- (0.0,0) node {\scriptsize $m$}
 -- (0.15, 0) 
-- (0.15,-0.5) 
-- (0.0,-0.5) node {\scriptsize $m$}
-- (-0.15,-0.5) 
 -- cycle;

\draw (0.825,0.05) node {\scriptsize $m-\frac{1}{2}$}
    (0.7,0) 
 -- (0.825,0) 
-- (0.95,0) 
-- (0.875,-0.5) 
-- (0.825,-0.5) 
-- (0.775,-0.5) 
 -- cycle;

\draw (-0.825,0.05) node {\scriptsize $m+\frac{1}{2}$}
    (-0.7,0) 
 -- (-0.825,0) 
-- (-0.95,0) 
-- (-0.875,-0.5) 
-- (-0.825,-0.5) 
-- (-0.775,-0.5) 
 -- cycle;

\draw (-1.25,-0.25) node {$\dots$};
\draw (1.35,-0.25) node {$\dots$};

\draw (1.825,0.05) node {\scriptsize $-\frac{n-2}{2}$}
    (1.7,0) 
 -- (1.825,0) 
-- (1.95,0) 
-- (1.875,-0.5) 
-- (1.825,-0.5) 
-- (1.775,-0.5) 
 -- cycle;

\draw (-1.825,0.05) node {\scriptsize $\frac{n-2}{2}$}
    (-1.7,0) 
 -- (-1.825,0) 
-- (-1.95,0) 
-- (-1.875,-0.5) 
-- (-1.825,-0.5) 
-- (-1.775,-0.5) 
 -- cycle;

 \draw[arrows = {-Stealth[]}]          (-0.15,0)   to [out=90,in=90]node[above]{$\nu$} (-2.3,0);
\draw[arrows = {-Stealth[]}]          (0.15,0)   to [out=90,in=90]node[above]{$-\nu$} (2.3,0);
 \end{tikzpicture}
 \end{center}  
is $\mathrm{Ind}_{M_1A_1N_1}^{U(n,1)}(J_1(2m;2\nu) \boxtimes \mathrm{triv} \boxtimes 1)$. By choosing $N$ appropriately, the fundamental module $X$ corresponding to this datum is the image of the intertwining operator $\iota$ in Equation \eqref{eq-iota}.
So the results in Section \ref{subsec-un1} imply that if 
\[0 \leq \nu \leq \min\{\frac{n-2}{2}-m,m+\frac{n-2}{2}\}+1,\]
then the Hermitian form on the $K$-types in $X$ are all positive definite. Otherwise, the form is indefinite on the lowest $K$-types $V_{(0,\dots,0|2m)}$ and the level $\mathfrak{p}^-$ $K$-type $V_{(0,\dots,0,-1|2m+1)}$. In conclusion, we have:

    


\begin{theorem} \label{thm-un1fund}
    Let $X$ be a fundamental representation of $U(n,1)$. Then its $\theta$-stable datum must be of the form:

\begin{center}
\begin{tikzpicture}
\draw
    (0,0) 
 -- (0.125,0) node {$\gamma$}
-- (0.25,0) 
-- (0.175,-0.5) 
-- (0.125,-0.5) node {\mbox{}}
-- (0.075,-0.5) 
 -- cycle;
 \end{tikzpicture} \ , \
\begin{tikzpicture}
\draw
    (0.325,-0.5) 
 -- (0.425,-0.5) 
-- (0.5,-1) 
-- (0.375,-1) node {$\gamma$}
-- (0.25,-1) 
 -- cycle;
\end{tikzpicture} \ , \
\begin{tikzpicture}
\draw
    (0,0) node (1) {}
 -- (0.15,0) node {$\gamma$}
 -- (0.3, 0) node (2) {}
-- (0.15,-0.5) 
-- (0,-0.5) node {$\gamma$}
-- (-0.15,-0.5) 
 -- cycle;
\path (1) edge     [loop above]       node[] {$0$}         (1);
\path (2) edge     [loop above]       node[] {$0$}         (2);
 \end{tikzpicture} \ , \
\begin{tikzpicture}
\draw
    (-0.15,0) node (1) {}
 -- (0.0,0) node {$\gamma$}
 -- (0.15, 0) node (2) {}
-- (0.3,-0.5) 
-- (0.15,-0.5) node {$\gamma$}
-- (0,-0.5) 
 -- cycle;
\path (1) edge     [loop above]       node[] {$0$}         (1);
\path (2) edge     [loop above]       node[] {$0$}         (2);
 \end{tikzpicture} \ , \
\begin{tikzpicture}
\draw
    (-0.3,0) 
 -- (0.0,0) node {\small $\gamma$\ $\gamma$}
 -- (0.3, 0) 
-- (0.15,-0.5) 
-- (0.0,-0.5) node {\small $\gamma$}
-- (-0.15,-0.5) 
 -- cycle;

\draw (0.825,0.05) node {\small $\gamma-1$}
    (0.7,0) 
 -- (0.825,0) 
-- (0.95,0) 
-- (0.875,-0.5) 
-- (0.825,-0.5) 
-- (0.775,-0.5) 
 -- cycle;

\draw (-0.825,0.05) node {\small $\gamma+1$}
    (-0.7,0) 
 -- (-0.825,0) 
-- (-0.95,0) 
-- (-0.875,-0.5) 
-- (-0.825,-0.5) 
-- (-0.775,-0.5) 
 -- cycle;

\draw (-1.25,-0.25) node {$\dots$};
\draw (1.35,-0.25) node {$\dots$};

\draw (1.825,0.05) node {\small $\gamma-b$}
    (1.7,0) 
 -- (1.825,0) 
-- (1.95,0) 
-- (1.875,-0.5) 
-- (1.825,-0.5) 
-- (1.775,-0.5) 
 -- cycle;

\draw (-1.825,0.05) node {\small $\gamma+a$}
    (-1.7,0) 
 -- (-1.825,0) 
-- (-1.95,0) 
-- (-1.875,-0.5) 
-- (-1.825,-0.5) 
-- (-1.775,-0.5) 
 -- cycle;

\draw[arrows = {-Stealth[]}]          (-0.3,0)   to [out=90,in=90]node[above]{$\nu$} (-2.3,0);
\draw[arrows = {-Stealth[]}]          (0.3,0)   to [out=90,in=90]node[above]{$-\nu$} (2.3,0);
 \end{tikzpicture} ($0 \leq \nu \leq \min\{a,b\}+1$)\ ,\ 
 \begin{tikzpicture}
\draw
    (-0.15,0) 
 -- (0.0,0) node {\scriptsize $\gamma$}
 -- (0.15, 0) 
-- (0.15,-0.5) 
-- (0.0,-0.5) node {\scriptsize $\gamma$}
-- (-0.15,-0.5) 
 -- cycle;

\draw (0.825,0.05) node {\scriptsize $\gamma-\frac{1}{2}$}
    (0.7,0) 
 -- (0.825,0)  
-- (0.95,0) 
-- (0.875,-0.5) 
-- (0.825,-0.5) 
-- (0.775,-0.5) 
 -- cycle;

\draw (-0.825,0.05) node {\scriptsize $\gamma+\frac{1}{2}$}
    (-0.7,0) 
 -- (-0.825,0)  
-- (-0.95,0) 
-- (-0.875,-0.5) 
-- (-0.825,-0.5) 
-- (-0.775,-0.5) 
 -- cycle;

\draw (-1.25,-0.25) node {$\dots$};
\draw (1.35,-0.25) node {$\dots$};

\draw (1.825,0.05) node {\scriptsize $\gamma-\frac{2b-1}{2}$}
    (1.7,0) 
 -- (1.825,0)  
-- (1.95,0) 
-- (1.875,-0.5) 
-- (1.825,-0.5) 
-- (1.775,-0.5) 
 -- cycle;

\draw (-1.825,0.05) node {\scriptsize $\gamma+\frac{2a-1}{2}$}
    (-1.7,0) 
 -- (-1.825,0)  
-- (-1.95,0) 
-- (-1.875,-0.5) 
-- (-1.825,-0.5) 
-- (-1.775,-0.5) 
 -- cycle;

\draw[arrows = {-Stealth[]}]          (-0.15,0)   to [out=90,in=90]node[above]{$\nu$} (-2.5,0);
\draw[arrows = {-Stealth[]}]          (0.15,0)   to [out=90,in=90]node[above]{$-\nu$} (2.5,0);
 \end{tikzpicture} ($0 \leq \nu \leq \min\{\frac{2a+1}{2},\frac{2b+1}{2}\}$)
\end{center}

    Otherwise, it is not unitary up to level $\mathfrak{p}$.
\end{theorem}

\subsection{General case}
We now consider the case of a general Hermitian, irreducible $(\mathfrak{g},K)$-module $X$ with real infinitesimal character. As discussed at the end of Section \ref{sec-fund}, 
partition the combinatorial $\theta$-stable datum of $X$ into fundamental data $$\mathcal{D} = \bigsqcup_{i=1}^k \mathcal{F}_i,$$
and consider the fundamental
$(\mathfrak{l}_i',L_i'\cap K)$-modules $\pi_i$ corresponding to $\mathcal{F}_i$ (cf. Remark \ref{rmk-shifted}).

Since $G = U(n,1)$, the fundamental datum $\mathcal{F}_i$ must either correspond to (i) a unitary module of $U(p_i,0)$ if $q_i = 0$, or (ii) a fundamental representation of $U(p_i,1)$ if $q_i = 1$. In the second case, if the $\mathcal{F}_i$ is not of the form given in Theorem \ref{thm-un1fund}, then the theorem implies that $\pi_i$ is not unitary up to level $\mathfrak{l}_i' \cap \mathfrak{p}$. Consequently, the bottom layer $K$-type arguments imply that $X$
 is also not unitary up to level $\mathfrak{p}$.  
 
As a result, all $\mathcal{F}_i$ must be of the form given in Theorem \ref{thm-un1fund}. Moreover, in such a case, it is easy to see from Lemma \ref{lem-good} that $X$ is cohomologically induced from $\mathfrak{q}'$ in the good range where unitarity is preserved (Theorem \ref{good range}). In conclusion, we have:
 \begin{corollary}
\label{cor-un1}
    Let $G = U(n,1)$, and $X$ be an irreducible, Hermitian $(\mathfrak{g},K)$-module with real infinitesimal character and $\theta$-stable datum $\mathcal{D} = \bigsqcup_{i=1}^k \mathcal{F}_i$. Then $X$ is unitary if and only if each $\mathcal{F}_i$ is as given in Theorem \ref{thm-un1fund}.
    Otherwise, it is not unitary up to level $\mathfrak{p}$.
\end{corollary}

\begin{example}\label{eg-u41}
We present an example for  $G = U(4,1)$. 
Up to a shift of a power of $\frac{\det}{|\det|}$, we
list all fundamental unitary representations of $G$.

\begin{itemize}
    \item[(i)] If there are only $(1,0)$ or $(0,1)$-blocks, e.g.
    \begin{center}
\begin{tikzpicture}
\draw
    (0,0) 
 -- (0.125,0) node {$2$}
-- (0.25,0) 
-- (0.175,-0.5) 
-- (0.125,-0.5) node {\mbox{}}
-- (0.075,-0.5) 
 -- cycle;
\draw
    (0.45,0) 
 -- (0.55,0) 
-- (0.625,-0.5) 
-- (0.5,-0.5) node {$1$}
-- (0.4,-0.5) 
 -- cycle;

 \draw
    (0.8,0) 
 -- (0.925,0) node {$0$}
-- (1.05,0) 
-- (0.975,-0.5) 
-- (0.925,-0.5) node {\mbox{}}
-- (0.875,-0.5) 
 -- cycle;

  \draw
    (1.2,0) 
 -- (1.325,0) node {$-1$}
-- (1.45,0) 
-- (1.375,-0.5) 
-- (1.325,-0.5) node {\mbox{}}
-- (1.275,-0.5) 
 -- cycle;

   \draw
    (1.6,0) 
 -- (1.725,0) node {$-2$}
-- (1.85,0) 
-- (1.775,-0.5) 
-- (1.725,-0.5) node {\mbox{}}
-- (1.675,-0.5) 
 -- cycle;
\end{tikzpicture} \end{center}
    Then it must be unitary, and  corresponds to a discrete series representation. 
    For instance, the above example corresponds to a discrete series representation with infinitesimal character $\rho = (2,1,0,-1,-2)$. More explicitly, it is equal to $A_{\mathfrak{b}}(0)$, where the $\theta$-stable Borel subalgebra $\mathfrak{b} = \mathfrak{t} + \mathfrak{n}$ is determined by the element $(2,0,-1,-2\ |\ 1) \in \mathfrak{t}^*$.
    
    \item[(ii)] If the datum contains a $(1,1)$ parallelogram block, then $\nu$ must be $= 0$ in order for the module to be unitary, e.g.
     \begin{center}
\begin{tikzpicture}
\draw
    (0.3,0) 
 -- (0.425,0) node {$1$}
-- (0.55,0) 
-- (0.475,-0.5) 
-- (0.425,-0.5) node {\mbox{}}
-- (0.375,-0.5) 
 -- cycle;

 \draw
    (0.8,0) 
 -- (0.925,0) node {$0$}
-- (1.05,0) 
-- (0.975,-0.5) 
-- (0.85,-0.5) node {$0$}
-- (0.725,-0.5) 
 -- cycle;

  \draw
    (1.2,0) 
 -- (1.325,0) node {$-1$}
-- (1.45,0) 
-- (1.375,-0.5) 
-- (1.325,-0.5) node {\mbox{}}
-- (1.275,-0.5) 
 -- cycle;

   \draw
    (1.6,0) 
 -- (1.725,0) node {$-2$}
-- (1.85,0) 
-- (1.775,-0.5) 
-- (1.725,-0.5) node {\mbox{}}
-- (1.675,-0.5) 
 -- cycle;
\end{tikzpicture} \end{center}   
    Indeed, all such modules are {\it limit of discrete series representations}, whose infinitesimal character can be directly read off from its datum, and the unique lowest $K$-type can be obtained by the algorithm in Example \ref{eg-lambdaa}. In the above example, the corresponding limit of discrete series module $A_{\mathfrak{b}}(\lambda)$ is also cohomologically induced from the same theta-stable Borel $\mathfrak{b}$
    as in the first case, but in the {\it weakly good} range instead.
\item[(iii)] Finally, consider the case when the datum contains a $(1,1)$ rectangle or $(2,1)$ trapezoid block. By symmetry, we assume the $(1,1)$ or $(2,1)$ block appears on the right of the datum - indeed, by `flipping' the data horizontally, one gets the contragredient representation (up to a $\det$-shift):
\begin{enumerate}
    \item \begin{tikzpicture}
\draw
    (-0.3,0) 
 -- (0.0,0) node {\small $0$\ $0$}
 -- (0.3, 0) 
-- (0.15,-0.5) 
-- (0.0,-0.5) node {\small $0$}
-- (-0.15,-0.5) 
 -- cycle;

\draw
    (-0.7,0) 
 -- (-0.825,0) node {\small $1$}
-- (-0.95,0) 
-- (-0.875,-0.5) 
-- (-0.825,-0.5) 
-- (-0.775,-0.5) 
 -- cycle;

\draw
    (-1.3,0) 
 -- (-1.425,0) node {\small $2$}
-- (-1.55,0) 
-- (-1.475,-0.5) 
-- (-1.425,-0.5) 
-- (-1.375,-0.5) 
 -- cycle;
\end{tikzpicture}, $0 \leq \nu \leq 1$;

\item \begin{tikzpicture}
\draw
    (-0.15,0) 
 -- (0.0,0) node {\scriptsize $\frac{1}{2}$}
 -- (0.15, 0) 
-- (0.15,-0.5) 
-- (0.0,-0.5) node {\scriptsize $\frac{1}{2}$}
-- (-0.15,-0.5) 
 -- cycle;

\draw
    (0.7,0) 
 -- (0.825,0) node {\scriptsize $0$}
-- (0.95,0) 
-- (0.875,-0.5) 
-- (0.825,-0.5) 
-- (0.775,-0.5) 
 -- cycle;

\draw
    (-0.7,0) 
 -- (-0.825,0) node {\scriptsize $1$}
-- (-0.95,0) 
-- (-0.875,-0.5) 
-- (-0.825,-0.5) 
-- (-0.775,-0.5) 
 -- cycle;

\draw
    (-1.4,0) 
 -- (-1.525,0) node {\scriptsize $2$}
-- (-1.65,0) 
-- (-1.575,-0.5) 
-- (-1.525,-0.5) 
-- (-1.475,-0.5) 
 -- cycle;

 \end{tikzpicture}, $0 \leq \nu \leq \frac{3}{2}$;

\item \begin{tikzpicture}
\draw
    (-0.3,0) 
 -- (0.0,0) node {\small $0$\ $0$}
 -- (0.3, 0) 
-- (0.15,-0.5) 
-- (0.0,-0.5) node {\small $0$}
-- (-0.15,-0.5) 
 -- cycle;

\draw
    (0.7,0) 
 -- (0.825,0) node {\small $-1$}
-- (0.95,0) 
-- (0.875,-0.5) 
-- (0.825,-0.5) 
-- (0.775,-0.5) 
 -- cycle;

\draw
    (-0.7,0) 
 -- (-0.825,0) node {\small $1$}
-- (-0.95,0) 
-- (-0.875,-0.5) 
-- (-0.825,-0.5) 
-- (-0.775,-0.5) 
 -- cycle;
\end{tikzpicture}, $0 \leq \nu \leq 2$.
\end{enumerate}
 For instance, if we take $\nu = 3/2$ in case (2), then by Lemma \ref{lem-good} it is cohomologically induced in the weakly good range from theta-stable parabolic $\mathfrak{q}_0' = \mathfrak{l}_0' + \mathfrak{u}_0'$ with $\mathfrak{l}_0 = \mathfrak{u}(1,0) \oplus \mathfrak{u}(3,1)$. 
 
 We focus on the  $\mathfrak{u}(3,1)$ subdatum. By computing its infinitesimal character and lowest $K$-type as in Example \ref{eg-lambdaa2}, the sub-datum \begin{tikzpicture}
\draw
    (-0.15,0) 
 -- (0.0,0) node {\scriptsize $\frac{1}{2}$}
 -- (0.15, 0) 
-- (0.15,-0.5) 
-- (0.0,-0.5) node {\scriptsize $\frac{1}{2}$}
-- (-0.15,-0.5) 
 -- cycle;

\draw
    (0.5,0) 
 -- (0.625,0) node {\scriptsize $0$}
-- (0.75,0) 
-- (0.675,-0.5) 
-- (0.625,-0.5) 
-- (0.575,-0.5) 
 -- cycle;

\draw
    (-0.5,0) 
 -- (-0.625,0) node {\scriptsize $1$}
-- (-0.75,0) 
-- (-0.675,-0.5) 
-- (-0.625,-0.5) 
-- (-0.575,-0.5) 
 -- cycle;
 \end{tikzpicture} 
 with $\nu = \frac{3}{2}$ corresponds to the unitary character $(\frac{\det}{|\det|})^{\frac{1}{2}}$ of (the double cover of) $L_2' = U(3,1)$. Therefore, the module corresponding to the full datum in $\mathfrak{u}(4,1)$ is an $A_{\mathfrak{q}'}(\lambda)$-module in the weakly good range. 

 \smallskip
 Again, by looking at the infinitesimal character and lowest $K$-type, one can conclude that the trivial module occurs in case (3) with $\nu = 2$.
\end{itemize}


\end{example}

\begin{remark} \label{rmk-un1}
    We end this section by noting that the arguments in this subsection can be generalized into the following:
    
    Let $G = U(p,q)$, and $X$ be an irreducible $(\mathfrak{g},K)$-module such that its $\theta$-stable datum is $\mathcal{D} = \bigsqcup_{i=1}^k \mathcal{F}_i$. Suppose each $\mathcal{F}_i$ corresponds to a fundamental module of real rank $\leq 1$, then $X$ is unitary if and only if $\mathcal{F}_i$ is of the form given in Theorem \ref{thm-un1fund}. 
    
    Note that we are also allowing $U(1,q)$ and $U(0,q)$-modules, where the fundamental data appearing in Theorem \ref{thm-un1fund} are flipped upside down in these cases. 
\end{remark}

\section{Unitary dual of $U(n,2)$} \label{sec-un2}

\subsection{Fundamental cases}
We first classify the fundamental, irreducible unitary representations:
\begin{enumerate}
\item[(a)] Suppose the fundamental datum consists only of $(1,0)$, $(0,1)$ blocks
\begin{tikzpicture}
    \draw (2.4,0)--
    (2.5,0)--
    (2.7,0.7)--(2.2,0.7)--cycle;
\end{tikzpicture}, 
\begin{tikzpicture}
    \draw (2.4,0)--
    (2.5,0)--
    (2.7,-0.7)--(2.2,-0.7)--cycle;
    \end{tikzpicture},
or $(1,1)$-parallelogram blocks
\begin{tikzpicture}
    \draw (-0.8,0)--(-0.4,0)--(-0.3,-0.7)--(-0.7,-0.7)--cycle ;
\end{tikzpicture}, \begin{tikzpicture}
    \draw (-0.8,-0.7)--(-0.4,-0.7)--(-0.3,-0)--(-0.7,-0)--cycle ;
\end{tikzpicture}, then Example \ref{eg-parallel} implies that the $\nu$-value of the $(1,1)$-parallelogram block must be zero. In all such cases, the data correspond to limit of discrete series representations. 

\item[(b)] there are $k+l+m$ $(1,0)$-blocks and two $(r,1)$-blocks ($r = 0,1,2$), i.e. 
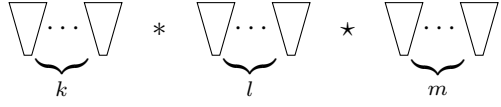
\begin{center}\begin{tikzpicture}
\foreach \x in {0,1,2.5,3.5,5,6}
	\draw (\x+0,0)--
 (\x+0.5,0)--(\x+0.3,-0.7)--
 (\x+0.2,-0.7)--
 cycle; 
\node at (0.75,-0.35) {$\cdots$};\node at (3.25,-0.35) {$\cdots$}; \node at (5.75,-0.35) {$\cdots$}; 
	\node at (2.0,-0.35) {$\ast$}; \node at (4.5,-0.35) {$\star $}; 
 \node at (0.7,-1.0) {$\underbrace{}_{k}$};
 \node at (3.2,-1.0) {$\underbrace{}_{l}$}; 
 \node at (5.7,-1.0) {$\underbrace{}_{m}$};
\end{tikzpicture}\end{center} 
where the $(r,1)$-blocks $\ast, \star$ are given by: \begin{itemize}
\item[(i)] both blocks are of the shapes \begin{tikzpicture} 
\draw (-1.6,0)--(-1.2,0)--(-1.2,-0.7)--(-1.6,-0.7)--cycle ;\end{tikzpicture} or \begin{tikzpicture}
    \draw (1.0,0)--(1.8,0)--(1.6,-0.7)--(1.2,-0.7)--cycle;
\end{tikzpicture};
\item[(ii)] one block is of the shape given in (i), and the other is of the shape \begin{tikzpicture}
    \draw (-0.8,0)--(-0.4,0)--(-0.3,-0.7)--(-0.7,-0.7)--cycle ;
\end{tikzpicture} , 
\begin{tikzpicture}
    \draw (0.2,0)--(0.6,0)--(0.5,-0.7)--(0.1,-0.7)--cycle ;
\end{tikzpicture};

\item[(iii)] one block is of the shape given in (i), and the other is of the shape 
\begin{tikzpicture}
    \draw (2.4,0)--
    (2.5,0)--
    (2.7,-0.7)--(2.2,-0.7)--cycle;
    \end{tikzpicture}
\end{itemize}

\item[(c)] there are $k+m$ $(1,0)$-blocks and one $(r,2)$-block ($r = 1,2,3$):
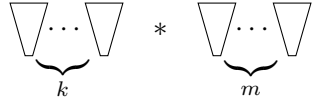
\begin{center}\begin{tikzpicture}
\foreach \x in {0,1,2.5,3.5}
	\draw (\x+0,0)--(\x+0.5,0)--(\x+0.3,-0.7)--
 (\x+0.2,-0.7)--
 cycle; 
\node at (0.75,-0.35) {$\cdots$};\node at (3.25,-0.35) {$\cdots$};
	\node at (2.0,-0.35) {$\ast$}; 
 \node at (0.7,-1.0) {$\underbrace{}_{k}$};
 \node at (3.2,-1.0) {$\underbrace{}_{m}$}; 
\end{tikzpicture}\end{center} 
where the $(r,2)$-block $\ast$ takes one of the following shapes:
\begin{itemize}
\item[(i)] \begin{tikzpicture} 
\draw (-1.3,0)--(-0.5,0)--(-0.5,-0.7)--(-1.3,-0.7)--cycle;
\end{tikzpicture}
or 
\begin{tikzpicture}
\draw (2.9,0)--(4.1,0)--(3.9,-0.7)--(3.1,-0.7)--cycle;     
\end{tikzpicture};

\item[(ii)] 
\begin{tikzpicture}
\draw (1.6,0)--(2.4,0)--(2.2,-0.7)--(1.4,-0.7)--cycle; 
\end{tikzpicture} or 
\begin{tikzpicture}
\draw (0.1,0)--(0.9,0)--(1.1,-0.7)--(0.3,-0.7)--cycle ;
\end{tikzpicture};
\item[(iii)] \begin{tikzpicture}
\draw (4.8,0)--(5.2,0)--(5.4,-0.7)--(4.6,-0.7)--cycle;
\end{tikzpicture}.
\end{itemize} 
\end{enumerate}

In the next few subsections, we will study Cases (b-c)(i), Cases (b-c)(ii) and Cases (b-c)(iii) individually.

\subsection{Unitary dual for Case (i)} \label{subsec-i}
For Cases (b)(i) and (c)(i), they are all {\it basic} in the sense of \cite{KS83}. By tensoring the module by a power of $\frac{\det}{|\det|}$, the $\lambda_a$-blocks in all these cases have content of the form:
\[\lambda_a^{\Theta,\Phi} := (\gamma, \gamma-1, \dots, \gamma - (k-1), \Theta, \gamma - k, \dots, -\gamma + m, \Phi, -\gamma + (m-1),  \dots, -\gamma+1, -\gamma\ |\  \Theta, \Phi )\]
where $\gamma = \frac{n-3}{2}$, and
    \[\gamma+1 > \Theta \geq \Phi > -\gamma-1.\]
For instance, if we take $\Theta = \Phi = \gamma - k$, we have
\[\lambda_a^{\Theta,\Theta} = (\gamma,  \dots, \gamma-(k-1), \Theta,  \Theta, \Theta, \gamma-(k+1),  \dots, -\gamma+1, -\gamma\ |\  \Theta, \Theta ),\]
i.e. we are in Case (c)(i) with a $(3,2)$-trapezoidal block (of content $\Theta$).

By the discussions in Section \ref{sec-irrep} again, the $\theta$-stable datum in Cases (b)(i) and (c)(i) with $\lambda_a$-content $\lambda_a^{\Theta,\Phi}$ corresponds to 
the lowest $K$-type subquotient $J(\Theta,\Phi;\nu_{\Theta},\nu_{\Phi})$ of the parabolically induced module:
\begin{equation}\label{ind_spherical}
\mathrm{Ind}_{M_2A_2N_2}^G\left(  J_1(2\Theta;2\nu_{\Theta}) \boxtimes J_1(2\Phi;2\nu_{\Phi})\boxtimes \mathrm{triv} \boxtimes 1\right)\quad (M_2A_2 = GL(1,\mathbb{C}) \times GL(1,\mathbb{C}) \times U(n-2,0)).
\end{equation}
It suffices to study the cases when $\nu_{\Theta}, \nu_{\Phi} \geq 0$. By choosing $N$ appropriately, one can make the module \eqref{ind_spherical} standard, so that by writing $x := \max\{\nu_{\Theta}, \nu_{\Phi}\}$ and $y := \min\{\nu_{\Theta}, \nu_{\Phi}\}$, 
$J(\Theta,\Phi;\nu_{\Theta}, \nu_{\Phi})$ is the image of the long intertwining operator:
$$(x,y) \xrightarrow{s_{\alpha,1}} (y, x) \xrightarrow{s_{\beta,1}} (y, -x) \xrightarrow{s_{\alpha,2}} (-x,y) \xrightarrow{s_{\beta,2}} (-x,-y),$$
where the $s_{\beta}$ intertwining operators are as given in Section \ref{subsec-un1}, and $s_{\alpha}$ operators are the 
$GL(1,\mathbb{C}) \times GL(1,\mathbb{C})$-intertwining operators.

Define the {\bf fundamental rectangle} by
$$\left\{(\nu_{\Theta}, \nu_{\Phi})\ |\ 0\leq \nu_{\Theta}\leq \frac{n-1}{2}-|\Theta|, 0\leq \nu_{\Phi}\leq \frac{n-1}{2}-|\Phi| \right\}.$$ 
By the results in Section \ref{subsec-un1}, the intertwining operators $s_{\beta,1}$, $s_{\beta,2}$ are injective in the interior of the rectangle. In the following theorem, we determine the unitarity $J(\Theta,\Phi; \nu_{\Theta}, \nu_{\Phi})$ for $(\nu_{\Theta}, \nu_{\Phi})$ inside the fundamental rectangle and outside the fundamental respectively:
\begin{theorem}[\cite{KS83}, Section 2] \label{thm-ci}
Let $J(\Theta,\Phi; \nu_{\Theta}, \nu_{\Phi})$ be a fundamental $(\mathfrak{g},K)$-module as defined above. In the first quadrant $\{(\nu_{\Theta}, \nu_{\Phi})| \nu_{\Theta}, \nu_{\Phi} \geq 0\}$, the points where $J(\Theta,\Phi; \nu_{\Theta}, \nu_{\Phi})$ is unitary are given as follows:
\begin{enumerate}
    \item[(a)] Inside the fundamental rectangle, the points inside the `triangles' 
    \begin{itemize} 
    \item $\{\nu_{\Theta} + \nu_{\Phi}\leq \Theta-\Phi+1\}$;  
    \item $\{\nu_{\Theta} + \nu_{\Phi} \leq \Theta-\Phi+k+1, \nu_{\Theta} - \nu_{\Phi}\geq \Theta-\Phi+k\}, k\in \mathbb{N}_+ $
    \item $\{\nu_{\Theta} + \nu_{\Phi} \leq \Theta-\Phi+k+1, \nu_{\Phi} - \nu_{\Theta}\geq \Theta-\Phi+k\}, k\in \mathbb{N}_+ $
    \end{itemize}
    and the lines 
    \begin{itemize}
        \item $\{\nu_{\Theta} - \nu_{\Phi}=\Theta-\Phi+k\}, k\in \mathbb{N}_+$. 
          \item $\{\nu_{\Phi} - \nu_{\Theta}=\Theta-\Phi+k\}, k\in \mathbb{N}_+$.
    \end{itemize}
    \item[(b)] Outside the fundamental rectangle, there are no unitary representations except when $\Theta \cdot \Phi  > 0$  or $\Theta = \Phi = 0$,  with
    \[ (\nu_{\Theta}, \nu_{\Phi}) = \begin{cases} (\frac{n-1}{2}-|\Theta|, \frac{n+1}{2}-|\Phi|);  (\frac{n+1}{2}-|\Theta|, \frac{n-1}{2}-|\Phi|) & \mathrm{if}\ \Theta = \Phi,
    \\ (\frac{n-1}{2}-|\Theta|, \frac{n+1}{2}-|\Phi|) & \mathrm{if}\  \Theta > \Phi>0
    \\ (\frac{n+1}{2}-|\Theta|, \frac{n-1}{2}-|\Phi|) & \mathrm{if}\ 0 > \Theta >\Phi \\
   \end{cases}.\]
\end{enumerate}
Otherwise, it is not unitary 
\begin{itemize}
\item up to level $\mathrm{Sym}^t(\mathfrak{k}^-) = V_{(0,\dots,0|t,-t)}$ for some $1\leq t\leq \max\{\frac{n-1}{2}-|\Theta|-\Theta+\Phi, \frac{n-1}{2}-|\Phi|-\Phi+\Theta\}$ in (a); or
\item up to level $\mathfrak{k}^{-} = V_{(0,\dots,0|1,-1)}$,  $\mathfrak{p}^{\pm}$  or $\wedge^2\mathfrak{p}^{\pm}$ in (b).
\end{itemize}
\end{theorem}

\begin{example} \label{eg-ci}
Let $G = U(9,2)$. We present two examples of unitary $J(\Theta,\Phi;\nu_{\Theta}, \nu_{\Phi})$:

\begin{itemize}
    \item[(1)] When $\Theta = \Phi = 0$, i.e. the $\lambda_a$ blocks are of the form \begin{tikzpicture}
\draw 
    (-0.3,0) 
 -- (0.0,0) node {\small $0$\ $0$\ $0$}
 -- (0.3, 0) 
-- (0.15,-0.5) 
-- (0.0,-0.5) node {\small $0$\ $0$}
-- (-0.15,-0.5) 
 -- cycle;

\draw (0.825,0) node {\small $-1$}
    (0.7,0) 
 -- (0.825,0)  
-- (0.95,0) 
-- (0.875,-0.5) 
-- (0.825,-0.5) 
-- (0.775,-0.5) 
 -- cycle;

\draw
    (-0.7,0) 
 -- (-0.825,0) node {\small $1$}
-- (-0.95,0) 
-- (-0.875,-0.5) 
-- (-0.825,-0.5) 
-- (-0.775,-0.5) 
 -- cycle;

\draw
    (-1.125,0) 
 -- (-1.25,0) node {\small $2$}
-- (-1.375,0) 
-- (-1.3,-0.5) 
-- (-1.25,-0.5) 
-- (-1.2,-0.5) 
 -- cycle;

\draw
    (1.225,0) 
 -- (1.35,0) node {\small $-2$}
-- (1.475,0) 
-- (1.4,-0.5) 
-- (1.35,-0.5) 
-- (1.3,-0.5) 
 -- cycle;

\draw
    (1.7,0) 
 -- (1.825,0) node {\small $-3$}
-- (1.95,0) 
-- (1.875,-0.5) 
-- (1.825,-0.5) 
-- (1.775,-0.5) 
 -- cycle;

\draw
    (-1.7,0) 
 -- (-1.825,0) node {\small $3$}
-- (-1.95,0) 
-- (-1.875,-0.5) 
-- (-1.825,-0.5) 
-- (-1.775,-0.5) 
 -- cycle;
 \end{tikzpicture}. Then $\lambda_a^{0,0}$ corresponds to the trivial $K = U(9) \times U(2)$-type by Theorem \ref{thm-datum}, and the spherical unitary dual $J(0,0;\nu_{\Theta}, \nu_{\Phi})$ is given by:
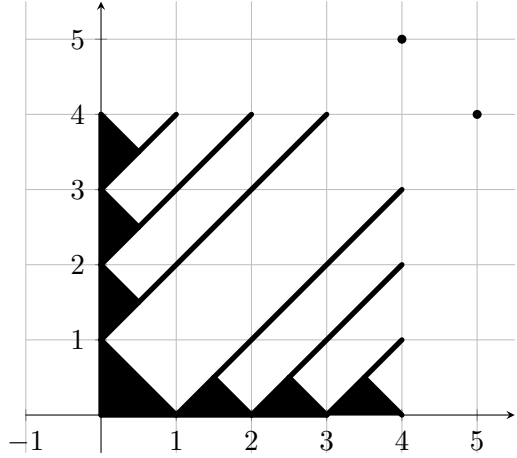
\begin{center}
\begin{tikzpicture}[line cap=round,line join=round,x=1cm,y=1cm]
\begin{axis}[
x=1cm,y=1cm,
axis lines=middle,
ymajorgrids=true,
xmajorgrids=true,
xmin=-1,
xmax=5.5,
ymin=-0.5,
ymax=5.5,
xtick={-1,0,...,5},
ytick={0,1,...,5},]
\clip(-0.5,-0.5) rectangle (5.5,5.5);
\fill[line width=2pt] (3,0) -- (3.5,0.5) -- (4,0) -- cycle;
\fill[line width=2pt] (1,0) -- (1.5,0.5) -- (2,0) -- cycle;
\fill[line width=2pt] (2,0) -- (2.5,0.5) -- (3,0) -- cycle;
\fill[line width=2pt] (1,0) -- (0,1) -- (0,0) -- cycle;
\fill[line width=2pt] (0,1) -- (0.5,1.5) -- (0,2) -- cycle;
\fill[line width=2pt] (0,2) -- (0.5,2.5) -- (0,3) -- cycle;
\fill[line width=2pt] (0,3) -- (0.5,3.5) -- (0,4) -- cycle;
\draw [line width=2pt] (1,0)-- (1.5,0.5);
\draw [line width=2pt] (1.5,0.5)-- (2,0);
\draw [line width=2pt] (2,0)-- (1,0);
\draw [line width=2pt] (2,0)-- (2.5,0.5);
\draw [line width=2pt] (2.5,0.5)-- (3,0);
\draw [line width=2pt] (3,0)-- (2,0);
\draw [line width=2pt] (1,0)-- (0,1);
\draw [line width=2pt] (0,1)-- (0,0);
\draw [line width=2pt] (0,0)-- (1,0);
\draw [line width=2pt] (0,1)-- (0.5,1.5);
\draw [line width=2pt] (0.5,1.5)-- (0,2);
\draw [line width=2pt] (0,2)-- (0,1);
\draw [line width=2pt] (0,2)-- (0.5,2.5);
\draw [line width=2pt] (0.5,2.5)-- (0,3);
\draw [line width=2pt] (0,3)-- (0,2);
\draw [line width=2pt] (0,3)-- (0.5,3.5);
\draw [line width=2pt] (0.5,3.5)-- (0,4);
\draw [line width=2pt] (0,4)-- (0,3);
\draw [line width=2pt] (2.5,0.5)-- (4,2);
\draw [line width=2pt] (1.5,0.5)-- (4,3);
\draw [line width=2pt] (0.5,1.5)-- (3,4);
\draw [line width=2pt] (0.5,2.5)-- (2,4);
\draw [line width=2pt] (0.5,3.5)-- (1,4);
\draw [line width=2pt] (3,0)-- (4,1);
\draw [line width=2pt] (3.5,0.5)-- (4,0);
\draw[black,fill=black] (5,4) circle (1.5pt);
\draw[black,fill=black] (4,5) circle (1.5pt);
\end{axis}
\end{tikzpicture}
\end{center}
Here the fundamental rectangle is given by the $0 \leq \nu_{\Theta} \leq 4$ and $0 \leq \nu_{\Phi} \leq 4$. And the isolated point $(5,4)$ (or $(4,5)$) outside of the rectangle corresponds to the trivial module $J(0,0;5,4) = \mathrm{triv}$.

\item[(2)] When $\Theta = 0$ and $\Phi = -2$, i.e. the $\lambda_a$-blocks are of the form \begin{tikzpicture}
\draw
    (-0.025,0) node {\small $0$} 
 -- (-0.25,0) 
 -- (-0.475, 0) node {\small $0$}
-- (-0.35,-0.5) 
-- (-0.25,-0.5) node {\small $0$}
-- (-0.15,-0.5) 
 -- cycle;

\draw
    (-0.7,0) 
 -- (-0.825,0) node {\small $1$}
-- (-0.95,0) 
-- (-0.875,-0.5) 
-- (-0.825,-0.5) 
-- (-0.775,-0.5) 
 -- cycle;

\draw
    (0.2,0) 
 -- (0.325,0) node {\small $-1$}
-- (0.45,0) 
-- (0.375,-0.5) 
-- (0.325,-0.5) 
-- (0.275,-0.5) 
 -- cycle;

\draw
    (0.7,0) node {\small $-2$}
 -- (0.95,0) 
-- (1.2,0) node {\small $-2$}
-- (1.05,-0.5) 
-- (0.95,-0.5) node {\small $-2$}
-- (0.85,-0.5) 
 -- cycle;

\draw
    (-1.225,0) 
 -- (-1.35,0) node {\small $2$}
-- (-1.475,0) 
-- (-1.4,-0.5) 
-- (-1.35,-0.5) 
-- (-1.3,-0.5) 
 -- cycle;

\draw
    (-1.7,0) 
 -- (-1.825,0) node {\small $3$}
-- (-1.95,0) 
-- (-1.875,-0.5) 
-- (-1.825,-0.5) 
-- (-1.775,-0.5) 
 -- cycle;

\draw
    (1.5,0) 
 -- (1.625,0) node {\small $-3$}
-- (1.75,0) 
-- (1.675,-0.5) 
-- (1.625,-0.5) 
-- (1.575,-0.5) 
 -- cycle;
\end{tikzpicture}. The fundamental rectangle is given by $0 \leq \nu_{\Theta} \leq 4$ and $0 \leq \nu_{\Phi} \leq 2$, and unitary dual of $J(0,-2;\nu_{\Theta}, \nu_{\Phi})$ is given by:
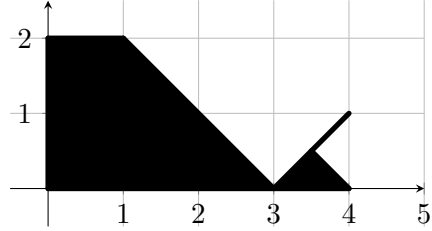
\begin{center}
    \begin{tikzpicture}[line cap=round,line join=round,x=1cm,y=1cm]
\begin{axis}[
x=1cm,y=1cm,
axis lines=middle,
ymajorgrids=true,
xmajorgrids=true,
xmin=-0.5,
xmax=5,
ymin=-0.5,
ymax=2.5,
xtick={0,1,...,5},
ytick={0,1,...,2},]
\clip(-0.5,-0.5) rectangle (5,2.5);
\fill[line width=2pt] (3,0) -- (1,2) -- (0,2) -- (0,0) -- cycle;
\fill[line width=2pt] (3,0) -- (3.5,0.5) -- (4,0) -- cycle;
\draw [line width=2pt] (3,0)-- (1,2);
\draw [line width=2pt] (1,2)-- (0,2);
\draw [line width=2pt] (0,2)-- (0,0);
\draw [line width=2pt] (0,0)-- (3,0);
\draw [line width=2pt] (3,0)-- (3.5,0.5);
\draw [line width=2pt] (3.5,0.5)-- (4,0);
\draw [line width=2pt] (4,0)-- (3,0);
\draw [line width=2pt] (3.5,0.5)-- (4,1);
\end{axis}
\end{tikzpicture} 
\end{center}

\end{itemize}

\end{example}

\begin{proof}
    We will give a sketch proof of the theorem:

\bigskip
\noindent (a) Suppose $(\nu_{\Theta}, \nu_{\Phi})$ is in the fundamental rectangle, we first consider the points on the $\nu_{\Theta}$ or $\nu_{\Phi}$-axis, i.e. $\nu_{\Phi} = 0$ or $\nu_{\Theta} =0$. In the first case,  $J(\Theta,\Phi;\nu_{\Theta},0)$ is the lowest $K$-type subquotient of
$$\mathrm{Ind}
_{M'A'N'}^G(J_1(2\Phi,0) \boxtimes \Sigma' \boxtimes 1) \quad (M'A' := GL(1,\mathbb{C}) \times U(n-1,1)),$$
where $\Sigma'$ is the $U(n-1,1)$-module whose $\theta$-stable datum is obtained from that of $J(\Theta,\Phi;\nu_{\Theta},0)$ by removing the $\Phi$'s in its content and the $\nu_{\Phi} = 0$'s in its $\nu$-coordinates. By the classification theorem of $\widehat{U(n-1,1)}$ in Section \ref{sec-un1}, $\Sigma'$ is unitary and hence $J(\Theta,\Phi;\nu_{\Theta},0)$ is also unitary.

As discussed before the theorem, one only needs to study the $GL(1,\mathbb{C}) \times GL(1,\mathbb{C})$-intertwining operators $s_{\alpha,1}$, $s_{\alpha,2}$ inside the fundamental rectangle. The points where the $s_{\alpha,i}$ are not injective (the reducibility lines) are precisely given by the theorem. Therefore, by continuity arguments, one can conclude that all the ``triangles'' appearing in Theorem \ref{thm-ci}(a) correspond to unitary representations.

As for the ``lines'' 
\begin{equation} \label{eq-line} \nu_{\Theta}-\nu_{\Phi} = \Theta - \Phi +k\end{equation} appearing in Theorem \ref{thm-ci}(a), note that the starting points $(\nu_{\Theta}, \nu_{\Phi}) = (\Theta - \Phi + k,0)$ of the lines correspond the {\it degenerate series representation} (cf. \cite[p.149]{KS83}):
$$J(\Theta,\Phi;\Theta - \Phi +k,0) = \mathrm{Ind}
_{M''A''N''}^G(J_2(2\Theta,2\Phi; 2(\Theta - \Phi +k), 0) \boxtimes \mathrm{triv} \boxtimes 1)\quad (M''A'' = GL(2,\mathbb{C}) \times U(n-2,0)),$$
where $J_2(\mu_1,\mu_2; \nu_1,\nu_2)$ is defined as the lowest $K$-type subquotient of $\mathrm{Ind}_{GL(1)^2}^{GL(2)}(J_1(\mu_1;\nu_1) \boxtimes J_1(\mu_2;\nu_2))$. To see the other points on the line \eqref{eq-line} are unitary, one only needs to show that the modules $J(\Theta,\Phi;\nu_{\Theta},\nu_{\Phi})$ are also of the form
\begin{equation}\label{eq-j2}
J(\Theta,\Phi;\nu_{\Theta},\nu_{\Phi}) = \mathrm{Ind}
_{M''A''N''}^G(J_2(2\Theta,2\Phi; 2\nu_{\Theta}, 2\nu_{\Phi}) \boxtimes \mathrm{triv} \boxtimes 1),\end{equation}
so that the signatures of all $K$-types remain unchanged for all $\nu_{\Phi} \geq 0$ along the line \eqref{eq-line}. To see why \eqref{eq-j2} holds, we only consider the case of $\nu_{\Theta} > \nu_{\Phi} \geq 0$, and the other case is similar. Consider the intertwining operators
\begin{align*}
s_{\alpha,1}:\  &\mathrm{Ind}_{M_2A_2N_2}^G\left(J_1(2\Theta;2\nu_{\Theta}) \boxtimes J_1(2\Phi;2\nu_{\Phi}) \boxtimes \mathrm{triv} \boxtimes 1\right) \\ &\longrightarrow \mathrm{Ind}_{M_2A_2N_2}^G\left( J_1(2\Phi;2\nu_{\Phi}) \boxtimes J_1(2\Theta;2\nu_{\Theta}) \boxtimes \mathrm{triv} \boxtimes 1\right)\end{align*}
and
\begin{align*}
s_{\alpha,2}: \ &\mathrm{Ind}_{M_2A_2N_2}^G\left(J_1(2\Phi;2\nu_{\Phi}) \boxtimes J_1(2\Theta;-2\nu_{\Theta}) \boxtimes \mathrm{triv} \boxtimes 1\right) \\
&\longrightarrow \mathrm{Ind}_{M_2A_2N_2}^G\left(J_1(2\Theta;-2\nu_{\Theta}) \boxtimes  J_1(2\Phi;2\nu_{\Phi}) \boxtimes \mathrm{triv} \boxtimes 1\right)\end{align*}
Note that $\mathrm{im}(s_{\alpha,1})$ is equal to the induced module \eqref{eq-j2}, so we only need to show that $\ker(s_{\alpha,2}) \subseteq \ker(s_{\alpha,1})$: Indeed, by induction in stages, one has that
\begin{equation} \label{eq-salpha1}
\ker(s_{\alpha,1}) =  \mathrm{Ind}
_{M_2A_2N_2}^G(J_1(2\Theta+k; 2\nu_{\Theta}-k) \boxtimes J_1(2\Phi-k; 2\nu_{\Phi}+k) \boxtimes \mathrm{triv} \boxtimes 1) 
\end{equation}
(this is the module $U_2(0)$ in \cite[Equation 8.6(a)]{KS83}). As for $\ker(s_{\alpha,2})$, by consider the $GL(1,\mathbb{C}) \times GL(1,\mathbb{C})$-intertwining operator again, $\ker(s_{\alpha,2}) = 0$ except when $2\nu_{\Phi} \in \mathbb{N}$, where the kernel is of the form
\begin{equation} \label{eq-salpha2}
\ker(s_{\alpha,2}) =  \mathrm{Ind}
_{M_2A_2N_2}^G(J_1(2\Theta+k+2\nu_{\Phi}; -k-2(\Theta-\Phi)) \boxtimes J_1(2\Phi-k-2\nu_{\Phi}; -k) \boxtimes \mathrm{triv} \boxtimes 1).
\end{equation}
By similar analysis using intertwining operators, one concludes that \eqref{eq-salpha2} is an irreducible module when $\nu_{\Phi}\neq 0$. In such a case, $\ker(s_{\alpha,2})$ occurs as a composition factor of $\ker(s_{\alpha,1})$: firstly, note that $\ker(s_{\alpha,1})$ in 
\eqref{eq-salpha1} has the same composition factor as
\[\mathrm{Ind}
_{M_2A_2N_2}^G(J_1(2\Theta+k; 2\nu_{\Theta}-k) \boxtimes J_1(2\Phi-k; -(2\nu_{\Phi}+k)) \boxtimes \mathrm{triv} \boxtimes 1).\]
By applying the $GL(1,\mathbb{C}) \times GL(1,\mathbb{C})$-intertwining operator on the above induced module as before, one concludes that it contains the induced module \eqref{eq-salpha2} as a composition factor (cf. \cite[Equation (8.7)]{KS83}). Now the result follows by noting that \eqref{eq-salpha2} occurs in the standard module \eqref{ind_spherical} with multiplicity one, which can be easily checked by computing the multiplicities of $K$-types, or simply by the last paragraph of \cite[p. 151]{KS83}.

\medskip
To conclude, we have accounted for the unitarizability of the points stated in Theorem \ref{thm-ci}(1). As for the other points inside the fundamental rectangle, the explicit formulas given in Section 10 of \cite{KS83} imply that $J(\Theta,\Phi;\nu_{\Theta},\nu_{\Phi})$ is not unitary on some level $Sym^t \mathfrak{k}^-$ $K$-types $V_{(0,\dots,0|2\Theta+t,2\Phi-t)}$, where $1\leq t\leq \max\{\frac{n-1}{2}-|\Theta|-\Theta+\Phi, \frac{n-1}{2}-|\Phi|-\Phi+\Theta\}$.

More precisely, the induced module \eqref{ind_spherical} contains $V_{(0,\dots,0|2\Theta+t,2\Phi-t)}$ with multiplicity one for all $t \geq 0$. By the calculations in Section \ref{subsec-un1}, the intertwining operators $s_{\beta,i}$ are all positive, so one only needs to consider the  $GL(1,\mathbb{C}) \times GL(1,\mathbb{C})$ intertwining operators $s_{\alpha,i}$. The composition of the intertwining operators $s_{\alpha,2} \circ s_{\alpha,1}$ on $V_{(0,\dots,0|2\Theta+t,2\Phi-t)}$  is equal to the scalar
\[
\prod_{i=1}^t\big(\frac{\nu_{\Theta}+\nu_{\Phi}-(\Theta-\Phi+i)}{\nu_{\Theta}+\nu_{\Phi}+(\Theta-\Phi+i)}\big)\cdot \prod_{i=1}^t\big(\frac{|\nu_{\Theta}-\nu_{\Phi}|-(\Theta-\Phi+i)}{|\nu_{\Theta}-\nu_{\Phi}|+(\Theta-\Phi+i)}\big).
\]
For the points $(\nu_{\Theta},\nu_{\Phi})$ inside the fundamental rectangle not listed in Theorem \ref{thm-ci}(1), there must be some $t \in \mathbb{N}_+$ such that
\[\begin{cases}
 0\leq |\nu_{\Theta}-\nu_{\Phi}|<(\Theta-\Phi)+t<\nu_{\Theta}+\nu_{\Phi}, &  t=1, \\
 (\Theta-\Phi)+t-1<|\nu_{\Theta}-\nu_{\Phi}|<(\Theta-\Phi)+t<\nu_{\Theta}+\nu_{\Phi}, & t\geq 2.\end{cases}\]
Then the signature of $V_{(0,\dots,0|2\Theta+t,2\Phi-t)}$ for this $t$ is negative (cf. \cite[p. 157 (1)]{KS83}). 

\bigskip
\noindent (b) As for the points outside the rectangle, one can again apply the results in Section \ref{subsec-un1} and induction in stages to study the signatures of the following $K$-types in $J(\Theta,\Phi;\nu_{\Theta},\nu_{\Phi})$:

\begin{itemize}
\item The level $\mathfrak{p}^+$ (resp. $\mathfrak{p}^-$) $K$-types $V_{(1,0,\cdots,0\ |\ \Theta-1,\Phi)}$ (resp. $V_{(0,\cdots,0,-1\ |\ \Theta,\Phi+1)}$), which has multiplicity one in the module \eqref{ind_spherical}: such $K$-types exists when $\Theta>\Phi$, and
the intertwining operator $s_{\alpha}$'s are positive, and the determinant of $s_{\beta,2}\circ s_{\beta,1}$ is 
\begin{equation} \label{eq-ppm}
-\frac{2\nu_{\Theta}+2\Theta-(n-1)}{2\nu_{\Theta}-2\Theta+(n-1)}\quad (\textrm{resp.}\ -\frac{2\nu_{\Phi}-2\Phi-(n-1)}{2\nu_{\Phi}+2\Phi+(n-1)}).  
\end{equation}

\item The level $\wedge^2\mathfrak{p}^+$ (resp. $\wedge^2\mathfrak{p}^-$) $K$-types $V_{(1,1,0,\cdots,0\ |\ \Theta-1,\Phi-1)}$ (resp. $V_{(0,\cdots,0,-1,-1\ |\ \Theta+1,\Phi+1)}$), which has multiplicity one in the module \eqref{ind_spherical}: the intertwining operator $s_{\alpha}$'s are positive, and the determinant of $s_{\beta,2}\circ s_{\beta,1}$ is 
\begin{equation} \label{eq-wedge2ppm}
\begin{aligned}
\frac{2\nu_{\Theta}+2\Theta-(n-1)}{2\nu_{\Theta}-2\Theta+(n-1)}\cdot \frac{2\nu_{\Phi}+2\Phi-(n-1)}{2\nu_{\Phi}-2\Phi+(n-1)}\quad 
(\textrm{resp.}\ \frac{2\nu_{\Theta}-2\Theta-(n-1)}{2\nu_{\Theta}+2\Theta+(n-1)}\cdot \frac{2\nu_{\Phi}-2\Phi-(n-1)}{2\nu_{\Phi}+2\Phi+(n-1)}).
\end{aligned}
\end{equation}
\item The level $\mathfrak{k}^-$ $K$-type $V_{(0,\cdots,0\ |\ \Theta+1,\Phi-1)}$, which has multiplicity one in the module \eqref{ind_spherical}: the intertwining operator $s_{\beta}$'s are positive, and the determinant of $s_{\alpha,2}\circ s_{\alpha,1}$ is 
\begin{equation} \label{eq-k-}
\frac{\nu_{\Theta}+\nu_{\Phi}-(\Theta-\Phi+1)}{\nu_{\Theta}+\nu_{\Phi}+(\Theta-\Phi+1)}\cdot \frac{|\nu_{\Theta}-\nu_{\Phi}|-(\Theta-\Phi+1)}{|\nu_{\Theta}-\nu_{\Phi}|+(\Theta-\Phi+1)}.
\end{equation}
\end{itemize}

\bigskip
We now study the unitarity of $J(\Theta,\Phi; \nu_{\Theta}, \nu_{\Phi})$ outside the fundamental rectangle: 

\noindent \fbox{\bf Sub-case I: $\Theta=\Phi$.} In this case, \eqref{eq-wedge2ppm} implies that the points in \[\left\{(\nu_{\Theta},\nu_{\Phi}) \Big\vert \begin{array}{l}\nu_{\Theta}>\frac{n-1}{2}-|\Theta|,\nu_{\Phi}<\frac{n-1}{2}-|\Phi|;\ \text{or}\\ \nu_{\Theta}<\frac{n-1}{2}-|\Theta|, \nu_{\Phi}>\frac{n-1}{2}-|\Phi|.\end{array} \right\}\] 
(above and on the right of the fundamental rectangle) are non-unitary at $\wedge^2\mathfrak{p}^{\pm}$. 
As for the area on the top right corner of the fundamental rectangle,  \eqref{eq-k-} implies that the points in 
\[\left\{(\nu_{\Theta},\nu_{\Phi}) \Big\vert \begin{array}{l}\frac{n-1}{2}-|\Theta|\leq \nu_{\Theta}\leq \frac{n-1}{2}-|\Theta|+1, \\
\frac{n-1}{2}-|\Phi|\leq \nu_{\Phi}\leq \frac{n-1}{2}-|\Phi|+1\end{array}\right\}\] 
are non-unitary up to level $\mathfrak{k}$ except for the points  
$(\frac{n+1}{2}-|\Theta|,\frac{n-1}{2}-|\Phi|)$ and $(\frac{n-1}{2}-|\Theta|,\frac{n+1}{2}-|\Phi|)$.
For the other points on the top right corner, they satisfy the hypothesis of Theorem \ref{thm-fund}, and hence they are non-unitary at level $\mathfrak{p}^+$ $K$-type $V_{(1,0,\dots,0|\Theta,\Phi-1)}$ (for $\Theta = \Phi \geq 0$) or $\mathfrak{p}^-$ $K$-type $V_{(0,\dots,0,-1|\Theta+1,\Phi)}$ (for $\Theta = \Phi \leq 0$).

\medskip
\noindent \fbox{\bf Sub-case II: $\Theta> \Phi$.} There are three possibilities: 
\begin{itemize}
    \item If $\Theta\geq 0\geq \Phi$, then \eqref{eq-ppm} implies that all the points outside the fundamental rectangle correspond to modules which are non-unitary up to level $\mathfrak{p}$.  
    \item If $\Theta>\Phi>0$, \eqref{eq-ppm} and \eqref{eq-wedge2ppm} would imply that all the points outside the fundamental rectangle are non-unitary up to level $\mathfrak{p}$ or $\wedge^2\mathfrak{p}$ except the ray 
    \[\left\{(\nu_{\Theta},\nu_{\Phi})\ |\ \nu_{\Theta}=\frac{n-1}{2}-|\Theta|, \nu_{\Phi}>\frac{n-1}{2}-|\Phi|\right\}.\] 
    Over the ray, \eqref{eq-k-} implies that the points with $\frac{n-1}{2}-|\Phi|<\nu_{\Phi}<\frac{n-1}{2}-|\Phi|+1$ are non-unitary up to level $\mathfrak{k}$, and Theorem \ref{thm-fund} implies that the points with $\nu_{\Phi}>\frac{n-1}{2}-|\Phi|+1$  are non-unitary  up to level $\mathfrak{p}$. Therefore, we are only left with one point $(\frac{n-1}{2}-|\Theta|,\frac{n+1}{2}-|\Phi|)$ outside the fundamental rectangle.
    \item If $0>\Theta>\Phi$, the argument is similar to that of $\Theta>\Phi>0$. There is only one point $(\frac{n+1}{2}-|\Theta|,\frac{n-1}{2}-|\Phi|)$ left outside the fundamental rectangle.
\end{itemize}

For these isolated points whose (non)-unitarity are not determined above, we will see in Section \ref{sec-coh} below that they are indeed all $A_{\mathfrak{q}}(\lambda)$-modules in the weakly fair range, and hence they must be unitary, and the result follows. 
\end{proof}

\subsection{Unitary dual for Case (ii)} \label{subsec-ii}
In Cases (b)(ii) and (c)(ii), there is a either a $(1,1)$ or $(2,2)$-parallelogram in the $\theta$-stable datum. When there is a $(1,1)$-parallelogram block in Case (b)(ii), we assume it is positioned at the $\star$ position without loss of generality. 
We will separate our study into three sub-cases, where the first two are in Case (b)(ii), and the last one is exactly Case (c)(ii).

\smallskip
\noindent \fbox{\bf Sub-case III:\begin{tikzpicture}
\foreach \x in {0,1,2.5,3.5,5,6}
	\draw (\x+0,0)--
 (\x+0.5,0)--(\x+0.3,-0.7)--
 (\x+0.2,-0.7)--
 cycle; 
\node at (0.75,-0.35) {$\cdots$};\node at (3.25,-0.35) {$\cdots$}; \node at (5.75,-0.35) {$\cdots$}; 
	\node at (2.0,-0.35) {$\ast$}; 
 
 \node at (4.55,-0.35) {$\star $}; 
 \draw (4.5,0) -- (4.9,0) -- (4.6,-0.7) -- (4.2,-0.7) -- cycle;
 \node at (0.7,-1.0) {$\underbrace{}_{k}$};
 \node at (3.2,-1.0) {$\underbrace{}_{l}$}; 
 \node at (5.7,-1.0) {$\underbrace{}_{m}$};
\end{tikzpicture}.} 
This is a special case of Case (b)(ii). We begin by partitioning the $\theta$-stable datum into the first $(k+1+l)$ blocks and the last $(m+1)$-blocks, and let $\pi_1 \boxtimes \pi_2$ be the $(\mathfrak{l}',L'\cap K)$-module corresponding to the (shifted) datum, where for $L' = L_1' \times L_2' = U(k+l+1,1) \times U(m+1,1)$ (if $\ast$ is a $(1,1)$-block) or $U(k+l+2,1) \times U(m+1,1)$ (if $\ast$ is a $(2,1)$-block).

We claim that the $(\mathfrak{g},K)$-module is unitary if and only if 
$\pi_1 \boxtimes \pi_2$ is unitary. In particular, $\pi_2$ must be a limit of discrete series representation by the first paragraph of Section \ref{subsec-un1fund}. 
Indeed, by the discussions in 
Section \ref{sec-un1}, $\pi_i$ is either unitary, or it is non-unitary up to level $\mathfrak{l}_i' \cap \mathfrak{p}$. In the latter case, the shapes of the blocks imply that (by \cite[Proposition 4.12]{W22}) the indefinite $(L' \cap K)$-types up to level $\mathfrak{l}' \cap \mathfrak{p}$ in $\pi_1 \boxtimes \pi_2$ is bottom-layer, and hence the $(\mathfrak{g},K)$-module corresponding to the full datum is also non-unitary up to level $\mathfrak{p}$. 

By applying the results of the unitary dual of $U(n,1)$ in Section \ref{sec-un1}, one concludes that the modules corresponding to this sub-case are unitary if and only if the $\nu$-parameter of the $\star$ block is $0$, and the $\nu$-parameter of the $\ast$ block must lie within $0 \leq x \leq \min\{k,l\}+\frac{a}{2}$ if $\ast$ is an $(a,1)$-block for $a = 1$ or $2$.

\medskip
\noindent \fbox{\bf Sub-case IV: 
\begin{tikzpicture}
\foreach \x in {0,1,2.5,3.5,5,6}
	\draw (\x+0,0)--
 (\x+0.5,0)--(\x+0.3,-0.7)--
 (\x+0.2,-0.7)--
 cycle; 
\node at (0.75,-0.35) {$\cdots$};\node at (3.25,-0.35) {$\cdots$}; \node at (5.75,-0.35) {$\cdots$}; 
	\node at (2.0,-0.35) {$\ast$}; 
 
 \node at (4.55,-0.35) {$\star $}; 
 \draw (4.2,0) -- (4.6,0) -- (4.9,-0.7) -- (4.5,-0.7) -- cycle;
 \node at (0.7,-1.0) {$\underbrace{}_{k}$};
 \node at (3.2,-1.0) {$\underbrace{}_{l}$}; 
 \node at (5.7,-1.0) {$\underbrace{}_{m}$};
\end{tikzpicture}.}
Similar to the previous sub-case, the $\nu$-parameter of the $\star$ block must be zero, and the $\nu$-parameter of the $\ast$ block lie within $0 \leq \nu \leq \min\{k,l+1\}+\frac{a}{2}$ if $\ast$ is an $(a,1)$-block for $a \in \{1,2\}$, otherwise the $(\mathfrak{l}_1', L_1' \cap K)$-module 
corresponding to the first $(k+l+2)$-blocks of the sub-datum
\begin{equation} \label{eq-bii2}
\begin{tikzpicture}
\foreach \x in {0,1,2.5,3.5}
	\draw (\x+0,0)--
 (\x+0.5,0)--(\x+0.3,-0.7)--
 (\x+0.2,-0.7)--
 cycle; 
\node at (0.75,-0.35) {$\cdots$};\node at (3.25,-0.35) {$\cdots$}; 
	\node at (2.0,-0.35) {$\ast$}; 
 
 \node at (4.55,-0.35) {$\star $}; 
 \draw (4.2,0) -- (4.6,0) -- (4.9,-0.7) -- (4.5,-0.7) -- cycle;
 \node at (0.7,-1.0) {$\underbrace{}_{k}$};
 \node at (3.2,-1.0) {$\underbrace{}_{l}$};

  \draw[arrows = {-Stealth[]}]          (2,0)   to [out=90,in=0]node[above]{$\nu$} (-0.3,0.3);
\draw[arrows = {-Stealth[]}]          (2,0)   to [out=90,in=180]node[above]{$-\nu$} (4.3,0.3);
\end{tikzpicture}\end{equation}
is non-unitary up to level $\mathfrak{l}_1' \cap \mathfrak{p}$, and bottom-layer arguments imply that the full $(\mathfrak{g},K)$-module is also non-unitary up to level $\mathfrak{p}$.

\smallskip
To complete the proof, let $P_{a}(k,l;\nu)$ be the irreducible $(\mathfrak{l}_1', L_1'\cap K)$-module corresponding to \eqref{eq-bii2}. We {\it claim} that for $0 \leq \nu < \min\{k,l+1\}+\frac{a}{2}$, $P_a(k,l;\nu)$ is isomorphic to the parabolically induced module:
\[P_{a}(k,l;\nu) = \mathrm{Ind}_{M_{L_1'}A_{L_1'}N_{L_1'}}^{L_1'}\left( J_1(2\gamma,2\nu) \boxtimes \Sigma' \boxtimes 1\right),\]
where 
\begin{itemize}
\item $M_{L_1'}A_{L_1'} = GL(1,\mathbb{C}) \times U(k+l+a,1)$, 
\item $\gamma$ is the content of the $\ast$-block, and 
\item $\Sigma'$ is the limit of discrete representation of $U(k+l+a,1)$ corresponding to the $\theta$-stable datum by removing two $\gamma$'s from \eqref{eq-bii2} (cf. Step (ii') in Section \ref{sec-tempered}).
\end{itemize}
Assuming that the claim holds for now, then by continuity arguments, one can also show that $P_a(k,l;\nu)$ is still unitary at $\nu = \min\{k,l+1\}+\frac{a}{2}$. In the general case when there are $m$ $(1,0)$-blocks on the right of \eqref{eq-bii2}, these blocks corresponds to a unitary representation $\pi_2$ of $U(m,0)$, and by Lemma \ref{lem-good}, $P_a(k,l;\nu) \boxtimes \pi_2$ is a unitary $(\mathfrak{l}', L' \cap K)$-module in the weakly good range. Therefore its unitarity is preserved under cohomological induction.

\medskip
\noindent \underline{\it Proof of claim:} Suppose $k < l+1$. Then one can separate the $\theta$-stable datum \eqref{eq-bii2} corresponding to $P_a(k,l;\nu)$ into $\mathcal{D}_1 \sqcup \mathcal{D}_2$, where $\mathcal{D}_1$ consists of the left $(2k+1)$ $\gamma$-blocks with $\ast$ in the middle, and $\mathcal{D}_2$ are the remaining blocks. For $0 \leq \nu < k+\frac{a}{2}$, the irreducible module $\phi_1$ corresponding to $\mathcal{D}_1$ is equal to
the standard module by looking at the kernel of the intertwining operator in Section \ref{subsec-un1}, 
and the irreducible module $\phi_2$ corresponding to $\mathcal{D}_2$ is automatically a standard module. 
So the standard module corresponding to $P_a(k,l;\nu)$ is equal to $R_{\mathfrak{q}}(\phi_1 \boxtimes \phi_2)$ by (cohomological) induction in stages. Moreover, since the latter is cohomologically induced in the good range (cf. Lemma \ref{lem-good}), it is automatically irreducible, and hence $P_a(k,l;\nu)$ is equal to the standard module.

\smallskip
In the second case when $k \geq l+1$, one can apply good range arguments again to reduce the study to $k = l+1$. As before, one can conclude that $P_{a}(l+1,l;\nu)$ is irreducible for $\nu < l+\frac{a}{2}$. 

At $\nu = l+\frac{a}{2}$, let $\Psi_1$ be a standard module of $U(2l+a+1,1)$ constructed in Section \ref{sec-irrep}, and $\Psi_2$ be a(n irreducible) limit of discrete series module of $U(1,1)$ corresponding to the following sub-data of \eqref{eq-bii2}: 
\begin{equation} \Psi_1 \quad \longleftrightarrow \quad \label{eq-blocksl}
    \begin{tikzpicture}
\foreach \x in {0,1,2.5,3.5}
	\draw (\x+0,0)--
 (\x+0.5,0)--(\x+0.3,-0.7)--
 (\x+0.2,-0.7)--
 cycle; 
\node at (0.75,-0.35) {$\cdots$};\node at (3.25,-0.35) {$\cdots$}; 
	\node at (2.0,-0.35) {$\ast$}; 
 
 \node at (0.7,-1.0) {$\underbrace{}_{l+1}$};
 \node at (3.2,-1.0) {$\underbrace{}_{l}$};

  \draw[arrows = {-Stealth[]}]          (2,0)   to [out=90,in=0]node[above]{\scriptsize $l+\frac{a}{2}$} (-0,0.3);
\draw[arrows = {-Stealth[]}]          (2,0)   to [out=90,in=180]node[above]{\scriptsize $-(l+\frac{a}{2})$} (4.3,0.3);
\end{tikzpicture}
\quad \quad \Psi_2 \longleftrightarrow \begin{tikzpicture}
     \node at (4.55,-0.35) {$\star $}; 
 \draw (4.2,0) -- (4.6,0) -- (4.9,-0.7) -- (4.5,-0.7) -- cycle;
\end{tikzpicture}
\end{equation} 
By the results in Section \ref{sec-irrep} and cohomological induction in stages, one has
$\mathcal{R}_{\mathfrak{q}}(\Psi_1 \boxtimes \Psi_2) = \mathrm{Ind}_{M_{L_1'}A_{L_1'}N_{L_1'}}^{L_1'}\left( J_1(2\gamma,2\nu) \boxtimes \Sigma' \boxtimes 1\right)$. Hence it suffices to show that $$\mathcal{R}_{\mathfrak{q}}(\Psi_1 \boxtimes \Psi_2) = P_{a}(l+1,l;l+\frac{a}{2})$$ is irreducible.

\smallskip
Now study the structure of the standard module $\Psi_1$: Indeed, by looking at the kernel of intertwining operator in Section \ref{subsec-un1}, $\Psi_1 = \psi_1 \oplus \psi_1'$ has two composition factors given by:
\begin{itemize}
\item The lowest $K$-type subquotient $\psi_1$ of $\Psi_1$, where $\psi_1|_{U(2l+a+1,1) \cap K} = \bigoplus_{k\geq 0} V^{U(2l+a+1,1)}_{(k,0,\dots,0|-k-1)}$ is a lowest weight representation (up to a power of $(\frac{\det}{|\det|})^{\frac{1}{2}} = V_{(\frac{1}{2},\dots, \frac{1}{2}|\frac{1}{2})}$); and
\item $\psi_1'$, where $\psi_1'|_{U(2l+a+1,1) \cap K} = \bigoplus_{b,c \geq 0} V^{U(2l+a+1,1)}_{(b,0\dots,-1-c|-b+c)}$ is the kernel of the intertwining operator (up to the same power of $(\frac{\det}{|\det|})^{\frac{1}{2}}$ as above). More explicitly, $\psi_1'$ corresponds to the datum:
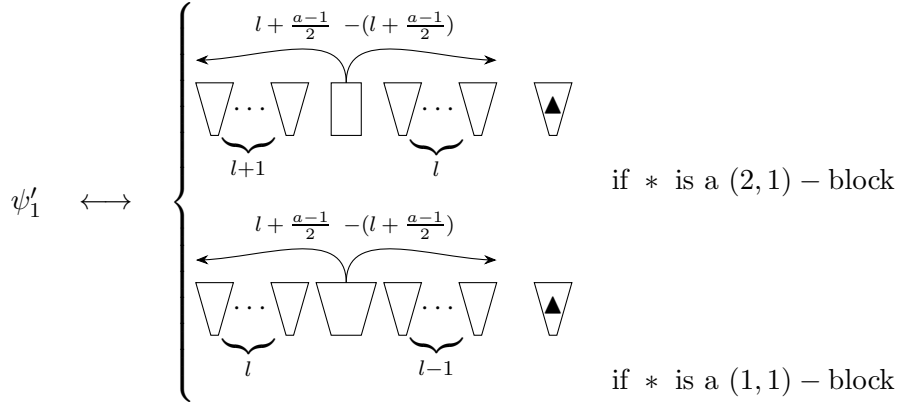
\begin{center}
$\psi_1' \quad \longleftrightarrow \quad \begin{cases}
        \begin{tikzpicture}
\foreach \x in {0,1,2.5,3.5, 4.5}
	\draw (\x+0,0)--
 (\x+0.5,0)--(\x+0.3,-0.7)--
 (\x+0.2,-0.7)--
 cycle; 
\node at (0.75,-0.35) {$\cdots$};\node at (3.25,-0.35) {$\cdots$}; 
	\draw (1.8,0) -- (2.2,0) -- (2.2,-0.7) -- (1.8,-0.7) -- cycle;

\node at (4.75,-0.3) {$\blacktriangle$};
 
 \node at (0.7,-1.0) {$\underbrace{}_{l+1}$};
 \node at (3.2,-1.0) {$\underbrace{}_{l}$};

  \draw[arrows = {-Stealth[]}]          (2,0)   to [out=90,in=0]node[above]{\scriptsize $l+\frac{a-1}{2}$} (-0,0.3);
\draw[arrows = {-Stealth[]}]          (2,0)   to [out=90,in=180]node[above]{\scriptsize $-(l+\frac{a-1}{2})$} (4,0.3);
\end{tikzpicture} &\text{
if}\ \ast\ \text{is a}\ (2,1)-\text{block}\\
        \begin{tikzpicture}
\foreach \x in {0,1,2.5,3.5, 4.5}
	\draw (\x+0,0)--
 (\x+0.5,0)--(\x+0.3,-0.7)--
 (\x+0.2,-0.7)--
 cycle; 
\node at (0.75,-0.35) {$\cdots$};\node at (3.25,-0.35) {$\cdots$}; 
	\draw (1.6,0) -- (2.4,0) -- (2.2,-0.7) -- (1.8,-0.7) -- cycle;
 
 \node at (0.7,-1.0) {$\underbrace{}_{l}$};
 \node at (3.2,-1.0) {$\underbrace{}_{l-1}$};

  \draw[arrows = {-Stealth[]}]          (2,0)   to [out=90,in=0]node[above]{\scriptsize $l+\frac{a-1}{2}$} (-0,0.3);
\draw[arrows = {-Stealth[]}]          (2,0)   to [out=90,in=180]node[above]{\scriptsize $-(l+\frac{a-1}{2})$} (4,0.3);

\node at (4.75,-0.3) {$\blacktriangle$};
\end{tikzpicture}
&\text{
if}\ \ast\ \text{is a}\ (1,1)-\text{block} \end{cases}$
\end{center}
where in both cases, the isolated $(1,0)$-block $\blacktriangle$ has content equal to that of $\star$ in \eqref{eq-bii2}. 
\end{itemize}
The structure of $\Psi_1 = \psi \oplus \psi_1'$ can be proved by induction on the number of $\gamma$-entries appearing in \eqref{eq-blocksl}, by noting that the sub-datum of $\psi_1'$ above without the $\blacktriangle$-block is equal to `$\Psi_1$' in \eqref{eq-blocksl} with one less $\gamma$-entry. 

As a consequence, one has:
\begin{equation} \label{eq-pi1'}
\mathcal{R}_{\mathfrak{q}}(\Psi_1 \boxtimes \Psi_2) = \mathcal{R}_{\mathfrak{q}}(\psi_1 \boxtimes \Psi_2) \oplus \mathcal{R}_{\mathfrak{q}}(\psi_1' \boxtimes \Psi_2),\end{equation}
where both modules on the right of \eqref{eq-pi1'} are cohomologically induced in the weakly good range, and the first term must be nonzero (and hence irreducible) since it has the same lowest $K$-type as the left-hand side of \eqref{eq-pi1'}. 

Now consider $\mathcal{R}_{\mathfrak{q}}(\psi_1' \boxtimes \Psi_2)$. 
By induction in stages, consider the sub-datum consisting of the rightmost two blocks:	
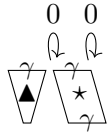
\[
\begin{tikzpicture}
    \draw (0,0)-- (0.25,0) node {\small $\gamma$} --
 (0.5,0)--(0.3,-0.7)--
 (0.2,-0.7)--
 cycle;  
 \draw (0.6,0) node (1) {}--(0.85,0) node {\small $\gamma$} --
 (1.1,0) node (2) {}--(1.3,-0.7)--(1.05,-0.7) node {\small $\gamma$} --
 (0.8,-0.7)--
 cycle;
\node at (0.25,-0.3) {$\blacktriangle$};
\node at (0.95,-0.3) {$\star$};
 \path (1) edge     [loop above]       node[] {$0$}         (1);
\path (2) edge     [loop above]       node[] {$0$}         (2);
\end{tikzpicture}\] 
The module corresponding to the above datum is cohomologically induced from a character in $\mathfrak{u}(1,0)$ and a limit of discrete series representation in $\mathfrak{u}(1,1)$ in the weakly good range, and has infinitesimal character $(\gamma, \gamma | \gamma)$. It is easy (by, for instance, directly applying the results in \cite{T01}) to conclude that the $U(2,1)$-module with the above datum and hence $\mathcal{R}_{\mathfrak{q}}(\psi_1' \boxtimes \Psi_2)$ must be zero. Hence \eqref{eq-pi1'} implies that $\mathcal{R}_{\mathfrak{q}}(\Psi_1 \boxtimes \Psi_2) = P(l+1,l;l+\frac{a}{2})$ is irreducible, and the claim holds for $\nu = l+\frac{a}{2}$. Finally, since there are no reducibility points for $P_{a}(l+1,l;\nu)$ with $l+\frac{a}{2} < \nu < l+1+\frac{a}{2}$, so the general claim follows. \qed

\bigskip
\noindent \fbox{\bf Sub-case V: \begin{tikzpicture}
\foreach \x in {0,1,2.5,3.5}
	\draw (\x+0,0)--(\x+0.5,0)--(\x+0.3,-0.7)--
 (\x+0.2,-0.7)--
 cycle; 
\node at (0.75,-0.35) {$\cdots$};\node at (3.25,-0.35) {$\cdots$};
	\node at (2.0,-0.35) {$\ast$}; 
 \node at (0.7,-1.0) {$\underbrace{}_{k}$};
 \node at (3.2,-1.0) {$\underbrace{}_{m}$}; 
\end{tikzpicture}, where $\ast$ is a $(2,2)$-parallelogram.} We begin by recalling the result of \cite{KS82}, which can also be obtained by \texttt{atlas}:
\begin{proposition} \label{prop-u22}
    Let $G = U(2,2)$. Consider the fundamental data
 \begin{tikzpicture}
\draw
    (-0.15,0) 
 -- (-0.3,0) node {\scriptsize $\frac{1}{2}$}
 -- (-0.45, 0) 
-- (-0.6,0) node {\scriptsize $\frac{1}{2}$}
-- (-0.75,0)
-- (-0.6,-0.5) 
-- (-0.45,-0.5) node {\scriptsize $\frac{1}{2}$}
-- (-0.3,-0.5)
-- (-0.15,-0.5) node {\scriptsize $\frac{1}{2}$}
-- (0,-0.5)
 -- cycle;

 \draw[arrows = {-Stealth[]}]          (-0.15,0)   to [out=90,in=00]node[above]{\scriptsize $-y,-x$} (0.15,0.3);
\draw[arrows = {-Stealth[]}]          (-0.75,0)   to [out=90,in=270]node[above]{\scriptsize $x,y$} (-0.9,0.3);
\end{tikzpicture}
or  \begin{tikzpicture}
\draw
    (0.15,0) 
 -- (0.3,0) node {\scriptsize $\frac{1}{2}$}
 -- (0.45, 0) 
-- (0.6,0) node {\scriptsize $\frac{1}{2}$}
-- (0.75,0)
-- (0.6,-0.5) 
-- (0.45,-0.5) node {\scriptsize $\frac{1}{2}$}
-- (0.3,-0.5)
-- (0.15,-0.5) node {\scriptsize $\frac{1}{2}$}
-- (0,-0.5)
 -- cycle;

 \draw[arrows = {-Stealth[]}]          (0.15,0)   to [out=90,in=00]node[above]{\scriptsize $x,y$} (-0.15,0.3);
\draw[arrows = {-Stealth[]}]          (0.75,0)   to [out=90,in=180]node[above]{\scriptsize $-y,-x$} (0.9,0.3);
\end{tikzpicture} with $x \geq y \geq 0$. Then the following holds:
\begin{enumerate}
    \item[(a)] The two data correspond to the same module if and only if $y > 0$. In such a case, the lowest $K$-types of the corresponding module are $V_{(1,1|0,0)}$ and $V_{(0,0|1,1)}$
    \item[(b)] Both data correspond to a unitary module if and only if $x+y \leq 1$. Otherwise, the module has indefinite signatures on
    \begin{itemize}
        \item \fbox{$y=0$:}\ $V_{(1,1|0,0)}$ and $V_{(1,0|1,0)}$ for  \begin{tikzpicture}
\draw
    (-0.15,0) 
 -- (-0.3,0) node {\scriptsize $\frac{1}{2}$}
 -- (-0.45, 0) 
-- (-0.6,0) node {\scriptsize $\frac{1}{2}$}
-- (-0.75,0)
-- (-0.6,-0.5) 
-- (-0.45,-0.5) node {\scriptsize $\frac{1}{2}$}
-- (-0.3,-0.5)
-- (-0.15,-0.5) node {\scriptsize $\frac{1}{2}$}
-- (0,-0.5)
 -- cycle;

 \draw[arrows = {-Stealth[]}]          (-0.15,0)   to [out=90,in=00]node[above]{\scriptsize $0,-x$} (0.15,0.3);
\draw[arrows = {-Stealth[]}]          (-0.75,0)   to [out=90,in=270]node[above]{\scriptsize $x,0$} (-0.9,0.3);
\end{tikzpicture}; or 
$V_{(0,0|1,1)}$ and $V_{(1,0|1,0)}$ for \begin{tikzpicture}
\draw
    (0.15,0) 
 -- (0.3,0) node {\scriptsize $\frac{1}{2}$}
 -- (0.45, 0) 
-- (0.6,0) node {\scriptsize $\frac{1}{2}$}
-- (0.75,0)
-- (0.6,-0.5) 
-- (0.45,-0.5) node {\scriptsize $\frac{1}{2}$}
-- (0.3,-0.5)
-- (0.15,-0.5) node {\scriptsize $\frac{1}{2}$}
-- (0,-0.5)
 -- cycle;

 \draw[arrows = {-Stealth[]}]          (0.15,0)   to [out=90,in=00]node[above]{\scriptsize $x,0$} (-0.15,0.3);
\draw[arrows = {-Stealth[]}]          (0.75,0)   to [out=90,in=180]node[above]{\scriptsize $0,-x$} (0.9,0.3);
\end{tikzpicture};

\smallskip
\item \fbox{$y>0$:} One of lowest $K$-types and a $\mathfrak{p}$-level $K$-type: $V_{(1,1|0,0)}$ and $V_{(1,0|1,0)}$;  $V_{(0,0|1,1)}$ and $V_{(1,0|1,0)}$.
    \end{itemize}
\end{enumerate}
\end{proposition}

Now we go to the study of Sub-case V in general. Since in the case when the $U(2,2)$-module in Proposition \ref{prop-u22} is non-unitary, its indefiniteness occurs at level $\mathfrak{p}^{\pm}$. One can check it is bottom-layer, and hence the full module is not unitary at the same level.

\medskip
To conclude this section, one has:
\begin{theorem} \label{thm-cii}
   For all the fundamental data in Case (ii), the corresponding module is either unitary (whose $\nu$-parameters are given by the discussions above) or it is non-unitary up to level $\mathfrak{p}$.
\end{theorem}

\subsection{Unitary dual for Case (iii)}
For Case (b)(iii), assume the $(1,1)$ or $(2,1)$ block is at $*$ position, e.g. 
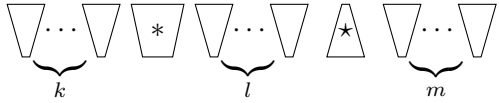
\begin{center}\begin{tikzpicture}
\foreach \x in {0,1,2.5,3.5,5,6}
	\draw (\x+0,0)--
 (\x+0.5,0)--(\x+0.3,-0.7)--
 (\x+0.2,-0.7)--
 cycle; 
\node at (0.75,-0.35) {$\cdots$};\node at (3.25,-0.35) {$\cdots$}; \node at (5.75,-0.35) {$\cdots$}; 

\draw (1.65,0)--(2.35,0)--(2.2,-0.7)--(1.8,-0.7)--cycle;
 \node at (2,-0.35) {$\ast $}; 
\draw (4.25,-0.7)--(4.75,-0.7)--(4.55,0)--(4.45,0)--cycle; 
 \node at (4.5,-0.35) {$\star $}; 
 \node at (0.7,-1.0) {$\underbrace{}_{k}$};
 \node at (3.2,-1.0) {$\underbrace{}_{l}$}; 
 \node at (5.7,-1.0) {$\underbrace{}_{m}$};
\end{tikzpicture}\end{center} 
This is similar to Sub-case III in Section \ref{subsec-ii}: Partition the $\theta$-stable datum into the first $(k+1+l)$ blocks and the last $(m+1)$-blocks, and consider the 
$(\mathfrak{l}',L'\cap K)$-module 
$\pi_1 \boxtimes \pi_2$ be the corresponding to the (shifted) datum. Then the $(\mathfrak{g},K)$-module is unitary if and only if the 
$(\mathfrak{l}',L'\cap K)$-module 
$\pi_1 \boxtimes \pi_2$ is unitary, i.e. the $\nu$-parameter of the $\ast$ block must lie within $0 \leq \nu \leq \min\{k,l\} + \frac{a}{2}$ if $\ast$ is an $(a,1)$-block. Otherwise, the module is not unitary up to level $\mathfrak{p}$.

\bigskip
Finally, for Case (c)(iii), the module is unitary if and only if the $(1,2)$-datum corresponds to a unitary module, i.e. the $\nu$-value of the $(1,2)$-block is $0 \leq \nu \leq 1$. Otherwise, the corresponding module in $U(1,2)$ is not unitary up to $\mathfrak{p}^+$ and $\mathfrak{p}^-$, and the bottom layer arguments imply that the module is also non-unitary up to $\mathfrak{p}$. 

\bigskip
To conclude this section, we have:
\begin{theorem} \label{thm-ciii}
    For all the fundamental data in Case (iii), the corresponding module is either unitary (whose $\nu$-parameters are given by the discussions above), or it is non-unitary up to level $\mathfrak{p}$.
\end{theorem}

\subsection{General case}
Now consider the general case of $U(n,2)$-modules where the $\theta$-stable datum $\mathcal{D}$ is partitioned into a disjoint union of fundamental data $\bigsqcup_{i=1}^k \mathcal{F}_i$. Then each datum must correspond to a $U(p_i,q_i)$-module with $\sum_i q_i = 2$. By carefully checking the bottom layer $K$-types, one concludes that:
\begin{corollary} \label{cor-un2}
    Let $G = U(n,2)$, and $X$ be a Hermitian, irreducible module with real infinitesimal character. Suppose its corresponding combinatorial $\theta$-stable datum $\mathcal{D} = \bigsqcup_{i=1}^k \mathcal{F}_i$ 
    is partitioned into fundamental data. Then $X$ is unitary if and only if each fundamental datum $\mathcal{F}_i$ is of the form given in Theorems \ref{thm-un1fund}, \ref{thm-ci}, \ref{thm-cii} and \ref{thm-ciii}. 
\end{corollary}
\begin{proof}
    The `only if' part can be obtained from Theorems \ref{thm-ci}, \ref{thm-cii} and \ref{thm-ciii} and a careful analysis of bottom-layer $K$-types. Special attention is given in the case of Theorem \ref{thm-ci}(a) when non-unitarity is up to $\mathrm{Sym}^t(\mathfrak{k}^-) = V_{(0,\dots,0|t,-t)}$, where one needs to check that the $t$'s are small enough to be bottom layer (see Example \ref{eg-symt} below). The `if' part is similar to the $U(n,1)$ case in Corollary \ref{cor-un1}: by looking at the possibilities of $\mathcal{F}_i$, one can invoke Lemma \ref{lem-good} to conclude that $X$ is cohomologically induced from the $\theta$-stable parabolic subalgebra defined by the partition $\mathcal{D} = \bigsqcup_{i=1}^k \mathcal{F}_i$ 
    in the weakly good range {\it with only one exception} - when $\mathcal{F}_i$ corresponds to the isolated basic module $J(\Theta,\Phi;\nu_{\Theta},\nu_{\Phi})$ in Theorem \ref{thm-ci}(b). In such a case, we will see in Section \ref{sec-coh} that $\mathcal{F}_i$ corresponds to an $A_{\mathfrak{q}}(\lambda)$-module, the other fundamental datum $\mathcal{F}_j$ can only be unitary representations of $U(p_j,0)$. By induction in stages, this implies $X$ is an $A_{\mathfrak{q}}(\lambda)$-module in the fair range, and hence it is also unitary.
\end{proof}

\begin{example} \label{eg-symt}
Let $G = U(10,3)$ and consider a $\theta$-stable data $\mathcal{D}$ whose $\lambda_a$-blocks are of the form:
\begin{center}
    \begin{tikzpicture}
\draw
    (-0.3,0) 
 -- (0.0,0) node {\small $0$\ $0$\ $0$}
 -- (0.3, 0) 
-- (0.15,-0.5) 
-- (0.0,-0.5) node {\small $0$\ $0$}
-- (-0.15,-0.5) 
 -- cycle;

\draw
    (0.7,0) 
 -- (0.825,0) node {\small $-1$}
-- (0.95,0) 
-- (0.875,-0.5) 
-- (0.825,-0.5) 
-- (0.775,-0.5) 
 -- cycle;

\draw
    (-0.7,0) 
 -- (-0.825,0) node {\small $1$}
-- (-0.95,0) 
-- (-0.875,-0.5) 
-- (-0.825,-0.5) 
-- (-0.775,-0.5) 
 -- cycle;

\draw
    (-1.125,0) 
 -- (-1.25,0) node {\small $2$}
-- (-1.375,0) 
-- (-1.3,-0.5) 
-- (-1.25,-0.5) 
-- (-1.2,-0.5) 
 -- cycle;

\draw
    (1.225,0) 
 -- (1.35,0) node {\small $-2$}
-- (1.475,0) 
-- (1.4,-0.5) 
-- (1.35,-0.5) 
-- (1.3,-0.5) 
 -- cycle;

\draw
    (1.7,0) 
 -- (1.825,0) node {\small $-3$}
-- (1.95,0) 
-- (1.875,-0.5) 
-- (1.825,-0.5) 
-- (1.775,-0.5) 
 -- cycle;

\draw
    (2.5,0) 
 -- (2.625,0) node {\small $-5$}
-- (2.75,0) 
-- (2.675,-0.5) 
-- (2.625,-0.5) 
-- (2.575,-0.5) 
 -- cycle;

\draw
    (-1.6,0) 
 -- (-1.725,0) node {\small $3$}
-- (-1.85,0) 
-- (-1.775,-0.5) 
-- (-1.725,-0.5) 
-- (-1.675,-0.5) 
 -- cycle;

\draw
    (-2.4,0) 
-- (-2.3,0) 
-- (-2.2,-0.5) 
-- (-2.35,-0.5) node {\small $5$}
-- (-2.5,-0.5) 
 -- cycle;
 \end{tikzpicture}
 \end{center}
Then $\mathcal{D} = \mathcal{F}_1 \sqcup \mathcal{F}_2 \sqcup \mathcal{F}_3$ is divided into three fundamental data with $\mathfrak{l}_0' =  \mathfrak{u}(0,1) \oplus \mathfrak{u}(9,2) \oplus \mathfrak{u}(1,0)$ satisfying the hypothesis of Remark \ref{rnk-un2}. Note that $\mathcal{F}_2$ corresponds to Case (c)(i) (or more precisely Example \ref{eg-ci}(1)) in Section \ref{subsec-i}. 

By Theorem \ref{thm-ci}, the non-unitary certificates of $\mathcal{F}_2$ are {\bf up to} level:
\begin{itemize}
\item $Sym^4(\mathfrak{k}^-) = (0, \dots, 0|4,-4)$, 
\item $\wedge^2(\mathfrak{p}^+) = (1,1,0,\dots,0|-1,-1)$, and 
\item $\wedge^2(\mathfrak{p}^-) = (0,\dots,0,-1,-1|1,1)$. 
\end{itemize}
Meanwhile, the lowest $K$-type of any fundamental representations corresponding to $\mathcal{F}_2$ has highest weight $(-1,\dots,-1,-2|9,1,1)$. 
Note that the weights
\begin{align*}
(-1,\dots,-1,-2|9,1,1)+(0,\dots,0,0|0,4,-4) &= (-1,\dots,-1,-2|9,5,-3)\\
(-1,\dots,-1,-2|9,1,1)+(1,1,0,\dots,0|0,-1,-1) &= (0,0,-1,\dots,-1,-2|9,0,0)\\
(-1,\dots,-1,-2|9,1,1)+(0,\dots,0,-1,-1,0|0,1,1) &= (-1,\dots,-1,-2,-2,-2|9,2,2)
\end{align*}
are all dominant, and hence all these $K$-types are $\mathfrak{q}'$-bottom layer. In other words, the (non)-unitarity criteria of $U(9,2)$ `extend' to $U(10,3)$, that is, the representation corresponding to $\mathcal{D}$ is unitary if and only if the fundamental representation corresponding to $\mathcal{F}_2$ is unitary.
\end{example}

Similar to Remark \ref{rmk-un1}, one has:
\begin{remark} \label{rnk-un2}
    Let $G = U(p,q)$, and $X$ be a $(\mathfrak{g},K)$-module such that its corresponding $\theta$-stable datum $\mathcal{D} = \bigsqcup_{i=1}^k \mathcal{F}_i$, where each fundamental datum $\mathcal{F}_i$ is of real rank $\leq 2$. Then $X$ is unitary only if each $\mathcal{F}_i$ is of the form given in Theorem \ref{thm-un1fund}, \ref{thm-ci}, \ref{thm-cii} and \ref{thm-ciii} (here, as in the case of Remark \ref{rmk-un1}, we are also allowing fundamental datum for $U(2,q)$). 
\end{remark}

We expect that the `if' statement in Remark \ref{rnk-un2} holds by some easy intertwining operator arguments. In any case, this leads us to the following:
\begin{conjecture}
    Let $G = U(p,q)$. Suppose the irreducible $(\mathfrak{g},K)$-module $X$ corresponds to the $\theta$-stable datum $\mathcal{D} = \bigsqcup_{i=1}^k \mathcal{F}_i$. Then $X$ is unitary if and only if each $\mathcal{F}_i$ corresponds to a unitary $(\mathfrak{l}_i',L_i' \cap K)$-module.
\end{conjecture}

\section{$A_{\mathfrak{q}}(\lambda)$-modules} \label{sec-coh}
In this section, we will determine which of the representations in Section \ref{sec-un1} and Section \ref{sec-un2} are cohomologically induced from a unitary character $Z = \mathbb{C}_{\lambda}$, i.e. $A_{\mathfrak{q}}(\lambda)$-modules. In such a case, the results in Theorem \ref{vanishing} can be generalized into:
\begin{theorem}
Let $\mathbb{C}_{\lambda}$ be a unitary $(\mathfrak{l},L\cap K)$-module. Suppose $\lambda$ is in the weakly fair range (cf. Definition \ref{good range}). Then the cohomologically induced module
$\mathcal{L}_{\mathfrak{q}}^S(\mathbb{C}_{\lambda})$ is zero for degrees $S \neq  \dim(\mathfrak{u} \cap \mathfrak{k})$. As for $S =  \dim(\mathfrak{u} \cap \mathfrak{k})$, the module
$A_{\mathfrak{q}}(\lambda) := \mathcal{L}_{\mathfrak{q}}^S(\mathbb{C}_{\lambda})$ is either unitary or zero.
\end{theorem}

In the special case when $G = U(p,q)$, the above theorem is extended to the {\bf mediocre range} \cite{T01}, where the above theorem remains valid. The choice of $\lambda$ satisfying the mediocre range condition is given explicitly \cite[Lemma 3.5(c)]{T01}. Furthermore, it is known that all $A_{\mathfrak{q}}(\lambda)$ in mediocre range are either zero or irreducible \cite[Theorem 3.1(b)]{T01}, and all $A_{\mathfrak{q}}(\lambda)$ in the mediocre range but not in the weakly fair range is isomorphic to $A_{\mathfrak{q}'}(\lambda')$ in the weakly fair range for another $\mathfrak{q}'$ and $\lambda'$ \cite[Theorem 9.1]{T01}. 

\medskip
The infinitesimal character of $A_{\mathfrak{q}}(\lambda)$ module is equal to $\Lambda = \lambda + \rho(\mathfrak{g})$. More explicitly, for $G = U(n,2)$, the coordinates of $\Lambda$ must be of the form: 
$$\begin{cases} \mathbb{Z} + \frac{1}{2} &\text{if}\ n\ \text{is even}\\ \mathbb{Z}  &\text{if}\ n\ \text{is odd}\end{cases}.$$ 
We say $\Lambda \in \mathfrak{h}^*$ is {\bf integral} if its coordinates satisfy the above condition.
The main theorem in this section is to understand the unitary dual of $U(n,2)$ with integral infinitesimal character:
\begin{theorem}[Conjecture 1.1, \cite{T01}] \label{thm-trapa}
    Let $G = U(n,2)$. Suppose $X$ is an irreducible, unitary $(\mathfrak{g},K)$-module having integral infinitesimal character. Then it is an $A_{\mathfrak{q}}(\lambda)$-module in the mediocre range.
\end{theorem}

For the rest of this section, we will determine which of the unitary modules classified in Section \ref{sec-un1} and Section \ref{sec-un2} are $A_{\mathfrak{q}}(\lambda)$ modules. There are two main reduction steps in the verification of Theorem \ref{thm-trapa}. Firstly, by Lemma \ref{lem-good}(b), if one can prove the above theorem holds for fundamental modules $X$, then the general result follows from induction in stages.
Secondly, one can further reduce to studying fundamental unitary modules that are not cohomologically induced from any $\mathfrak{q}' \supset \mathfrak{q}$ in the weakly good range, which can be easily determined by Lemma \ref{lem-good}(a). For instance, in the basic case (Section \ref{subsec-i}), all the integral points in the interior of the fundamental rectangle are cohomologically induced from a proper $\theta$-stable Levi subalgebra in the weakly good range, and they can be omitted in the proof of the theorem.

\medskip
Consequently, one only needs to study the following cases:
\begin{enumerate}

\item The only case for $U(n,1)$ is the trivial module (cf. the last paragraph of Example \ref{eg-u41}, corresponding to the data:
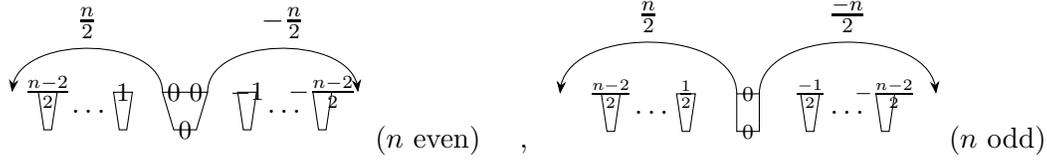
\begin{center}
\begin{tikzpicture}
\draw
    (-0.3,0) 
 -- (0.0,0) node {\small $0$\ $0$}
 -- (0.3, 0) 
-- (0.15,-0.5) 
-- (0.0,-0.5) node {\small $0$}
-- (-0.15,-0.5) 
 -- cycle;

\draw
    (0.7,0) 
 -- (0.825,0) node {\small $-1$}
-- (0.95,0) 
-- (0.875,-0.5) 
-- (0.825,-0.5) 
-- (0.775,-0.5) 
 -- cycle;

\draw
    (-0.7,0) 
 -- (-0.825,0) node {\small $1$}
-- (-0.95,0) 
-- (-0.875,-0.5) 
-- (-0.825,-0.5) 
-- (-0.775,-0.5) 
 -- cycle;

\draw (-1.25,-0.25) node {$\dots$};
\draw (1.35,-0.25) node {$\dots$};

\draw
    (1.7,0) 
 -- (1.825,0) node {\small $-\frac{n-2}{2}$}
-- (1.95,0) 
-- (1.875,-0.5) 
-- (1.825,-0.5) 
-- (1.775,-0.5) 
 -- cycle;

\draw
    (-1.7,0) 
 -- (-1.825,0) node {\small $\frac{n-2}{2}$}
-- (-1.95,0) 
-- (-1.875,-0.5) 
-- (-1.825,-0.5) 
-- (-1.775,-0.5) 
 -- cycle;

\draw[arrows = {-Stealth[]}]          (-0.3,0)   to [out=90,in=90]node[above]{$\frac{n}{2}$} (-2.3,0);
\draw[arrows = {-Stealth[]}]          (0.3,0)   to [out=90,in=90]node[above]{$-\frac{n}{2}$} (2.3,0);
 \end{tikzpicture} ($n$ even) \quad ,\quad 
 \begin{tikzpicture}
\draw
    (-0.15,0) 
 -- (0.0,0) node {\scriptsize $0$}
 -- (0.15, 0) 
-- (0.15,-0.5) 
-- (0.0,-0.5) node {\scriptsize $0$}
-- (-0.15,-0.5) 
 -- cycle;

\draw
    (0.7,0) 
 -- (0.825,0) node {\scriptsize $\frac{-1}{2}$}
-- (0.95,0) 
-- (0.875,-0.5) 
-- (0.825,-0.5) 
-- (0.775,-0.5) 
 -- cycle;

\draw
    (-0.7,0) 
 -- (-0.825,0) node {\scriptsize $\frac{1}{2}$}
-- (-0.95,0) 
-- (-0.875,-0.5) 
-- (-0.825,-0.5) 
-- (-0.775,-0.5) 
 -- cycle;

\draw (-1.25,-0.25) node {$\dots$};
\draw (1.35,-0.25) node {$\dots$};

\draw
    (1.7,0) 
 -- (1.825,0) node {\scriptsize $-\frac{n-2}{2}$}
-- (1.95,0) 
-- (1.875,-0.5) 
-- (1.825,-0.5) 
-- (1.775,-0.5) 
 -- cycle;

\draw
    (-1.7,0) 
 -- (-1.825,0) node {\scriptsize $\frac{n-2}{2}$}
-- (-1.95,0) 
-- (-1.875,-0.5) 
-- (-1.825,-0.5) 
-- (-1.775,-0.5) 
 -- cycle;

\draw[arrows = {-Stealth[]}]          (-0.15,0)   to [out=90,in=90]node[above]{$\frac{n}{2}$} (-2.5,0);
\draw[arrows = {-Stealth[]}]          (0.15,0)   to [out=90,in=90]node[above]{$\frac{-n}{2}$} (2.5,0);
 \end{tikzpicture} ($n$ odd)
\end{center}

\item For $U(n,2)$, the basic modules 
\begin{center}
$J\left(0,\Phi; \frac{n-1}{2},\frac{n-1}{2}+\Phi-k\right),$ \quad where $\Phi \leq 0$, $k > 0$ and $\frac{n-1}{2}+\Phi-k \geq 0$.
\end{center}
They are the intersection points of the `lines' in Theorem \ref{thm-cii}(a) with the boundary of the fundamental rectangle.
For instance, 
for $n=9$, $J(0,-2;4,l)$ ($l = 0,1$) corresponds to the datum:
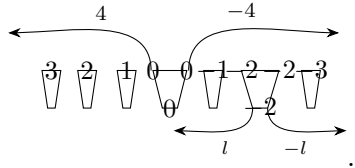
\begin{center}
\begin{tikzpicture}
\draw
    (-0.025,0) node {\small $0$} 
 -- (-0.25,0) 
 -- (-0.475, 0) node {\small $0$}
-- (-0.35,-0.5) 
-- (-0.25,-0.5) node {\small $0$}
-- (-0.15,-0.5) 
 -- cycle;

\draw
    (-0.7,0) 
 -- (-0.825,0) node {\small $1$}
-- (-0.95,0) 
-- (-0.875,-0.5) 
-- (-0.825,-0.5) 
-- (-0.775,-0.5) 
 -- cycle;

\draw
    (0.2,0) 
 -- (0.325,0) node {\small $-1$}
-- (0.45,0) 
-- (0.375,-0.5) 
-- (0.325,-0.5) 
-- (0.275,-0.5) 
 -- cycle;

\draw
    (0.7,0) node {\small $-2$}
 -- (0.95,0) 
-- (1.2,0) node {\small $-2$}
-- (1.05,-0.5) 
-- (0.95,-0.5) node {\small $-2$}
-- (0.85,-0.5) 
 -- cycle;

\draw
    (-1.225,0) 
 -- (-1.35,0) node {\small $2$}
-- (-1.475,0) 
-- (-1.4,-0.5) 
-- (-1.35,-0.5) 
-- (-1.3,-0.5) 
 -- cycle;

\draw
    (-1.7,0) 
 -- (-1.825,0) node {\small $3$}
-- (-1.95,0) 
-- (-1.875,-0.5) 
-- (-1.825,-0.5) 
-- (-1.775,-0.5) 
 -- cycle;

\draw
    (1.5,0) 
 -- (1.625,0) node {\small $-3$}
-- (1.75,0) 
-- (1.675,-0.5) 
-- (1.625,-0.5) 
-- (1.575,-0.5) 
 -- cycle;

  \draw[arrows = {-Stealth[]}]          (-0.5,0)   to [out=90,in=0]node[above]{\scriptsize $4$} (-2.4,0.5);
\draw[arrows = {-Stealth[]}]          (0,0)   to [out=90,in=180]node[above]{\scriptsize $-4$} (2,0.5);

  \draw[arrows = {-Stealth[]}]          (0.85,-0.5)   to [out=-90,in=0]node[below]{\tiny $l$} (-0.2,-0.8);
\draw[arrows = {-Stealth[]}]          (1.05,-0.5)   to [out=-90,in=180]node[below]{\tiny $-l$} (2.1,-0.8);
\end{tikzpicture}.
\end{center}
\begin{proposition} \label{prop-53}
    Let $X = J\left(0,\Phi; \frac{n-1}{2},\ell \right)$ be a basic module such that 
    \begin{itemize}
        \item $\Phi \leq 0$;
        \item $0 \leq \ell < \frac{n-1}{2} + \Phi$;
        \item the infinitesimal character of $X$ is integral. 
    \end{itemize} 
    Then $X$ is isomorphic to an $A_{\mathfrak{q}}(\lambda)$ module in the mediocre range, with $\mathfrak{l}_0 = \mathfrak{u}(n-1,1) + \mathfrak{u}(0,1) + \mathfrak{u}(1,0)$.
\end{proposition}
\begin{proof}
We will only prove the case when $\Phi \in \mathbb{Z}$ is a non-positive integer. In such a case, the infinitesimal character of $X$ is 
$$\Lambda := \left(\frac{n-1}{2}, \frac{n-3}{2}, \dots, -\frac{n-3}{2}, -\frac{n-1}{2}, \Phi + \ell, \Phi - \ell\right)$$ 
and lowest $K$-type $\kappa := V_{(0,\dots,0 | 0, 2\Phi)}$. By looking at the structure of the $\theta$-stable datum, it is easy to check that there is a unique module having lowest $K$-type $\kappa$ and infinitesimal character $\Lambda$ (this is not true in general though - see Example \ref{eg-coind} below). Therefore, one only needs to check that the $A_{\mathfrak{q}}(\lambda)$ module in the proposition has the correct lowest $K$-type and infinitesimal character.

\medskip

Now consider the $\theta$-stable parabolic subalgebra given in the proposition. One has:
\begin{align*}
    \rho(\mathfrak{g}) &= \left(\frac{n+1}{2}, \frac{n-1}{2}, \dots, -\frac{n-5}{2}, -\frac{n+1}{2}\ |\ -\frac{n-3}{2}, -\frac{n-1}{2}\right) \\
    2\rho(\mathfrak{u} \cap \mathfrak{p}) &= (1, \dots, 1, -2\ |\ 1, -n+2) \\
    S^{\bullet}(\mathfrak{u} \cap \mathfrak{p}) &= \bigoplus_{a,b,c \geq 0} V^L_{(a,0,\dots,0; -b-c\ |\ b; c-a)}
\end{align*}
where $V^L_{\mu}$ is an $(L \cap K) \cong U(n-1) \times U(1) \times U(1) \times U(1)$-module. Take 
$$\mathbb{C}_{\lambda} := V^L_{(-1,\dots,-1; \Phi - \ell + \frac{n+1}{2}\ |\ -1; \Phi + \ell + \frac{n-1}{2})},$$
so that its infinitesimal character $\lambda + \rho(\mathfrak{g})$ is equal to $\Lambda$ (up to permutation of coordinates), and
\begin{equation} \label{eq-sup}
    S^{\bullet}(\mathfrak{u} \cap \mathfrak{p}) \otimes \mathbb{C}_{\lambda + 2\rho(\mathfrak{u} \cap \mathfrak{p})} = \bigoplus_{a,b,c \geq 0} V^L_{(a,0,\dots,0; \Phi - \ell + \frac{n-3}{2}-b-c\ |\ b; \Phi + \ell - \frac{n-3}{2}+c-a)}
\end{equation}
Recall the Blattner formula for weakly fair range $A_{\mathfrak{q}}(\lambda)$ modules (cf. \cite[Equation 5.108(a)]{KV95}), which works for mediocre range modules since cohomological induction  also occur in degree $S = \dim(\mathfrak{u} \cap \mathfrak{k})$ only:
\begin{equation} \label{eq-blattner}
\begin{aligned}
    \left[A_{\mathfrak{q}}(\lambda): V_{\kappa}\right] 
    = &\sum_{i=0}^S (-1)^i \dim\left(\mathrm{Hom}_{L \cap K}\left(S^{\bullet}(\mathfrak{u} \cap \mathfrak{p}) \otimes \mathbb{C}_{\lambda + 2\rho(\mathfrak{u})}, H_i(\mathfrak{u} \cap \mathfrak{k}, V_{\kappa}) \otimes V^L_{2\rho(\mathfrak{u}\cap \mathfrak{k})}\right)\right)\\
    = &\sum_{i=0}^S (-1)^i \dim\left(\mathrm{Hom}_{L \cap K}\left(S^{\bullet}(\mathfrak{u} \cap \mathfrak{p}) \otimes \mathbb{C}_{\lambda + 2\rho(\mathfrak{u} \cap \mathfrak{p})}, H_i(\mathfrak{u} \cap \mathfrak{k}, V_{\kappa}) \right)\right)
    \end{aligned}
\end{equation}
    where the first equality comes from the Poincar\'e duality (\cite[Corollary 3.13 and Corollary 5.72]{KV95}). 
    
    \smallskip
    The right-hand side of \eqref{eq-blattner} can be computed by Kostant's theorem on Lie algebra cohomology. 
    More explicitly, when $i =0$, 
    $H_i(\mathfrak{u} \cap \mathfrak{k}, V_{\kappa}) = V^L_{(0,\dots,0; 0 | 0; 2\Phi)}$. On the other hand, one can take $a = b = 0$ and $c = \Phi - \ell + \frac{n-3}{2}$ in \eqref{eq-sup} so that it contributes a nonzero term in \eqref{eq-blattner}. Moreover, it is easy to see that when $i > 0$, the $(L \cap K)$-types appearing in $H_i(\mathfrak{u} \cap \mathfrak{k}, V_{\kappa})$ must be of the form
    $V^L_{(\xi_1 \dots, \xi_{n-1}; \zeta | \theta_1; \theta_2)}$ with some $\xi_r < 0$ or $\theta_1 < 0$, which does not appear in \eqref{eq-sup}. Therefore, one can conclude that $[A_{\mathfrak{q}}(\lambda): V_{\kappa}] =1$. 
    
    Finally, it is not hard to check that there are no other $K$-types in $A_{\mathfrak{q}}(\lambda)$ having a smaller norm than $V_{\kappa}$. As a consequence, $A_{\mathfrak{q}}(\lambda)$ has infinitesimal character $\Lambda$ and lowest $K$-type $V_{\kappa}$. By our above discussions, one concludes that $X \cong A_{\mathfrak{q}}(\lambda)$.
\end{proof}

\begin{example} \label{eg-coind}
Let $G = U(4,3)$, and $X_1$ and $X_2$ be the irreducible $(\mathfrak{g},K)$-modules corresponding to the following data:
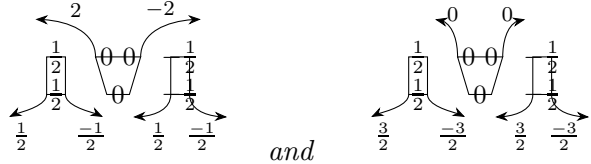
\begin{center}
\begin{tikzpicture}
\draw
    (-0.3,0) 
 -- (0.0,0) node {\small $0$\ $0$}
 -- (0.3, 0) 
-- (0.15,-0.5) 
-- (0.0,-0.5) node {\small $0$}
-- (-0.15,-0.5) 
 -- cycle;

  \draw[arrows = {-Stealth[]}]          (-0.3,0)   to [out=90,in=0]node[above]{\scriptsize $2$} (-1.1,0.5);
\draw[arrows = {-Stealth[]}]          (0.3,0)   to [out=90,in=180]node[above]{\scriptsize $-2$} (1.1,0.5);

\draw
    (0.7,0) 
 -- (0.825,0) node {\small $-\frac{1}{2}$}
-- (0.95,0) 
-- (0.95,-0.5) 
-- (0.825,-0.5)  node {\small $-\frac{1}{2}$}
-- (0.7,-0.5) 
 -- cycle;

  \draw[arrows = {-Stealth[]}]          (0.7,-0.5)   to [out=-90,in=0]node[below]{\scriptsize $\frac{1}{2}$} (0.2,-0.8);
\draw[arrows = {-Stealth[]}]          (0.95,-0.5)   to [out=-90,in=180]node[below]{\scriptsize $\frac{-1}{2}$} (1.45,-0.8);

\draw
    (-0.7,0) 
 -- (-0.825,0) node {\small $\frac{1}{2}$}
-- (-0.95,0) 
-- (-0.95,-0.5) 
-- (-0.825,-0.5) node {\small $\frac{1}{2}$}
-- (-0.7,-0.5) 
 -- cycle;

  \draw[arrows = {-Stealth[]}]          (-0.7,-0.5)   to [out=-90,in=0]node[below]{\scriptsize $\frac{-1}{2}$} (-0.2,-0.8);
\draw[arrows = {-Stealth[]}]          (-0.95,-0.5)   to [out=-90,in=180]node[below]{\scriptsize $\frac{1}{2}$} (-1.45,-0.8); 
\end{tikzpicture}
\quad and \quad
\begin{tikzpicture}
\draw
    (-0.3,0) 
 -- (0.0,0) node {\small $0$\ $0$}
 -- (0.3, 0) 
-- (0.15,-0.5) 
-- (0.0,-0.5) node {\small $0$}
-- (-0.15,-0.5) 
 -- cycle;

  \draw[arrows = {-Stealth[]}]          (-0.3,0)   to [out=90,in=0]node[above]{\scriptsize $0$} (-0.6,0.5);
\draw[arrows = {-Stealth[]}]          (0.3,0)   to [out=90,in=180]node[above]{\scriptsize $0$} (0.6,0.5);

\draw
    (0.7,0) 
 -- (0.825,0) node {\small $-\frac{1}{2}$}
-- (0.95,0) 
-- (0.95,-0.5) 
-- (0.825,-0.5)  node {\small $-\frac{1}{2}$}
-- (0.7,-0.5) 
 -- cycle;

  \draw[arrows = {-Stealth[]}]          (0.7,-0.5)   to [out=-90,in=0]node[below]{\scriptsize $\frac{3}{2}$} (0.2,-0.8);
\draw[arrows = {-Stealth[]}]          (0.95,-0.5)   to [out=-90,in=180]node[below]{\scriptsize $\frac{-3}{2}$} (1.45,-0.8);

\draw
    (-0.7,0) 
 -- (-0.825,0) node {\small $\frac{1}{2}$}
-- (-0.95,0) 
-- (-0.95,-0.5) 
-- (-0.825,-0.5) node {\small $\frac{1}{2}$}
-- (-0.7,-0.5) 
 -- cycle;

  \draw[arrows = {-Stealth[]}]          (-0.7,-0.5)   to [out=-90,in=0]node[below]{\scriptsize $\frac{-3}{2}$} (-0.2,-0.8);
\draw[arrows = {-Stealth[]}]          (-0.95,-0.5)   to [out=-90,in=180]node[below]{\scriptsize $\frac{3}{2}$} (-1.45,-0.8); 
\end{tikzpicture}
\end{center}
respectively. Then $X_1$ is not isomorphic to $X_2$, yet they have the same lowest $K$-type $V_{(0,0,0,0|1,0,-1)}$ and infinitesimal character $\Lambda = (2,1,0,0,0,-1,-2)$. 

The above modules can also be realized using \texttt{atlas}:
\begin{verbatim}
atlas> set G = U(4,3)
atlas> set x1 = parameter(G,603,[3,1,0,0,0,-1,-3],[4,1,-1,0,1,-1,-4]/2)
atlas> set x2 = parameter(G,640,[3,2,0,0,0,-2,-3],[3,3,0,0,0,-3,-3]/2)
\end{verbatim}
In particular, the last entry of the above input are precisely the $\nu$-values of the modules. Now one can check if they have the same infinitesimal character and lowest $K$-type:
\begin{verbatim}
atlas> LKT(x1) = LKT(x2)
Value: true
atlas> infinitesimal_character(x1) = infinitesimal_character(x2)
Value: true
\end{verbatim}
And finally:
\begin{verbatim}
atlas> x1 = x2
Value: false
\end{verbatim}
\end{example}

\item For $U(n,2)$,  the basic modules
\begin{center}
$J\left(\Theta,\Phi; \frac{n-1}{2}-\Theta,\frac{n-1}{2}+\Phi\right),$\quad where $\Theta \geq 0 \geq \Phi = \Theta - \frac{n-2}{2}$.
\end{center}
They are the intersection points of the edge of the first `triangle' in Theorem \ref{thm-cii}(a) with the corner of the fundamental rectangle. For instance, for $n = 7$,
$J(1,\frac{-3}{2}; 2,\frac{3}{2})$
corresponds to the datum:
\begin{center}
\begin{tikzpicture}
\draw
    (-0.025,0) node {\small $1$} 
 -- (-0.25,0) 
 -- (-0.475, 0) node {\small $1$}
-- (-0.35,-0.5) 
-- (-0.25,-0.5) node {\small $1$}
-- (-0.15,-0.5) 
 -- cycle;

\draw
    (-0.7,0) 
 -- (-0.825,0) node {\small $2$}
-- (-0.95,0) 
-- (-0.875,-0.5) 
-- (-0.825,-0.5) 
-- (-0.775,-0.5) 
 -- cycle;

\draw
    (0.2,0) 
 -- (0.325,0) node {\small $0$}
-- (0.45,0) 
-- (0.375,-0.5) 
-- (0.325,-0.5) 
-- (0.275,-0.5) 
 -- cycle;

\draw
    (0.6,0) 
 -- (0.725,0) node {\small $-1$}
-- (0.85,0) 
-- (0.775,-0.5) 
-- (0.725,-0.5) 
-- (0.675,-0.5) 
 -- cycle;

\draw
    (1.05,0) 
 -- (1.2,0) node {\small $\frac{-3}{2}$}
-- (1.35,0) 
-- (1.35,-0.5) 
-- (1.2,-0.5) node {\small $\frac{-3}{2}$}
-- (1.05,-0.5) 
 -- cycle;

\draw
    (1.5,0) 
 -- (1.625,0) node {\small $-2$}
-- (1.75,0) 
-- (1.675,-0.5) 
-- (1.625,-0.5) 
-- (1.575,-0.5) 
 -- cycle;

  \draw[arrows = {-Stealth[]}]          (-0.5,0)   to [out=90,in=0]node[above]{\scriptsize $2$} (-1.1,0.5);
\draw[arrows = {-Stealth[]}]          (0,0)   to [out=90,in=180]node[above]{\scriptsize $-2$} (0.6,0.5);

  \draw[arrows = {-Stealth[]}]          (1.05,-0.5)   to [out=-90,in=0]node[below]{\scriptsize $\frac{3}{2}$} (0.4,-0.8);
\draw[arrows = {-Stealth[]}]          (1.35,-0.5)   to [out=-90,in=180]node[below]{\scriptsize $\frac{-3}{2}$} (2,-0.8);
\end{tikzpicture}. 
\end{center}

As in Proposition \ref{prop-53}, one has:
\begin{proposition}
The unitary modules in (3) are all $A_{\mathfrak{q}}(\lambda)$-modules in the fair range, with $\mathfrak{l}_0 = \mathfrak{u}(n-2\Theta-1,1) + \mathfrak{u}(n+2\Phi-1,1)$.
\end{proposition}

\item For $U(n,2)$, the basic modules
\begin{center} $J(\Theta,\Phi; \frac{n+1}{2}-\Theta,\frac{n-1}{2}-\Phi)$,\quad where $\Theta = \frac{-1}{2} \geq \Phi$ or $\Theta = \Phi = 0$.\end{center} 
They are the isolated points in Theorem \ref{thm-cii}(b).
For instance, for $n=9$, the module $J(\frac{-1}{2},-1;\frac{9}{2},3)$ corresponds to the datum:
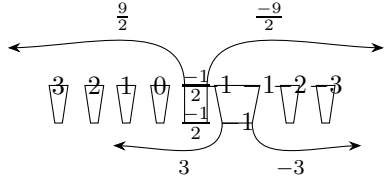
\begin{center}
\begin{tikzpicture}
\draw
    (-0.05,0) 
 -- (0.1,0) node {\small $\frac{-1}{2}$}
 -- (0.25, 0) 
-- (0.25,-0.5) 
-- (0.1,-0.5) node {\small $\frac{-1}{2}$}
-- (-0.05,-0.5) 
 -- cycle;

  \draw[arrows = {-Stealth[]}]          (-0.05,0)   to [out=90,in=0]node[above]{\scriptsize $\frac{9}{2}$} (-2.4,0.5);
\draw[arrows = {-Stealth[]}]          (0.25,0)   to [out=90,in=180]node[above]{\scriptsize $\frac{-9}{2}$} (2.6,0.5);

\draw
    (0.35,0) 
 -- (0.65,0) node {\small $-1$\ $-1$}
-- (0.95,0) 
-- (0.85,-0.5) 
 -- (0.65,-0.5) node {\small $-1$}
-- (0.45,-0.5) 
 -- cycle;

  \draw[arrows = {-Stealth[]}]          (0.45,-0.5)   to [out=-90,in=0]node[below]{\scriptsize $3$} (-1,-0.8);
\draw[arrows = {-Stealth[]}]          (0.85,-0.5)   to [out=-90,in=180]node[below]{\scriptsize $-3$} (2.3,-0.8);

\draw
    (-0.7,0) 
 -- (-0.825,0) node {\small $1$}
-- (-0.95,0) 
-- (-0.875,-0.5) 
-- (-0.825,-0.5) 
-- (-0.775,-0.5) 
 -- cycle;

\draw
    (-0.25,0) 
 -- (-0.375,0) node {\small $0$}
-- (-0.5,0) 
-- (-0.425,-0.5) 
-- (-0.375,-0.5) 
-- (-0.325,-0.5) 
 -- cycle;

\draw
    (-1.125,0) 
 -- (-1.25,0) node {\small $2$}
-- (-1.375,0) 
-- (-1.3,-0.5) 
-- (-1.25,-0.5) 
-- (-1.2,-0.5) 
 -- cycle;

\draw
    (1.225,0) 
 -- (1.35,0) node {\small $-2$}
-- (1.475,0) 
-- (1.4,-0.5) 
-- (1.35,-0.5) 
-- (1.3,-0.5) 
 -- cycle;

\draw
    (1.7,0) 
 -- (1.825,0) node {\small $-3$}
-- (1.95,0) 
-- (1.875,-0.5) 
-- (1.825,-0.5) 
-- (1.775,-0.5) 
 -- cycle;

\draw
    (-1.6,0) 
 -- (-1.725,0) node {\small $3$}
-- (-1.85,0) 
-- (-1.775,-0.5) 
-- (-1.725,-0.5) 
-- (-1.675,-0.5) 
 -- cycle;
 \end{tikzpicture}.
 \end{center}

When $\Theta = \Phi = 0$, it corresponds to the trivial module (with $\mathfrak{l}_0 = \mathfrak{u}(n,2)$). Otherwise, one has:
\begin{proposition}
The unitary modules in (4) for $\Theta = -\frac{1}{2}$ are all $A_{\mathfrak{q}}(\lambda)$-modules in the mediocre range,  with $\mathfrak{l}_0 = \mathfrak{u}(n,1) + \mathfrak{u}(0,1)$.
\end{proposition}

\item For $U(n,2)$, the modules
\[\begin{cases} P_{1}(\frac{n-1}{2},\frac{n-3}{2};\frac{n}{2})  &\mathrm{if}\ n\ \textrm{is odd}\\ P_{2}(\frac{n-2}{2},\frac{n-4}{2};\frac{n}{2}) &\mathrm{if}\ n\ \textrm{is even} \end{cases}
\]
defined in Section \ref{subsec-ii}. For instance, for $n=7$, $P_{1}(3,2; \frac{7}{2})$ corresponds to the datum:
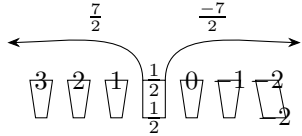
\begin{center}
    \begin{tikzpicture}
\foreach \x in {0,0.5,1,2,2.5}
	\draw (\x+0,0)--
 (\x+0.3,0)--(\x+0.225,-0.5)--
 (\x+0.075,-0.5)--
 cycle; 

\draw (1.5,0) -- (1.8,0) -- (1.8,-0.5) -- (1.5,-0.5) -- cycle;
\draw (1.65,0) node {\small $\frac{1}{2}$} ;
\draw (1.65,-0.5) node {\small $\frac{1}{2}$} ;
\draw (0.15,0) node {\small $3$} ;
\draw (0.65,0) node {\small $2$} ;
\draw (1.15,0) node {\small $1$} ;
\draw (2.15,0) node {\small $0$} ;
\draw (2.65,0) node {\small $-1$} ;
\draw (3.15,0) node {\small $-2$} ;
\draw (3.25,-0.5) node {\small $-2$} ;
 
 \draw (3,0) -- (3.3,0) -- (3.4,-0.5) -- (3.1,-0.5) -- cycle;
  \draw[arrows = {-Stealth[]}]          (1.5,0)   to [out=90,in=0]node[above]{\scriptsize $\frac{7}{2}$} (-0.3,0.5);
\draw[arrows = {-Stealth[]}]          (1.8,0)   to [out=90,in=180]node[above]{\scriptsize $\frac{-7}{2}$} (3.6,0.5);
\end{tikzpicture}

\begin{proposition}
The unitary modules in (5) are all $A_{\mathfrak{q}}(\lambda)$-modules in the fair range,  with $\mathfrak{l}_0 = \mathfrak{u}(n,1) + \mathfrak{u}(0,1)$. 
\end{proposition}

\end{center}

\item For $U(2,2)$, the modules corresponding to the parallelograms:
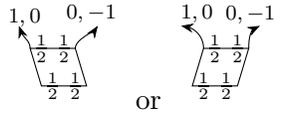
\begin{center}
\begin{tikzpicture}
\draw
    (-0.15,0) 
 -- (-0.3,0) node {\scriptsize $\frac{1}{2}$}
 -- (-0.45, 0) 
-- (-0.6,0) node {\scriptsize $\frac{1}{2}$}
-- (-0.75,0)
-- (-0.6,-0.5) 
-- (-0.45,-0.5) node {\scriptsize $\frac{1}{2}$}
-- (-0.3,-0.5)
-- (-0.15,-0.5) node {\scriptsize $\frac{1}{2}$}
-- (0,-0.5)
 -- cycle;

 \draw[arrows = {-Stealth[]}]          (-0.15,0)   to [out=90,in=00]node[above]{\scriptsize $0,-1$} (0.15,0.3);
\draw[arrows = {-Stealth[]}]          (-0.75,0)   to [out=90,in=270]node[above]{\scriptsize $1,0$} (-0.9,0.3);
\end{tikzpicture}
or  \begin{tikzpicture}
\draw
    (0.15,0) 
 -- (0.3,0) node {\scriptsize $\frac{1}{2}$}
 -- (0.45, 0) 
-- (0.6,0) node {\scriptsize $\frac{1}{2}$}
-- (0.75,0)
-- (0.6,-0.5) 
-- (0.45,-0.5) node {\scriptsize $\frac{1}{2}$}
-- (0.3,-0.5)
-- (0.15,-0.5) node {\scriptsize $\frac{1}{2}$}
-- (0,-0.5)
 -- cycle;

 \draw[arrows = {-Stealth[]}]          (0.15,0)   to [out=90,in=00]node[above]{\scriptsize $1,0$} (-0.15,0.3);
\draw[arrows = {-Stealth[]}]          (0.75,0)   to [out=90,in=180]node[above]{\scriptsize $0,-1$} (0.9,0.3);
\end{tikzpicture}.
\end{center}

\begin{proposition}
The unitary modules in (6) are $A_{\mathfrak{q}}(\lambda)$ modules in the weakly fair range, with $\mathfrak{l}_0 = \mathfrak{u}(2,1) + \mathfrak{u}(0,1)$ 
and 
$\mathfrak{l}_0 = \mathfrak{u}(1,2) + \mathfrak{u}(1,0)$ respectively.
\end{proposition}
\end{enumerate}
Now Theorem \ref{thm-trapa} follows by combining the results in (1) -- (6) above.
\begin{remark}
    We end this section by remarking that Theorem \ref{thm-trapa} holds for general $U(p,q)$ under the assumption that $X$ is a spherical module (cf. \cite[Theorem 1]{B04}). Under such assumption, the Levi factor inside $\mathfrak{l}_0 = \mathfrak{u}(p_1,q_1) \oplus \dots \oplus \mathfrak{u}(p_k,q_k)$ of largest rank ($=p_i + q_i - 1$) is determined explicitly in \cite[Section 7.8]{B04}, which is similar to the cases we studied above. One therefore expects that the techniques in \cite{B04} can be generalized to all fundamental modules. 

    Moreover, with the explicit knowledge of these modules, one can determine the Dirac series of $U(n,2)$, and check whether the necessary conditions for an $A_{\mathfrak{q}}(\lambda)$-module to be nonzero \cite[Remark 5.8]{DW21} holds or not.
\end{remark}

\section{Special Unipotent Representations} \label{sec-unip}
Motivated by the study of automorphic representations of locally symmetric spaces, Arthur introduced some families of representations for real and complex reductive groups called {\bf special unipotent representations}. These representations are well-studied in the complex groups case \cite{BV85}, and they possess certain favorable properties conjectured by Arthur. 

More explicitly, these representations are conjectured to be unitary, and it is expected that they are the main building blocks of the unitary dual of $G$. Indeed, in the case of complex classical groups, it was shown in \cite{B89} that all special unipotent representations are unitary. 
Recently, Barbasch, Ma, Sun and Zhu \cite{BMSZ21}, \cite{BMSZ22} constructed all special unipotent representations for classical real reductive groups, and proved that they are unitary. Their constructions are based on Howe's dual pair correspondence, and it is not clear how the Langlands parameters behave under the correspondence in general. 
In this section, we will determine the $\theta$-stable datum of all special unipotent representations for $G = U(n,1)$ and $U(n,2)$.

\medskip
To begin with, recall that special unipotent representations are partitioned by the complex special nilpotent orbits 
$\mathcal{O}^{\vee} \subseteq {}^L\mathfrak{g}$ in the Langlands dual of $G$. More precisely, all special unipotent representations attached to $\mathcal{O}^{\vee}$ have infinitesimal character $\frac{1}{2}h^{\vee}$ and complex associated variety equal to the closure of the $d_{LSBV}(\mathcal{O}^{\vee}) \subseteq \mathfrak{g}$, where $d_{LSBV}$ is the Lusztig-Spaltenstein-Barbasch-Vogan dual map. 

Denote $\Pi(\mathcal{O}^{\vee})$ 
be the (weak) packet of special unipotent representations attached to $\mathcal{O}^{\vee}$. In \cite{BMSZ22}, an algorithm is given to compute the cardinality of $\Pi(\mathcal{O}^{\vee})$, and the main result in \cite{BMSZ21} is to construct these representations and showed that they are unitary. 

\medskip
Now focus on $G = U(p,q)$. In such a case, the nilpotent orbits $\mathcal{O}^{\vee}$ are parametrized by partitions $(r_1 \geq r_2 \geq \dots \geq r_k)$ of $p+q$, and $d_{LSBV}$ is given by the transpose of the partition. We say $r_i \in \mathbb{N}$ is of {\bf good parity} if $r_i \equiv p+q\ (mod\ 2)$, otherwise it is of {\bf bad parity}.

Here is the main result in \cite{BMSZ21} and \cite{BMSZ22} on the structure of $\Pi(\mathcal{O}^{\vee})$:
\begin{theorem} \label{thm-bmsz}
Let $G = U(p,q)$, and $\mathcal{O}^{\vee}$ be a nilpotent orbit in ${}^L\mathfrak{g}$ corresponding to partition of the form:
\[\mathcal{O}^{\vee} = (g_1 \geq g_2 \geq \dots \geq g_k) \sqcup (b_1 \geq b_2 \geq \dots \geq b_l),\] 
where $g_i$ are of good parity and $b_i$ are of bad parity. Then 
\begin{center}
$|\Pi(\mathcal{O}^{\vee})| \neq 0$ if and only if $l = 2t$ is even, and $b_{2i-1} = b_{2i}$ for all $i$.
\end{center}
Under such assumption, let $\mathcal{O}_{good}^{\vee} := (g_1 \geq g_2 \geq \dots \geq g_k)$ be a sub-partition of $\mathcal{O}^{\vee}$ consisting only of good parity row sizes. Then
$$\Pi(\mathcal{O}^{\vee}) = \left\{ \mathrm{Ind}_{(\prod_{i=1}^{t} GL(b_{2i},\mathbb{C}) \times U(p^-,q^-))N}^{U(p,q)}\left(\boxtimes_{i=1}^t \mathrm{triv} \boxtimes X \boxtimes 1 \right)\ |\ X \in \Pi(\mathcal{O}_{good}^{\vee})\right\},$$
where $p^- := p-\sum_{i=1}^t b_{2i}$ and $q^- = q-\sum_{i=1}^t b_{2i}$. Moreover, $|\Pi(\mathcal{O}_{good}^{\vee})|$ is equal to the number of signed partitions of $\mathcal{O}_{good}^{\vee}$ with $p^-$ pluses and $q^-$ minuses.
\end{theorem}
In view of the above theorem, it suffices to describe the $\theta$-stable datum of $\Pi(\mathcal{O}^{\vee})$ whose row sizes are of good parity. Indeed, if $\pi = \mathrm{Ind}_{(GL(b) \times G')N}^{G}(\mathrm{triv} \boxtimes \pi' \boxtimes 1)$, then the $\theta$-stable datum of $\pi$ can be obtained from that of $\pi'$ by adding a $(b,b)$-block with content $0$ and $\nu$-coordinates
\[\nu = (\frac{b-1}{2}, \frac{b-1}{2}, \frac{b-3}{2}, \frac{b-3}{2}, \dots,  -\frac{b-3}{2}, -\frac{b-3}{2}, -\frac{b-1}{2}, -\frac{b-1}{2}).\]

In the following sections, we will describe the $\theta$-stable datum corresponding to $\mathcal{O}^{\vee}$ not equal to the principal nilpotent orbit (which corresponds to the trivial representation of $G$). 
\subsection{$G = U(n,1)$} 
\begin{itemize}
    \item $\mathcal{O}^{\vee} = (n-1,1,1)$. One has $|\Pi(\mathcal{O}^{\vee})| = 1$ regardless of the parity of $n$,  and the representation has infinitesimal character $\Lambda := (\frac{n-2}{2}, \dots, -\frac{n-2}{2}; 0; 0)$. This is real parabolically induced from the trivial representation of the Levi subgroup $GL(1) \times U(n-1,0)$ by Theorem \ref{thm-bmsz}.
\end{itemize}

\subsection{$G = U(n,2)$, $n$ even}
\begin{itemize}
\item $\mathcal{O}^{\vee} = (n,2)$. In such a case, $|\Pi(\mathcal{O}^{\vee})| = 3$, and the representations have infinitesimal character $\Lambda := (\frac{n-1}{2}, \dots, -\frac{n-1}{2}; \frac{1}{2}, -\frac{1}{2})$. For $n \geq 4$, the $\theta$-stable data corresponding to these three modules are:
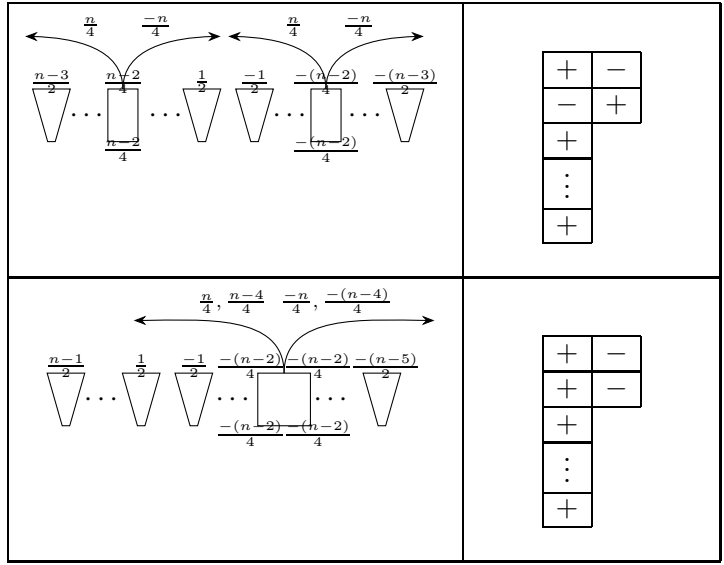
\begin{center}
\begin{longtable}{|c|c|}
\hline 
$\theta$-stable datum & Associated variety \tabularnewline
\hline 
\hline 
    \begin{tikzpicture}
\foreach \x in {-0.7,1.8,2.5,3.5}
	\draw (\x+0,0)--
 (\x+0.5,0)--(\x+0.3,-0.7)--
 (\x+0.2,-0.7)--
 cycle; 
\node at (1.5,-0.35) {$\cdots$};
\node at (0.2,-0.35) {$\cdots$};
\node at (3.25,-0.35) {$\cdots$}; 

\draw (0.5,0) -- (1.2,0) --
(1.2,-0.7) -- (0.5,-0.7) -- cycle;

\node at (0.5,0.1) {\tiny $\frac{n-2}{4}$};
\node at (1.2,0.1) {\tiny $\frac{n-2}{4}$};
\node at (0.5,-0.8) {\tiny $\frac{n-2}{4}$};
\node at (1.2,-0.8) {\tiny $\frac{n-2}{4}$};

  \draw[arrows = {-Stealth[]}]          (0.85,0)   to [out=90,in=0]node[above]{\tiny $\frac{n}{4}, \frac{n-4}{4}$} (-1.15,0.7);
\draw[arrows = {-Stealth[]}]          (0.85,0)   to [out=90,in=180]node[above]{\tiny $\frac{-n}{4}, \frac{-(n-4)}{4}$} (2.85,0.7);

\node at (3.75,0.1) {\tiny $\frac{-(n-1)}{2}$};
\node at (2.75,0.1) {\tiny $\frac{-1}{2}$};
\node at (2.05,0.1) {\tiny $\frac{1}{2}$};
\node at (-0.45,0.1) {\tiny $\frac{n-5}{2}$};

\end{tikzpicture} &  \begin{tabular}{|c|c}
\hline 
$-$ & \multicolumn{1}{c|}{+}\tabularnewline
\hline 
$-$ & \multicolumn{1}{c|}{+}\tabularnewline
\hline 
+ & \tabularnewline
\cline{1-1} 
$\vdots$ & \tabularnewline
\cline{1-1} 
+ & \tabularnewline
\cline{1-1} 
\multicolumn{1}{c}{} & \tabularnewline
\end{tabular} \tabularnewline
\hline 
    \begin{tikzpicture}
\foreach \x in {-1.2,0.8,1.5,3.5}
	\draw (\x+0,0)--
 (\x+0.5,0)--(\x+0.3,-0.7)--
 (\x+0.2,-0.7)--
 cycle; 

\node at (-0.95,0.1) {\tiny $\frac{n-3}{2}$};
\node at (1.05,0.1) {\tiny $\frac{1}{2}$};
\node at (1.75,0.1) {\tiny $\frac{-1}{2}$};
\node at (3.75,0.1) {\tiny $\frac{-(n-3)}{2}$};

\node at (-0.45,-0.35) {$\cdots$};
\node at (3.25,-0.35) {$\cdots$}; 

\draw (-0.2,0) -- (0.2,0) --
(0.2,-0.7) -- (-0.2,-0.7) -- cycle;

  \draw[arrows = {-Stealth[]}]          (0,0)   to [out=90,in=0]node[above]{\tiny $\frac{n}{4}$} (-1.3,0.7);
\draw[arrows = {-Stealth[]}]          (0,0)   to [out=90,in=180]node[above]{\tiny $\frac{-n}{4}$} (1.3,0.7);

\node at (0,0.1) {\tiny $\frac{n-2}{4}$};
\node at (2.7,0.1) {\tiny $\frac{-(n-2)}{4}$};
\node at (0,-0.8) {\tiny $\frac{n-2}{4}$};
\node at (2.7,-0.8) {\tiny $\frac{-(n-2)}{4}$};

\draw (2.5,0) -- (2.9,0) --
(2.9,-0.7) -- (2.5,-0.7) -- cycle;

  \draw[arrows = {-Stealth[]}]          (2.7,0)   to [out=90,in=0]node[above]{\tiny $\frac{n}{4}$} (1.4,0.7);
\draw[arrows = {-Stealth[]}]          (2.7,0)   to [out=90,in=180]node[above]{\tiny $\frac{-n}{4}$} (4,0.7);

\node at (0.6,-0.35) {$\cdots$};
\node at (2.25,-0.35) {$\cdots$}; 

\end{tikzpicture} &  \begin{tabular}{|c|c}
\hline 
+ & \multicolumn{1}{c|}{$-$}\tabularnewline
\hline 
$-$ & \multicolumn{1}{c|}{+}\tabularnewline
\hline 
+ & \tabularnewline
\cline{1-1} 
$\vdots$ & \tabularnewline
\cline{1-1} 
+ & \tabularnewline
\cline{1-1} 
\multicolumn{1}{c}{} & \tabularnewline
\end{tabular} \tabularnewline
\hline 
    \begin{tikzpicture}
\foreach \x in {0.7,-1.8,-2.5,-3.5}
	\draw (\x-0,0)--
 (\x-0.5,0)--(\x-0.3,-0.7)--
 (\x-0.2,-0.7)--
 cycle; 
\node at (-1.5,-0.35) {$\cdots$};
\node at (-0.2,-0.35) {$\cdots$};
\node at (-3.25,-0.35) {$\cdots$}; 

\draw (-0.5,0) -- (-1.2,0) --
(-1.2,-0.7) -- (-0.5,-0.7) -- cycle;

\node at (-0.4,0.1) {\tiny $\frac{-(n-2)}{4}$};
\node at (-1.3,0.1) {\tiny $\frac{-(n-2)}{4}$};
\node at (-0.4,-0.8) {\tiny $\frac{-(n-2)}{4}$};
\node at (-1.3,-0.8) {\tiny $\frac{-(n-2)}{4}$};

  \draw[arrows = {-Stealth[]}]          (-0.85,0)   to [out=90,in=180]node[above]{\tiny $\frac{-n}{4}, \frac{-(n-4)}{4}$} (1.15,0.7);
\draw[arrows = {-Stealth[]}]          (-0.85,0)   to [out=90,in=0]node[above]{\tiny $\frac{n}{4}, \frac{n-4}{4}$} (-2.85,0.7);

\node at (-3.75,0.1) {\tiny $\frac{n-1}{2}$};
\node at (-2.75,0.1) {\tiny $\frac{1}{2}$};
\node at (-2.05,0.1) {\tiny $\frac{-1}{2}$};
\node at (0.5,0.1) {\tiny $\frac{-(n-5)}{2}$};
\end{tikzpicture} &  \begin{tabular}{|c|c}
\hline 
+ & \multicolumn{1}{c|}{$-$}\tabularnewline
\hline 
+ & \multicolumn{1}{c|}{$-$}\tabularnewline
\hline 
+ & \tabularnewline
\cline{1-1} 
$\vdots$ & \tabularnewline
\cline{1-1} 
+ & \tabularnewline
\cline{1-1} 
\multicolumn{1}{c}{} & \tabularnewline
\end{tabular} \tabularnewline
\hline 
\end{longtable}
\end{center}
\vspace{-5mm}
More explicitly, for the first datum in the table, the results in Case (4) of Section \ref{sec-coh} imply that it is an $A_{\mathfrak{q}}(\lambda)$-module with $\mathfrak{l}_0 = \mathfrak{u}(\frac{n}{2}+2,2) \oplus \mathfrak{u}(1,0)^{\oplus n}$ in the weakly good range. For the second datum, it is a module studied in Case (3) of Section \ref{sec-coh}, where $\mathfrak{l}_0 = \mathfrak{u}(\frac{n}{2},1) + \mathfrak{u}(\frac{n}{2},1)$. Similarly, the third datum corresponds to $\mathfrak{l}_0 = \mathfrak{u}(1,0)^{\oplus n} \oplus \mathfrak{u}(\frac{n}{2}+2,2)$. 
By the algorithm given in \cite[Lemma 5.6]{T01}, one can compute the associated variety of the $A_{\mathfrak{q}}(\lambda)$ module, which is given on the right column of the above table. Since they have the correct complex associated variety $d_{LSBV}(\mathcal{O}^{\vee})$ and infinitesimal character, they are the three special unipotent representations in $\Pi(\mathcal{O}^{\vee})$.

\medskip
As for $U(2,2)$, the three $\theta$-stable data are:
\begin{tikzpicture}
\draw (0.4,0)--
 (0.8,0)--(0.65,-0.5)--
 (0.55,-0.5)--
 cycle; 

\node at (0.6,0) {\scriptsize $\frac{-1}{2}$};

\draw (-0.4,-0.5)--
 (-0.8,-0.5)--(-0.65,0)--
 (-0.55,0)--
 cycle; 

\node at (-0.6,-0.5) {\scriptsize $\frac{1}{2}$};

\draw
    (-0.15,0) 
 -- (0,0) node {\scriptsize $0$}
 -- (0.15, 0) 
-- (0.15,-0.5) 
-- (0,-0.5) node {\scriptsize $0$}
-- (-0.15,-0.5)
 -- cycle;

 \draw[arrows = {-Stealth[]}]          (-0.15,0)   to [out=90,in=0]node[above]{\scriptsize $\frac{1}{2}$} (-0.7,0.3);
\draw[arrows = {-Stealth[]}]          (0.15,0)   to [out=90,in=0]node[above]{\scriptsize \scriptsize $\frac{-1}{2}$} (0.7,0.3);
\end{tikzpicture}, 
\begin{tikzpicture}
\draw
    (-0.25,0) 
 -- (0,0) node {\scriptsize $0\ 0$}
 -- (0.25, 0) 
-- (0.25,-0.5) 
-- (0,-0.5) node {\scriptsize $0\ 0$}
-- (-0.25,-0.5)
 -- cycle;

 \draw[arrows = {-Stealth[]}]          (-0.2,0)   to [out=90,in=0]node[above]{\scriptsize $\frac{1}{2},\frac{1}{2}$} (-0.7,0.3);
\draw[arrows = {-Stealth[]}]          (0.2,0)   to [out=90,in=0]node[above]{\scriptsize \scriptsize $\frac{-1}{2},\frac{-1}{2}$} (0.7,0.3);
\end{tikzpicture}, 
\begin{tikzpicture}

\draw (0.4,-0.5)--
 (0.8,-0.5)--(0.65,0)--
 (0.55,0)--
 cycle; 

\node at (0.6,-0.5) {\scriptsize $\frac{-1}{2}$};

\draw (-0.4,0)--
 (-0.8,0)--(-0.65,-0.5)--
 (-0.55,-0.5)--
 cycle; 

\node at (-0.6,0) {\scriptsize $\frac{1}{2}$};

\draw
    (-0.15,0) 
 -- (0,0) node {\scriptsize $0$}
 -- (0.15, 0) 
-- (0.15,-0.5) 
-- (0,-0.5) node {\scriptsize $0$}
-- (-0.15,-0.5)
 -- cycle;

 \draw[arrows = {-Stealth[]}]          (-0.15,0)   to [out=90,in=0]node[above]{\scriptsize $\frac{1}{2}$} (-0.7,0.3);
\draw[arrows = {-Stealth[]}]          (0.15,0)   to [out=90,in=0]node[above]{\scriptsize \scriptsize $\frac{-1}{2}$} (0.7,0.3);
\end{tikzpicture} respectively.

\item $\mathcal{O}^{\vee} = (n-2,2,2)$. In such a case, $|\Pi(\mathcal{O}^{\vee})| = 1$ and the representation has infinitesimal character $\Lambda := (\frac{n-3}{2}, \dots, -\frac{n-3}{2}; \frac{1}{2}, -\frac{1}{2}; \frac{1}{2}, -\frac{1}{2})$.
It is real parabolically induced from the trivial module of the Levi subgroup $GL(2,\mathbb{C}) \times U(n-2,0)$.

\item $\mathcal{O}^{\vee} = (n,1,1)$. In such a case, $|\Pi(\mathcal{O}^{\vee})| = 1$ and the representation has infinitesimal character $\Lambda := (\frac{n-1}{2}, \dots, -\frac{n-1}{2}; 0; 0)$. It is real parabolically induced from the trivial module of the Levi subgroup $GL(1,\mathbb{C}) \times U(n-1,1)$. Note that its associated variety is the closure of the union of two $K$-nilpotent orbits. 
\item $\mathcal{O}^{\vee} = (n-2,1,1,1,1)$.  In such a case, $|\Pi(\mathcal{O}^{\vee})| = 1$ and the representation has infinitesimal character $\Lambda := (\frac{n-3}{2}, \dots, -\frac{n-3}{2}; 0; 0; 0; 0)$. It is real parabolically induced from the trivial module of the Levi subgroup $GL(1,\mathbb{C}) \times GL(1,\mathbb{C}) \times U(n-2,0)$.
\end{itemize}

\subsection{$G = U(n,2)$, $n$ odd}
\begin{itemize}
\item $\mathcal{O}^{\vee} = (n-2,2,2)$. In such a case, $|\Pi(\mathcal{O}^{\vee})| = 1$  and the representation has infinitesimal character $\Lambda := (\frac{n-3}{2}, \dots, -\frac{n-3}{2}; \frac{1}{2}, -\frac{1}{2}; \frac{1}{2}, -\frac{1}{2})$. It is real parabolically induced from the trivial module of the Levi subgroup $GL(2,\mathbb{C}) \times U(n-2,0)$. 
\item $\mathcal{O}^{\vee} = (n,1,1)$. In such a case, $|\Pi(\mathcal{O}^{\vee})| = 2$  and the representations have infinitesimal character $\Lambda := (\frac{n-1}{2}, \dots, -\frac{n-1}{2}; 0; 0)$.
These modules have $\theta$-stable data of the form:
\begin{center}
\begin{tabular}{|c|c|}
\hline 
$\theta$-stable datum & Associated variety \tabularnewline
\hline 
\hline 
\begin{tikzpicture}
\foreach \x in {-1.7,0.3,0.9,1.5,3.5}
	\draw (\x+0,0)--
 (\x+0.5,0)--(\x+0.3,-0.7)--
 (\x+0.2,-0.7)--
 cycle; 

\node at (-1.45,0.1) {\tiny $\frac{n-3}{2}$};
\node at (0.55,0.1) {\tiny $-1$};
\node at (1.15,0.1) {\tiny $0$};
\node at (1.75,0.1) {\tiny $1$};
\node at (3.75,0.1) {\tiny $\frac{-(n-3)}{2}$};

\node at (-0.95,-0.35) {$\cdots$};
\node at (3.25,-0.35) {$\cdots$}; 

\draw (-0.7,0) -- (-0.3,0) --
(-0.3,-0.7) -- (-0.7,-0.7) -- cycle;

  \draw[arrows = {-Stealth[]}]          (-0.5,0)   to [out=90,in=0]node[above]{\tiny $\frac{n-1}{4}$} (-2,0.7);
\draw[arrows = {-Stealth[]}]          (-0.5,0)   to [out=90,in=180]node[above]{\tiny $\frac{-(n-1)}{4}$} (1,0.7);

\node at (-0.5,0.1) {\tiny $\frac{n-1}{4}$};
\node at (2.7,0.1) {\tiny $\frac{-(n-1)}{4}$};
\node at (-0.5,-0.8) {\tiny $\frac{n-1}{4}$};
\node at (2.7,-0.8) {\tiny $\frac{-(n-1)}{4}$};

\draw (2.5,0) -- (2.9,0) --
(2.9,-0.7) -- (2.5,-0.7) -- cycle;

  \draw[arrows = {-Stealth[]}]          (2.7,0)   to [out=90,in=0]node[above]{\tiny $\frac{n-1}{4}$} (1.2,0.7);
\draw[arrows = {-Stealth[]}]          (2.7,0)   to [out=90,in=180]node[above]{\tiny $\frac{-(n-1)}{4}$} (4.2,0.7);

\node at (0.1,-0.35) {$\cdots$};
\node at (2.25,-0.35) {$\cdots$}; 

\end{tikzpicture} & \begin{tabular}{|c|cc}
\hline 
$-$ & \multicolumn{1}{c|}{+} & \multicolumn{1}{c|}{$-$}\tabularnewline
\hline 
+ &  & \tabularnewline
\cline{1-1} 
$\vdots$ &  & \tabularnewline
\cline{1-1} 
+ &  & \tabularnewline
\cline{1-1} 
+ &  & \tabularnewline
\cline{1-1} 
\multicolumn{1}{c}{} &  & \tabularnewline
\end{tabular} \tabularnewline
\hline 
\begin{tikzpicture}
\foreach \x in {0,1,2.5,3.5}
	\draw (\x+0,0)--
 (\x+0.5,0)--(\x+0.3,-0.7)--
 (\x+0.2,-0.7)--
 cycle; 
\node at (0.7,-0.35) {$\cdots$};\node at (3.25,-0.35) {$\cdots$}; 

\draw (1.65,0)--(2.35,0)--(2.2,-0.7)--(1.8,-0.7)--cycle;
  \draw[arrows = {-Stealth[]}]          (2,0)   to [out=90,in=0]node[above]{\tiny $\frac{n-1}{2}, 0$} (-0.5,0.7);
\draw[arrows = {-Stealth[]}]          (2,0)   to [out=90,in=180]node[above]{\tiny $0, \frac{-(n-1)}{2}$} (4.5,0.7);

\node at (2,0.1) {\tiny $0\ 0\ 0$};
\node at (2,-0.8) {\tiny $0\ 0$};
\node at (0.25,0.1) {\tiny $\frac{n-3}{2}$};
\node at (1.25,0.1) {\tiny $1$}; 
\node at (2.75,0.1) {\tiny $-1$};
\node at (3.75,0.1) {\tiny $\frac{-(n-3)}{2}$};
\end{tikzpicture} &  \begin{tabular}{|c|cc}
\hline 
+ & \multicolumn{1}{c|}{$-$} & \multicolumn{1}{c|}{+}\tabularnewline
\hline 
+ &  & \tabularnewline
\cline{1-1} 
$\vdots$ &  & \tabularnewline
\cline{1-1} 
+ &  & \tabularnewline
\cline{1-1} 
$-$ &  & \tabularnewline
\cline{1-1} 
\multicolumn{1}{c}{} &  & \tabularnewline
\end{tabular} \tabularnewline
\hline 
\end{tabular}
\end{center}
In particular, the first datum is an $A_{\mathfrak{q}}(\lambda)$ module in the weakly good range, with $\theta$-stable Levi $\mathfrak{l}_0 = \mathfrak{u}(\frac{n-1}{2},1) \oplus \mathfrak{u}(1,0) \oplus  \mathfrak{u}(\frac{n-1}{2},1)$; while the second datum is of Case (1) in Section \ref{sec-coh}
with $\mathfrak{l}_0 = \mathfrak{u}(n-1,1) \oplus \mathfrak{u}(0,1) \oplus \mathfrak{u}(1,0)$.

\item $\mathcal{O}^{\vee} = (n-2,1,1,1,1)$. In such a case, $|\Pi(\mathcal{O}^{\vee})| = 1$ and the representation has infinitesimal character $\Lambda := (\frac{n-3}{2}, \dots, -\frac{n-3}{2}; 0; 0; 0; 0)$. It is real parabolically induced from the trivial module of the Levi subgroup $GL(1,\mathbb{C}) \times GL(1,\mathbb{C}) \times U(n-2,0)$.
\end{itemize}

In view of the above calculations, we make the following:
\begin{conjecture}
    All special unipotent representations of $U(p,q)$ are fundamental.
\end{conjecture}

\section{Acknowledgement}
This project is supported by National Natural
Science Foundation of China (grant no. 12371033) and
Shenzhen Science and Technology Innovation Committee (grant no. 20220818094918001).

\end{document}